\pdfoutput=1
\documentclass[11pt]{article}

\usepackage[utf8]{inputenc}
\usepackage[pdftex]{graphicx,xcolor}
\usepackage{array,bookmark,fullpage}
\usepackage{microtype}
\usepackage{amsmath,amssymb,amsthm,mathtools,bm,tikz-cd,stmaryrd}
\usepackage[nottoc]{tocbibind}
\usepackage{hyperref}
\hypersetup{pdftex,colorlinks=true,allcolors=blue,unicode,psdextra}
\usepackage[capitalize]{cleveref}

\usepackage{mymacros}

\DeclareMathOperator\Mod{Mod}
\DeclareMathOperator\Iso{Iso}
\DeclareMathOperator\MOD{\!{Mod}}
\DeclareMathOperator\Act{\!{Act}}
\DeclareMathOperator\Idl{Idl}
\DeclareMathOperator\Ind{\!{Ind}}
\DeclareMathOperator\Sp{Sp}
\DeclareMathOperator\Eq{\!{Eq}}
\DeclareMathOperator\Ins{\!{Ins}}
\DeclareMathOperator\Inv{\!{Inv}}
\DeclareMathOperator\Sh{\!{Sh}}
\DeclareMathOperator\Desc{\!{Desc}}

\newcommand*\op{\mathrm{op}}

\DeclarePairedDelimiter\den{\llbracket}{\rrbracket}

\mathlig{=>}{\Rightarrow}
\mathlig{<=}{\Leftarrow}
\mathlig{<=>}{\Leftrightarrow}
\mathlig{|-}{\vdash}
\mathlig{|=}{\models}

\let\tri\triangle

\let\defn\textbf

\newcommand{\relmiddle}[1]{\nonscript\;\middle#1\nonscript\;}

\newcommand*\stag[1]{\phantomsection\tag{\textsc{\thesection.#1}}\label{seq:#1}}

\begin{document}

\title{Borel functors, interpretations, and strong conceptual completeness for $\@L_{\omega_1\omega}$}
\author{Ruiyuan Chen\thanks{Research partially supported by NSERC PGS D}}
\date{}
\maketitle

\begin{abstract}
We prove a strong conceptual completeness theorem (in the sense of Makkai) for the infinitary logic $\@L_{\omega_1\omega}$: every countable $\@L_{\omega_1\omega}$-theory can be canonically recovered from its standard Borel groupoid of countable models, up to a suitable syntactical notion of equivalence.  This implies that given two theories $(\@L, \@T)$ and $(\@L', \@T')$ (in possibly different languages $\@L, \@L'$), every Borel functor $\MOD(\@L', \@T') -> \MOD(\@L, \@T)$ between the respective groupoids of countable models is Borel naturally isomorphic to the functor induced by some $\@L'_{\omega_1\omega}$-interpretation of $\@T$ in $\@T'$.  This generalizes a recent result of Harrison-Trainor--Miller--Montalbán in the $\aleph_0$-categorical case.
\renewcommand{\thefootnote}{}
\setcounter{footnote}{0}
\end{abstract}

\section{Introduction}

A ``strong conceptual completeness'' theorem for a logic, in the sense of Makkai \cite{Mscc}, is a strengthening of the usual completeness theorem which allows the syntax of a theory in that logic to be completely recovered from its semantics, up to a suitable notion of equivalence.  In this paper, we prove such a result for the infinitary logic $\@L_{\omega_1\omega}$.

Let $\@L$ be a countable first-order language and $\@T$ be a countable $\@L_{\omega_1\omega}$-theory.  An \defn{$(\@L_{\omega_1\omega}, \@T)$-imaginary sort} $A$ is a certain kind of syntactical name for a countable set $A^\@M$ uniformly definable from each countable model $\@M$ of $\@T$, which is built up from $\@L_{\omega_1\omega}$-formulas by taking (formal) countable disjoint unions and quotients by definable equivalence relations.  Given two imaginary sorts $A, B$, an \defn{$(\@L_{\omega_1\omega}, \@T)$-definable function} $f : A -> B$ is a syntactical name for a uniformly definable function $f^\@M : A^\@M -> B^\@M$ for each model $\@M$; formally, $f$ is given by a $\@T$-equivalence class of (families of) $\@L_{\omega_1\omega}$-formulas defining the graph of such a function.  The notion of \defn{definable relation} $R \subseteq A$ on an imaginary sort $A$ is defined similarly.  See \cref{sec:imag} for the precise definitions.

Let $\MOD(\@L, \@T)$ denote the \defn{standard Borel groupoid of countable models of $\@T$}, whose space of objects $\Mod(\@L, \@T)$ is the standard Borel space of models of $\@T$ whose underlying set is an initial segment of $\#N$, and whose morphisms are isomorphisms between models.

Before stating the strong conceptual completeness theorem for $\@L_{\omega_1\omega}$, we first state some of its consequences.  Given a countable $\@L_{\omega_1\omega}$-theory $\@T$, an \defn{$\@L'_{\omega_1\omega}$-interpretation} $F : (\@L, \@T) -> (\@L', \@T')$ of $\@T$ in another countable $\@L'_{\omega_1\omega}$-theory $\@T'$ (in possibly a different language $\@L'$) consists of:
\begin{itemize}
\item  an $(\@L'_{\omega_1\omega}, \@T')$-imaginary sort\footnote{Here $\#X$ is thought of as the $\@T$-imaginary sort which names the underlying set of a model; see \cref{list:pretopos-structure}.} $F(\#X)$;
\item  for each $n$-ary relation symbol $R \in \@L$, an $n$-ary definable relation $F(R) \subseteq F(\#X)^n$;
\item  for each $n$-ary function symbol $f \in \@L$, an $n$-ary definable function $F(f) : F(\#X)^n -> F(\#X)$;
\item  such that ``applying'' $F$ to the axioms in $\@T$ results in $\@L'$-sentences implied by $\@T'$.
\end{itemize}
Given such an interpretation $F$, every countable model $\@M = (M, R^\@M, f^\@M)_{R, f \in \@L'}$ of $\@T'$ gives rise to a countable model $F^*(\@M) = (F(\#X)^\@M, F(R)^\@M, F(f)^\@M)_{R, f \in \@L}$ of $\@T$; this yields a Borel functor
\begin{align*}
F^* : \MOD(\@L', \@T') -> \MOD(\@L, \@T)
\end{align*}
(after suitable coding to make the underlying set of $F^*(\@M)$ an initial segment of $\#N$).  Conversely,

\begin{theorem}
\label{thm:2interp-eso}
Every Borel functor $\MOD(\@L', \@T') -> \MOD(\@L, \@T)$ is Borel naturally isomorphic to $F^*$ for some interpretation $F : (\@L, \@T) -> (\@L', \@T')$.
\end{theorem}

This generalizes the Borel version of the main result of Harrison-Trainor--Miller--Montalbán \cite[Theorem~9]{HMM}, which is the case where $\@T, \@T'$ are Scott sentences, i.e., they each have a single countable model up to isomorphism.

For an $(\@L_{\omega_1\omega}, \@T)$-imaginary sort $A$, let
\begin{align*}
\den{A} := \{(\@M, a) \mid \@M \in \Mod(\@L, \@T) \AND a \in A^\@M\}
\end{align*}
be the disjoint union of the countable sets $A^\@M$ defined by $A$ in all models $\@M \in \Mod(\@L, \@T)$.  There is a natural standard Borel structure on $\den{A}$, and we have the fiberwise countable Borel projection map $\pi : \den{A} -> \Mod(\@L, \@T)$.  Since the set $A^\@M$ is uniformly defined for each model $\@M$, an isomorphism between models $\@M \cong \@N$ induces a bijection $A^\@M \cong A^\@N$.  This gives a Borel action of the groupoid $\MOD(\@L, \@T)$ on $\den{A}$, turning $\den{A}$ into a \defn{fiberwise countable Borel $\MOD(\@L, \@T)$-space}, i.e., a standard Borel space $X$ equipped with a fiberwise countable Borel map $p : X -> \Mod(\@L, \@T)$ and a Borel action of $\MOD(\@L, \@T)$.  The core result of this paper is

\begin{theorem}
\label{thm:borel-interp-eso}
Every fiberwise countable Borel $\MOD(\@L, \@T)$-space is isomorphic to $\den{A}$ for some $(\@L_{\omega_1\omega}, \@T)$-imaginary sort $A$.
\end{theorem}

In other words, every Borel isomorphism-equivariant assignment of a countable set to every countable model of $\@T$ is named by some imaginary sort.

In order to place \cref{thm:2interp-eso,thm:borel-interp-eso} in their proper context, we organize the $(\@L_{\omega_1\omega}, \@T)$-imaginary sorts and definable functions into a category, the \defn{syntactic Boolean $\omega_1$-pretopos} of $\@T$, denoted
\begin{align*}
\-{\ang{\@L \mid \@T}}^B_{\omega_1}.
\end{align*}
The syntactic Boolean $\omega_1$-pretopos is the categorical ``Lindenbaum--Tarski algebra'' of an $\@L_{\omega_1\omega}$-theory: it ``remembers'' the logical structure of the theory, such as $\@T$-equivalence classes of formulas and implications between them, while ``forgetting'' irrelevant syntactic details.  See \cite{MR} or \cite[D]{Jeleph} for the general theory of syntactic pretoposes and related notions.

Let $\Act_{\omega_1}^B(\MOD(\@L, \@T))$ denote the category of fiberwise countable Borel $\MOD(\@L, \@T)$-spaces and Borel equivariant maps between them.  Given a definable function $f : A -> B$, taking the disjoint union of the functions $f^\@M : A^\@M -> B^\@M$ for every countable model $\@M \in \Mod(\@L, \@T)$ yields a Borel $\MOD(\@L, \@T)$-equivariant map $\den{f} : \den{A} -> \den{B}$.  We thus have a functor
\begin{align*}
\den{-} : \-{\ang{\@L \mid \@T}}^B_{\omega_1} --> \Act_{\omega_1}^B(\MOD(\@L, \@T)),
\end{align*}
which can be thought of as taking syntax to semantics for the $\@L_{\omega_1\omega}$-theory $\@T$.  Our main result is

\begin{theorem}
\label{thm:borel-interp-equiv}
The functor $\den{-} : \-{\ang{\@L \mid \@T}}^B_{\omega_1} --> \Act_{\omega_1}^B(\MOD(\@L, \@T))$ is an equivalence of categories.
\end{theorem}

We say that two theories $(\@L, \@T)$, $(\@L', \@T')$ (in possibly different languages) are \defn{Morita equivalent} if their syntactic Boolean $\omega_1$-pretoposes are equivalent categories; informally, this means that they are different presentations of the ``same'' theory.  Thus, by \cref{thm:borel-interp-equiv}, a theory can be recovered up to Morita equivalence from its standard Borel groupoid of countable models.

To see the connection of \cref{thm:borel-interp-equiv} with \cref{thm:borel-interp-eso}, as well as the sense in which it is a strong form of the completeness theorem, note that (by general category theory) the statement that $\den{-}$ is an equivalence may be broken into three parts:
\begin{enumerate}
\item[(i)]  $\den{-}$ is conservative, i.e., injective when restricted to the lattice of subobjects of each imaginary sort $A \in \-{\ang{\@L \mid \@T}}^B_{\omega_1}$.  This is equivalent to the Lopez-Escobar completeness theorem for $\@L_{\omega_1\omega}$ \cite{Lop}, provided that in the definitions above of imaginary sorts and definable functions, when we say e.g., that a formula $\phi$ defines the graph of a function in models of $\@T$, we actually mean that $\@T$ proves various $\@L_{\omega_1\omega}$-sentences which say ``$\phi$ is the graph of a function''.  (If we instead interpret these conditions semantically, then conservativity becomes vacuous.)
\item[(ii)]  $\den{-}$ is full on subobjects, i.e., surjective when restricted to subobject lattices.  This is equivalent to the Lopez-Escobar definability theorem for isomorphism-invariant Borel sets \cite{Lop}.
\item[(iii)]  $\den{-}$ is essentially surjective.  This is \cref{thm:borel-interp-eso}.
\end{enumerate}
We will explain the equivalences in (i) and (ii) when we prove \cref{thm:borel-interp-equiv} in \cref{sec:etale-interp-equiv,sec:borel-interp-equiv}.

\Cref{thm:borel-interp-equiv} is the essence of a strong conceptual completeness theorem for $\@L_{\omega_1\omega}$.  There is a large family of such theorems known for various kinds of logic.  Typically, these take the form of a ``Stone-type duality'' arising from a dualizing (or ``schizophrenic'') object equipped with two commuting kinds of structure; see \cite{PT} or \cite[VI~\S4]{Jstone} for the general theory of such dualities.  Here is a partial list of such dualities interpreted as strong conceptual completeness theorems:
\begin{itemize}

\item  The original Stone duality between Boolean algebras and compact Hausdorff zero-dimensional spaces arises from equipping the set $2 = \{0, 1\}$ with both kinds of structure.  When interpreted as a strong conceptual completeness theorem for (finitary) propositional logic, Stone duality says that the Lindenbaum--Tarski algebra $\ang{\@L \mid \@T}$ of a propositional theory $(\@L, \@T)$ may be recovered as the algebra of clopen sets of its space of models.

\item  A version of Łoś's theorem says that ultraproducts on the category of sets commute with the structure (finite intersection, finite union, etc.)\ used to interpret finitary first-order logic $\@L_{\omega\omega}$.  Makkai \cite{Multra} proved that for an $\@L_{\omega\omega}$-theory $\@T$, its syntactic Boolean pretopos $\-{\ang{\@L \mid \@T}}^B_\omega$ (defined similarly to $\-{\ang{\@L \mid \@T}}^B_{\omega_1}$ above but for $\@L_{\omega\omega}$) may be recovered as the category of ultraproduct-preserving actions of the category of models and elementary embeddings.

\item  Analogous results of Gabriel--Ulmer, Lawvere, Makkai and others (see \cite{ARlpac}, \cite{ALR}, \cite{Mexact}) apply to various well-behaved fragments of $\@L_{\omega\omega}$.

\item  In perhaps the closest relative to this paper, Awodey--Forssell \cite{AF} proved that for an $\@L_{\omega\omega}$-theory $\@T$, $\-{\ang{\@L \mid \@T}}^B_\omega$ may be recovered as certain continuous actions of the topological groupoid of models on subsets of a fixed set of large enough cardinality.

\end{itemize}
We will explain how to view \cref{thm:borel-interp-equiv} as a Stone-type duality theorem in \cref{sec:stone}.  One benefit of doing so is that general duality theory then automatically yields \cref{thm:2interp-eso}.  Indeed, the whole of \cref{thm:borel-interp-equiv} is equivalent to the following strengthening of \cref{thm:2interp-eso}: given two theories $(\@L, \@T), (\@L', \@T')$, the functor $F |-> F^*$ taking interpretations $F : (\@L, \@T) -> (\@L', \@T')$ and ``definable natural isomorphisms'' between them to Borel functors and Borel natural isomorphisms is an equivalence of groupoids.  (See \cref{thm:2interp-equiv} for a precise statement.)

The proof of \cref{thm:borel-interp-equiv} is by reduction to a continuous version of the result.  Recall that a \defn{countable fragment} $\@F$ of $\@L_{\omega_1\omega}$ is a countable set of $\@L_{\omega_1\omega}$-formulas containing atomic formulas and closed under subformulas, $\@L_{\omega\omega}$-logical operations, and variable substitutions.  Given a countable fragment $\@F$ containing a countable theory $\@T$, we define the notions of \defn{$(\@F, \@T)$-imaginary sort} and \defn{$(\@F, \@T)$-definable function} in the same way as the $(\@L_{\omega_1\omega}, \@T)$- versions above, except that the formulas involved must be \emph{countable disjunctions of formulas in $\@F$}.  The resulting category is called the \defn{syntactic $\omega_1$-pretopos}, denoted
\begin{align*}
\-{\ang{\@F \mid \@T}}_{\omega_1}.
\end{align*}
Let $\Mod(\@F, \@T)$ denote $\Mod(\@L, \@T)$ equipped with the Polish topology induced by the countable fragment $\@F$; see \cite[Ch.~11]{Gao}.  Let $\MOD(\@F, \@T)$ denote the \defn{Polish groupoid of countable models of $\@T$}, whose space of objects is $\Mod(\@F, \@T)$ and whose morphisms are isomorphisms with the usual pointwise convergence topology.
We say that a topological space $X$ equipped with a continuous map $p : X -> \Mod(\@F, \@T)$ is \defn{countable étalé over $\Mod(\@F, \@T)$} if $X$ has a countable cover by open sets $U \subseteq X$ such that $p|U$ is an open embedding; a \defn{countable étalé $\MOD(\@F, \@T)$-space} is such a space equipped with a continuous action of $\MOD(\@F, \@T)$.  For an $(\@F, \@T)$-imaginary sort $A$, the space $\den{A}$ is countable étalé over $\Mod(\@F, \@T)$ in a canonical way; and we get a functor
\begin{align*}
\den{-} : \-{\ang{\@F \mid \@T}}_{\omega_1} --> \Act_{\omega_1}(\MOD(\@F, \@T)),
\end{align*}
where $\Act_{\omega_1}(\MOD(\@F, \@T))$ is the category of countable étalé $\MOD(\@F, \@T)$-spaces and continuous equivariant maps.  We now have the following continuous analog of \cref{thm:borel-interp-equiv}:

\begin{theorem}
\label{thm:etale-interp-equiv}
The functor $\den{-} : \-{\ang{\@F \mid \@T}}_{\omega_1} --> \Act_{\omega_1}(\MOD(\@F, \@T))$ is an equivalence of categories.
\end{theorem}

The proof of \cref{thm:borel-interp-equiv} from \cref{thm:etale-interp-equiv} in \cref{sec:borel-interp-equiv} uses techniques from invariant descriptive set theory, in particular Vaught transforms and the Becker--Kechris method for topological realization of Borel actions (see \cite{BKpol} or \cite{Gao}), to show that every fiberwise countable Borel action can be realized as a countable étalé action by picking a large enough countable fragment $\@F$.  Along the way, we will prove the following more abstract result, which may be of independent interest:

\begin{theorem}
\label{thm:borel-action-etale}
Let $\!G$ be an open Polish groupoid and $X$ be a fiberwise countable Borel $\!G$-space.  Then there is a finer open Polish groupoid topology on $\!G$, such that letting $\!G'$ be the resulting Polish groupoid, $X$ is Borel isomorphic to a countable étalé $\!G'$-space.
\end{theorem}

As for \cref{thm:etale-interp-equiv}, we will give a direct proof in \cref{sec:etale-interp-equiv}.  The proof we give is analogous to that of the duality result of Awodey--Forssell \cite{AF} for $\@L_{\omega\omega}$-theories mentioned above.

However, as with Awodey--Forssell's result, in some sense the proper context for \cref{thm:etale-interp-equiv} is the theory of groupoid representations for toposes.  As such, we will sketch in \cref{sec:joyal-tierney} an alternative proof of (a generalization of) \cref{thm:etale-interp-equiv} using the Joyal--Tierney representation theorem \cite{JT}.  While this proof (together with its prerequisite definitions and lemmas) is admittedly much longer than the direct proof, it uses only straightforward variations of well-known concepts and arguments, thereby showing that \cref{thm:etale-interp-equiv} is in some sense a purely ``formal'' consequence of standard topos theory.

\bigskip

We have tried to organize this paper so as to minimize the amount of category theory needed in the earlier sections.  We begin with basic definitions involving groupoids and étalé spaces in \cref{sec:groupoid}, followed by the proof of \cref{thm:borel-action-etale} in \cref{sec:borel-action-etale}.   We then give in \cref{sec:imag,sec:isogpd,sec:den} the precise definitions of the syntactic (Boolean) $\omega_1$-pretopos, the groupoid of countable models, and the functor $\den{-}$.  Along the way, we introduce the notion of an ``$\omega_1$-coherent theory'', which generalizes that of a countable fragment.  In \cref{sec:lopez-escobar}, we present a version of Vaught's proof of Lopez-Escobar's (definability) theorem; this will be needed in what follows.  In \cref{sec:etale-interp-equiv}, we give the direct proof of \cref{thm:etale-interp-equiv}, which is then used (along with the proof of \cref{thm:borel-action-etale}) to prove \cref{thm:borel-interp-equiv} in \cref{sec:borel-interp-equiv}.

The rest of the paper involves more heavy-duty categorical notions.  In \cref{sec:interp}, we define the notion of an interpretation $F$ between theories and the induced functor $F^*$ between the groupoids of models; this defines a (contravariant) pseudofunctor from the $2$-category of theories to the $2$-category of standard Borel groupoids.  In \cref{sec:stone}, we explain how \cref{thm:borel-interp-equiv} may be viewed as a Stone-type duality, yielding \cref{thm:2interp-eso}; we also explain how this latter result may be viewed as a generalization of the main result of \cite{HMM}.  In \cref{sec:kloc,sec:lkpcat,sec:kcohthy}, we give some prerequisite definitions and lemmas, which are used in the proof of (a generalization of) \cref{thm:etale-interp-equiv} from the Joyal--Tierney theorem, in \cref{sec:joyal-tierney}.

\medskip
\textit{Acknowledgments.}  We would like to thank Alexander Kechris for providing some comments on a draft of this paper.

\section{(Quasi-)Polish spaces and groupoids}
\label{sec:groupoid}

In this section, we recall some basic definitions involving groupoids and their actions, étalé spaces, and quasi-Polish spaces.

For sets $X, Y, Z$ and functions $f : X -> Z$ and $g : Y -> Z$, the \defn{fiber product} or \defn{pullback} is
\begin{align*}
X \times_Z Y := \{(x, y) \mid f(x) = g(y)\} \subseteq X \times Y.
\end{align*}
The maps $f, g$ are hidden in the notation and will be explicitly specified if not clear from context.

A (small) \defn{groupoid} $\!G = (G^0, G^1)$ consists of a set of \emph{objects} $G^0$, a set of \emph{morphisms} $G^1$, \emph{source} and \emph{target} maps $\partial_1, \partial_0 : G^1 -> G^0$ (note the order; usually $g \in G^1$ with $(\partial_1(g), \partial_0(g)) = (x, y)$ is denoted $g : x -> y$), a \emph{unit} map $\iota : G^0 -> G^1$ (usually denoted $\iota(x) = 1_x$), a \emph{multiplication} map $\mu : G^1 \times_{G^0} G^1 -> G^1$ (usually denoted $\mu(h, g) = h \cdot g$; here $\times_{G^0}$ means fiber product with respect to $\partial_1$ on the left and $\partial_0$ on the right), and an \emph{inversion} map $\nu : G^1 -> G^1$, subject to the usual axioms ($\partial_0(h \cdot g) = \partial_0(h)$, $g^{-1} \cdot g = 1_{\partial_1(g)}$, etc.).

An \defn{action} of a groupoid $\!G$ on a set $X$ equipped with a map $p : X -> G^0$ is a map $a : G^1 \times_{G^0} X -> X$ (usually denoted $a(g, x) = g \cdot x$; here $\times_{G^0}$ means $\partial_1$ on the left and $p$ on the right) satisfying the usual axioms ($p(g \cdot x) = \partial_0(g)$, $1_{p(x)} \cdot x = x$, and $(h \cdot g) \cdot x = h \cdot (g \cdot x)$).  The set $X$ equipped with the map $p$ and the action is also called a \defn{$\!G$-set}.  A \defn{$\!G$-equivariant map} between two $\!G$-sets $p : X -> G^0$ and $q : Y -> G^0$ is a map $f : X -> Y$ such that $p = q \circ f$ and $f(g \cdot x) = g \cdot f(x)$.

The \defn{trivial action} of a groupoid $\!G$ on $G^0$ equipped with the identity $1_{G^0} : G^0 -> G^0$ is given by $g \cdot x := y$ for $g : x -> y$.  Note that for any $\!G$-set $p : X -> G^0$, $p$ is $\!G$-equivariant.

A \defn{topological groupoid} is a groupoid $\!G = (G^0, G^1)$ in which $G^0, G^1$ are topological spaces and the structure maps $\partial_1, \partial_0, \iota, \mu, \nu$ are continuous.  We will usually be concerned with topological groupoids $\!G$ which are \defn{open}, meaning that $\partial_1$ (equivalently $\partial_0, \mu$) is an open map.  A \defn{continuous action} of a topological groupoid $\!G$ is an action $a : G^1 \times_{G^0} X -> X$ on $p : X -> G^0$ such that $X$ is a topological space and $p, a$ are continuous.  In that case, $(X, p, a)$ is a \defn{continuous $\!G$-space}.

A continuous map $f : X -> Y$ between topological spaces is \defn{étalé} (or a \defn{local homeomorphism}) if $X$ has a cover by open sets $U \subseteq X$ such that $f|U : U -> Y$ is an open embedding ($U$ is then an \defn{open section over $f(U) \subseteq Y$}), and \defn{countable étalé} if such a cover can be taken to be countable.  A continuous action of a topological groupoid $\!G$ on $p : X -> G^0$ is \defn{(countable) étalé} if $p$ is (countable) étalé.  We denote the category of countable étalé $\!G$-spaces and continuous $\!G$-equivariant maps by
\begin{align*}
\Act_{\omega_1}(\!G).
\end{align*}

We will need the following standard facts about (countable) étalé maps:

\begin{lemma}
\label{thm:etale}
Let $X, Y, Z$ be topological spaces.
\begin{enumerate}
\item[(i)]  Étalé maps are open.
\item[(ii)]  (Countable) étalé maps are closed under composition.
\item[(iii)]  If $f : X -> Y$ is continuous, $g : Y -> Z$ is (countable) étalé, and $g \circ f$ is (countable) étalé, then $f$ is (countable) étalé.
\item[(iv)]  If $f : X -> Z$ is continuous and $g : Y -> Z$ is (countable) étalé, then the pullback of $g$ along $f$, i.e., the projection $p : X \times_Z Y -> X$, is (countable) étalé.
\item[(v)]  If $E \subseteq X \times X$ is an equivalence relation such that either (equivalently both) of the projections $p, q : E -> X$ is open, then the quotient map $h : X -> X/E$ is open.
\item[(vi)]  If $f : X -> Y$ is an open surjection, $g : Y -> Z$ is continuous, and $g \circ f$ is (countable) étalé, then $g$ is (countable) étalé.
\end{enumerate}
\end{lemma}
\begin{proof}
Most of these are proved in standard references on sheaf theory (at least without the countability restrictions); see e.g., \cite[\S2.3]{Ten}.

(i):  If $f : X -> Y$ is étalé, with $X$ covered by open sections $U_i$, then for any open $V \subseteq X$, $f(V) = \bigcup_i f(U_i \cap V)$ which is open.

(ii):  If $f : X -> Y$ and $g : Y -> Z$ are étalé, with $X$ covered by open $f$-sections $U_i$ and $Y$ covered by open $g$-sections $V_j$, then $X$ is covered by open $(g \circ f)$-sections $U_i \cap g^{-1}(V_j)$.

(iii):  If $U \subseteq X$ is an open $(g \circ f)$-section, then $f(U) \subseteq Y$ is open, since if $Y = \bigcup_j V_j$ is a cover by open $g$-sections, then $f(U) = \bigcup_j (V_j \cap g^{-1}((g \circ f)(U \cap f^{-1}(V_j))))$.  Thus if $U \subseteq X$ is an open $(g \circ f)$-section, then $U$ is also an open $f$-section.

(iv):  If $V \subseteq Y$ is an open $g$-section, then $X \times_Z V \subseteq X \times_Z Y$ is an open $p$-section.

(v):  If $U \subseteq X$ is open, then $h^{-1}(h(U)) = [U]_E = p(E \cap (X \times U)) \subseteq X$ is open, whence $h(U) \subseteq X/E$ is open.

(vi):  If $U \subseteq X$ is an open $(g \circ f)$-section, then $f(U) \subseteq Y$ is an open $g$-section.
\end{proof}

For basic descriptive set theory, see \cite{Kcdst}.  A \defn{standard Borel groupoid} is a groupoid $\!G = (G^0, G^1)$ in which $G^0, G^1$ are standard Borel spaces and the structure maps $\partial_1, \partial_0, \iota, \mu, \nu$ are Borel.  A \defn{Borel action} of a standard Borel groupoid $\!G$ (or a \defn{Borel $\!G$-space}) is an action $a$ on $p : X -> G^0$ such that $X$ is a standard Borel space and $p, a$ are Borel maps.  The action is \defn{fiberwise countable} if $p$ is (and we call a Borel set $A \subseteq X$ a \defn{Borel section over $f(A)$} if $p|A$ is injective).  We denote the category of fiberwise countable Borel $\!G$-spaces and Borel $\!G$-equivariant maps by
\begin{align*}
\Act^B_{\omega_1}(\!G).
\end{align*}

We will be considering the following generalization of Polish spaces.  We say that a subset of $A \subseteq X$ of a topological space $X$ is $\*\Pi^0_2$ if it is a countable intersection $A = \bigcap_i (U_i \cup F_i)$ where $U_i \subseteq X$ is open and $F_i \subseteq X$ is closed.  Note that if every closed set in $X$ is $G_\delta$ (a countable intersection of open sets, e.g., if $X$ is metrizable), then a set if $\*\Pi^0_2$ iff it is $G_\delta$.  A \defn{quasi-Polish space} is a topological space which is homeomorphic to a $\*\Pi^0_2$ subset of a countable power of the Sierpiński space $\#S = \{0 < 1\}$ (with $\{0\}$ closed but not open).  Quasi-Polish spaces are closed under countable products, $\*\Pi^0_2$ subsets, and continuous open $T_0$ images.  A space is Polish iff it is quasi-Polish and regular.  Every quasi-Polish space can be made Polish by adjoining countably many closed sets to the topology, hence is in particular a standard Borel space.  For these and other basic facts about quasi-Polish spaces, see \cite{deB}.

The following is a version of a more abstract ``presentability'' result; see \cref{thm:etale-kloc}.

\begin{lemma}
\label{thm:etale-qpol}
Let $f : X -> Y$ be a countable étalé map with $Y$ quasi-Polish.  Then $X$ is quasi-Polish.
\end{lemma}
\begin{proof}
Let $\@A$ be a countable basis of open sections in $X$, closed under binary intersections.  Consider
\begin{align*}
g : X &--> Y \times \#S^\@A \\
x &|--> (f(x), \chi_A(x))_{A \in \@A},
\end{align*}
where $\chi_A$ is the indicator function of $A$.  Clearly $g$ is a continuous embedding.  We claim that for $(y, i_A)_{A \in \@A} \in Y \times \#S^\@A$,
\begin{align*}
(y, i_A)_A \in g(X) \iff \underbrace{\exists A\, (i_A = 1)}_{(1)} \AND \forall A, B\, (\underbrace{(i_A \wedge i_B = i_{A \cap B})}_{(2)} \AND \underbrace{(A \subseteq B \implies i_A = \chi_{f(A)}(y) \wedge i_B}_{(3)})).
\end{align*}
$\Longrightarrow$ is straightforward.  For $\Longleftarrow$, given $(y, i_A)_A$ satisfying the right-hand side, by (1), let $A_0$ be such that $i_{A_0} = 1$; then (2) and (3) give that for all $A$, $i_A = 1$ iff $i_{A \cap A_0} = 1$ iff $y \in f(A \cap A_0)$, so in particular $y \in f(A_0)$, whence letting $x \in A_0$ be (unique) such that $f(x) = y$, it is easily verified that $(y, i_A)_A = g(x)$.  Clearly the right-hand side is $\*\Pi^0_2$ in $Y \times \#S^\@A$, so $X$ is quasi-Polish.
\end{proof}

A \defn{(quasi-)Polish groupoid} $\!G$ is a topological groupoid such that $G^0, G^1$ are (quasi-)Polish.  See Lupini \cite{Lup} for the basic theory of Polish groupoids (note that by \emph{Polish groupoid}, \cite{Lup} refers to a slight generalization of what we are calling \emph{open Polish groupoid}).  A \defn{(quasi-)Polish $\!G$-space} is a continuous $\!G$-space $p : X -> G^0$ such that $X$ is (quasi-)Polish.  By \cref{thm:etale-qpol}, for quasi-Polish $\!G$, every countable étalé $\!G$-space is quasi-Polish.

Let $\!G$ be an open Polish groupoid and $p : X -> G^0$ be a Borel $\!G$-space.  For a $\partial_1$-fiberwise open set $U \subseteq G^1$ and a Borel set $A \subseteq X$, the \defn{Vaught transforms} are defined by
\begin{align*}
A^{\tri U} &:= \{x \in X \mid \exists^* g \in \partial_1^{-1}(p(x)) \cap U\, (g \cdot x \in A)\}, \\
A^{*U} &:= \{x \in X \mid \forall^* g \in \partial_1^{-1}(p(x)) \cap U\, (g \cdot x \in A)\}.
\end{align*}
Here $\exists^*$ and $\forall^*$ denote Baire category quantifiers; see \cite[8.J]{Kcdst}.  We also put
\begin{align*}
A^{\oast U} &:= \{x \in X \mid \partial_1^{-1}(p(x)) \cap U \ne \emptyset \AND \forall^* g \in \partial_1^{-1}(p(x)) \cap U\, (g \cdot x \in A)\} \\
&= A^{*U} \cap p^{-1}(\partial_1(U)) \\
&= A^{*U} \cap A^{\tri U}
\end{align*}
(in \cite{Lup}, this is defined to be $A^{*U}$).  We will usually be interested in the case where $U \subseteq G^1$ is open, but it is convenient to have the more general notation available.  Here are some basic properties of the Vaught transforms:

\begin{lemma}
\label{thm:vaught}
\begin{enumerate}
\item[(a)]  $\neg A^{*U} = (\neg A)^{\tri U}$.
\item[(b)]  $A^{\tri U}$ preserves countable unions in each argument.
\item[(c)]  For any basis of open sets $V_i \subseteq G^1$, $A^{\tri U} = \bigcup_{V_i \subseteq U} A^{\oast V_i}$.  (It is enough to assume that the $V_i$ form a ``$\partial_1$-fiberwise weak basis for $U$'', i.e., every nonempty open subset of a $\partial_1$-fiber of $U$ contains a nonempty $\partial_1$-fiber of some $V_i \subseteq U$.)
\item[(d)]  For $U$ open, $(A^{*U})^{*V} = A^{*(U \cdot V)}$ (where $U \cdot V = \{g \cdot h \mid g \in U \AND h \in V \AND \partial_1(g) = \partial_0(h)\}$).
\item[(e)]  If $\partial_0^{-1}(p(A)) \cap U = \partial_0^{-1}(p(A)) \cap V$, then $A^{\tri U} = A^{\tri V}$.
\end{enumerate}
\end{lemma}
\begin{proof}
(a)--(c) are standard (see \cite[2.10.2]{Lup}).

(e) is trivial, amounting to $g \cdot x \in A \implies g \in \partial_0^{-1}(\partial_0(g)) = \partial_0^{-1}(p(g \cdot x)) \subseteq \partial_0^{-1}(p(A))$.

For (d), we have
\begin{align*}
x \in (A^{*U})^{*V}
&\iff \forall^* h \in \partial_1^{-1}(p(x)) \cap V\, \forall^* g \in \partial_1^{-1}(p(h \cdot x)) \cap U = \partial_1^{-1}(\partial_0(h)) \cap U\, (g \cdot h \cdot x \in A); \\
\intertext{applying the Kuratowski--Ulam theorem for open maps (see \cite[A.1]{MT}) to the projection $U \times_{G^0} (\partial_1^{-1}(p(x)) \cap V) -> \partial_1^{-1}(p(x)) \cap V$ (a pullback of $\partial_1 : U -> G^0$, hence open) yields}
&\iff \forall^* (g, h) \in U \times_{G^0} (\partial_1^{-1}(p(x)) \cap V)\, (g \cdot h \cdot x \in A); \\
\intertext{applying Kuratowski--Ulam to the multiplication $\mu : U \times_{G^0} (\partial_1^{-1}(p(x)) \cap V) -> \partial_1^{-1}(p(x))$ yields}
&\iff \forall^* k \in \partial_1^{-1}(p(x))\, \forall^* (g, h) \in \mu^{-1}(k)\, (k \cdot x \in A) \\
&\iff \forall^* k \in \partial_1^{-1}(p(x)) \cap (U \cdot V)\, (k \cdot x \in A) \\
&\iff x \in A^{*(U \cdot V)}.
\qedhere
\end{align*}
\end{proof}

\section{Étalé realizations of fiberwise countable Borel actions}
\label{sec:borel-action-etale}

In this section, we  prove \cref{thm:borel-action-etale}.

Let $\!G$ be an open Polish groupoid, $p : X -> G^0$ be a Borel $\!G$-space, and $\@U$ be a countable basis of open sets in $G^1$.  By the proof of \cite[4.1.1]{Lup} (the Becker--Kechris theorem for Polish groupoid actions), if we let $\@A$ be a countable Boolean algebra of Borel subsets of $X$ generating a Polish topology and closed under $A |-> A^{\tri U}$ for each $U \in \@U$, then
\begin{align*}
\@A^{\tri \@U} := \{A^{\tri U} \mid A \in \@A \AND U \in \@U\}
\end{align*}
generates a topology making $X$ into a Polish $\!G$-space.

\begin{lemma}
\label{thm:vaught-basis}
Under these hypotheses, $\@A^{\tri \@U}$ forms a basis for a topology.
\end{lemma}
\begin{proof}
Clearly $X = X^{\tri G^1} = \bigcup_{U \in \@U} X^{\tri U}$ is covered by $\@A^{\tri \@U}$.  Let $x \in A_1^{\tri U_1} \cap A_2^{\tri U_2}$ where $A_1, A_2 \in \@A$ and $U_1, U_2 \in \@U$; we must find $A_3 \in \@A$ and $U_3 \in \@U$ such that $x \in A_3^{\tri U_3} \subseteq A_1^{\tri U_1} \cap A_2^{\tri U_2}$.  Let $V_1 \subseteq U_1$ and $V_2 \subseteq U_2$ be open so that
\begin{align*}
x \in A_1^{\oast V_1} \cap A_2^{\oast V_2}.
\end{align*}
Let $h_1 \in \partial_1^{-1}(p(x)) \cap V_1$ and $h_2 \in \partial_1^{-1}(p(x)) \cap V_2$.  We have $h_2 = (h_2 \cdot h_1^{-1}) \cdot h_1$, so by continuity of multiplication, there are open $V_3 \ni h_1$ and $V_4 \ni h_2 \cdot h_1^{-1}$ such that $V_4 \cdot V_3 \subseteq V_2$.  Let $U_3, U_4 \in \@U$ be such that $h_2 \cdot h_1^{-1} \in U_4 \subseteq V_4$ and $h_1 \in U_3 \subseteq V_1 \cap V_3 \cap \partial_0^{-1}(\partial_1(U_4))$; the latter is possible since $\partial_0(h_1) = \partial_1(h_2 \cdot h_1^{-1}) \in \partial_1(U_4)$.  Then $U_4 \cdot U_3 \subseteq V_4 \cdot V_3 \subseteq V_2$, so since $h_1 \in \partial_1^{-1}(p(x)) \cap U_3 \ne \emptyset$, from $x \in A_2^{\oast V_2}$ we get
\begin{align*}
x \in A_2^{\oast (U_4 \cdot U_3)} \subseteq (A_2^{*U_4})^{\oast U_3}.
\end{align*}
Since also $U_3 \subseteq V_1$, $\partial_1^{-1}(p(x)) \cap U_3 \ne \emptyset$, and $x \in A_1^{\oast V_1}$, we have $x \in A_1^{\oast U_3}$, so
\begin{align*}
x \in A_1^{\oast U_3} \cap (A_2^{*U_4})^{\oast U_3} = (A_1 \cap A_2^{*U_4})^{\oast U_3} \subseteq (A_1 \cap A_2^{*U_4})^{\tri U_3}.
\end{align*}
Now suppose $y \in (A_1 \cap A_2^{*U_4})^{\tri U_3}$.  Then $y \in A_1^{\tri U_3} \subseteq A_1^{\tri V_1} \subseteq A_1^{\tri U_1}$.  Let $W \subseteq U_3$ be open so that
\begin{align*}
y \in (A_2^{*U_4})^{\oast W}.
\end{align*}
Since $W \subseteq U_3 \subseteq \partial_0^{-1}(\partial_1(U_4))$, it is easily seen that $(A_2^{*U_4})^{\oast W} = A_2^{\oast (U_4 \cdot W)}$, whence
\begin{align*}
y \in A_2^{\oast (U_4 \cdot W)} \subseteq A_2^{\tri (U_4 \cdot W)} \subseteq A_2^{\tri (U_4 \cdot U_3)} \subseteq A_2^{\tri V_2} \subseteq A_2^{\tri U_2}.
\end{align*}
Thus $(A_1 \cap A_2^{*U_4})^{\tri U_3} \subseteq A_1^{\tri U_1} \cap A_2^{\tri U_2}$.  Put $A_3 := A_1 \cap A_2^{*U_4}$.
\end{proof}

\begin{lemma}
\label{thm:vaught-section}
Let $A \subseteq X$ be a Borel section (i.e., $p|A : A -> G^0$ is injective) and let $U \subseteq G^1$ be $\partial_1$-fiberwise open.  Then $A^{\oast U}$ is also a Borel section.  Furthermore, if $U, V \subseteq G^1$ are open with $V \cdot V^{-1} \subseteq U^{-1} \cdot U$, then $V^{-1} \cdot A^{\oast U}$ (and hence $(A^{\oast U})^{\tri V} \subseteq V^{-1} \cdot A^{\oast U}$) is also a Borel section.
\end{lemma}
\begin{proof}
Let $x, y \in A^{\oast U}$ with $p(x) = p(y)$.  Then
\begin{align*}
\forall^* g \in \partial_1^{-1}(p(x)) \cap U = \partial_1^{-1}(p(y)) \cap U\, (g \cdot x, g \cdot y \in A).
\end{align*}
Let $g \in \partial_1^{-1}(p(x)) \cap U$ such that $g \cdot x, g \cdot y \in A$.  Then $p(g \cdot x) = \partial_0(g) = p(g \cdot y)$, so since $A$ is a Borel section, $g \cdot x = g \cdot y$, whence $x = y$.  So $A^{\oast U}$ is a Borel section.

Now suppose $U, V$ are open, and let $x, y \in V^{-1} \cdot A^{\oast U}$ with $p(x) = p(y)$.  Let $g \in \partial_1^{-1}(p(x)) \cap V$ and $h \in \partial_1^{-1}(p(y)) \cap V$ with $g \cdot x, h \cdot y \in A^{\oast U}$.  Since $h \cdot g^{-1} \in V \cdot V^{-1} \subseteq U^{-1} \cdot U$, there are $k, l \in U$ such that $h \cdot g^{-1} = l^{-1} \cdot k$, i.e., $k \cdot g = l \cdot h$.  In particular, $k \cdot g = l \cdot h \in \partial_1^{-1}(p(x)) \cap (U \cdot g) \cap (U \cdot h) \ne \emptyset$, so that from $g \cdot x \in A^{\oast U}$ we get $x \in A^{\oast (U \cdot g)}$ and hence $x \in A^{\oast ((U \cdot g) \cap (U \cdot h))}$, and similarly $y \in A^{\oast ((U \cdot g) \cap (U \cdot h))}$.  By the first claim, $x = y$.
\end{proof}

\begin{lemma}
\label{thm:vaught-asttri}
For Borel $A \subseteq X$ and open $W \subseteq G^1$,
\begin{align*}
A^{\tri W} = \bigcup \{(A^{\oast U})^{\tri V} \mid U, V \in \@U \AND U \cdot V \subseteq W \AND V \cdot V^{-1} \subseteq U^{-1} \cdot U\}.
\end{align*}
\end{lemma}
\begin{proof}
For $g \in W$, since $g = g \cdot 1_{\partial_1(g)}$ and $\mu$ is continuous, there are $U, V \in \@U$ such that $U \cdot V \subseteq W$, $g \in U$, and $1_{\partial_1(g)} \in V$; thus
\begin{align*}
W = \bigcup \{U \cdot V \mid U, V \in \@U \AND U \cdot V \subseteq W\}.
\end{align*}
So
\begin{align*}
A^{\tri W}
&= A^{\tri \bigcup \{U \cdot V \mid U, V \in \@U \AND U \cdot V \subseteq W\}} \\
&= \bigcup \{A^{\tri (U \cdot V)} \mid U, V \in \@U \AND U \cdot V \subseteq W\} \\
&= \bigcup \{(A^{\tri U})^{\tri V} \mid U, V \in \@U \AND U \cdot V \subseteq W\} \\
&= \bigcup \{(\bigcup \{A^{\oast U'} \mid U \supseteq U' \in \@U\})^{\tri V} \mid U, V \in \@U \AND U \cdot V \subseteq W\} \\
&= \bigcup \{(A^{\oast U'})^{\tri V} \mid U', U, V \in \@U \AND U' \subseteq U \AND U \cdot V \subseteq W\} \\
&= \bigcup \{(A^{\oast U})^{\tri V} \mid U, V \in \@U \AND U \cdot V \subseteq W\}.
\end{align*}
Since $\partial_0^{-1}(p(A^{\oast U})) \cap V = \partial_0^{-1}(p(A^{\oast U})) \cap \partial_0^{-1}(\partial_1(U)) \cap V$, we have
\begin{align*}
(A^{\oast U})^{\tri V} = (A^{\oast U})^{\tri (\partial_0^{-1}(\partial_1(U)) \cap V)}.
\end{align*}
For any open $U, V \subseteq G^1$ with $V \subseteq \partial_0^{-1}(\partial_1(U))$, for each $g \in V$, there is some $h \in U$ with $\partial_0(g) = \partial_1(h)$, whence $g \cdot g^{-1} = 1_{\partial_0(g)} = 1_{\partial_1(h)} = h^{-1} \cdot h \in U^{-1} \cdot U$, whence by continuity of $\mu$ there are open $V_1 \ni g$ and $V_2 \ni g^{-1}$ such that $V_1 \cdot V_2 \subseteq U^{-1} \cdot U$, whence letting $V' \in \@U$ with $g \in V' \subseteq V \cap V_1 \cap V_2^{-1}$, we have $V' \cdot V^{\prime-1} \subseteq U^{-1} \cdot U$; thus
\begin{align*}
V = \bigcup \{V' \in \@U \mid V' \subseteq V \AND V' \cdot V^{\prime-1} \subseteq U^{-1} \cdot U\}.
\end{align*}
So from above we get
\begin{align*}
A^{\tri W}
&= \bigcup \{(A^{\oast U})^{\tri (\partial_0^{-1}(\partial_1(U)) \cap V)} \mid U, V \in \@U \AND U \cdot V \subseteq W\} \\
&= \bigcup \{(A^{\oast U})^{\tri \bigcup \{V' \in \@U \mid V' \subseteq \partial_0^{-1}(\partial_1(U)) \cap V \AND V' \cdot V^{\prime-1} \subseteq U^{-1} \cdot U\}} \mid U, V \in \@U \AND U \cdot V \subseteq W\} \\
&= \bigcup \{(A^{\oast U})^{\tri V'} \mid U, V, V' \in \@U \AND U \cdot V \subseteq W \AND V' \subseteq \partial_0^{-1}(\partial_1(U)) \cap V \AND V' \cdot V^{\prime-1} \subseteq U^{-1} \cdot U\} \\
&= \bigcup \{(A^{\oast U})^{\tri V} \mid U, V \in \@U \AND U \cdot V \subseteq W \AND V \subseteq \partial_0^{-1}(\partial_1(U)) \AND V \cdot V^{-1} \subseteq U^{-1} \cdot U\} \\
&= \bigcup \{(A^{\oast U})^{\tri (\partial_0^{-1}(\partial_1(U)) \cap V)} \mid U, V \in \@U \AND U \cdot V \subseteq W \AND V \cdot V^{-1} \subseteq U^{-1} \cdot U\} \\
&= \bigcup \{(A^{\oast U})^{\tri V} \mid U, V \in \@U \AND U \cdot V \subseteq W \AND V \cdot V^{-1} \subseteq U^{-1} \cdot U\}.
\qedhere
\end{align*}
\end{proof}

\begin{lemma}
\label{thm:vaught-push-pull}
Let $p : X -> G^0$ and $q : Y -> G^0$ be Borel $\!G$-spaces, $f : X -> Y$ be a $\!G$-equivariant map, and $U \subseteq G^1$ be $\partial_1$-fiberwise open.  Then for any $B \subseteq Y$, $f^{-1}(B^{\tri U}) = f^{-1}(B)^{\tri U}$.  If furthermore $f$ is fiberwise countable, then for any $A \subseteq X$, $f(A^{\tri U}) = f(A)^{\tri U}$.
\end{lemma}
\begin{proof}
The first claim is straightforward.
For the second claim, we have $A \subseteq f^{-1}(f(A))$ whence $A^{\tri U} \subseteq f^{-1}(f(A))^{\tri U} = f^{-1}(f(A)^{\tri U})$ whence $f(A^{\tri U}) \subseteq f(A)^{\tri U}$ (regardless of fiberwise countability).  Conversely, if $f$ is fiberwise countable and $y \in f(A)^{\tri U}$, then
\begin{align*}
\{g \in \partial_1^{-1}(q(y)) \cap U \mid g \cdot y \in f(A)\} \subseteq \bigcup_{x \in f^{-1}(y)} \{g \in \partial_1^{-1}(p(x)) \cap U \mid g \cdot x \in A\}
\end{align*}
since given $g$ in the left-hand side we have $g \cdot y = f(x')$ for some $x' \in A$ whence we may take $x := g^{-1} \cdot x'$; since the left-hand side is non-meager and the union on the right is countable, some term in it is non-meager, i.e., there is $x \in f^{-1}(y)$ such that $x \in A^{\tri U}$, whence $y = f(x) \in f(A^{\tri U})$.
\end{proof}

\begin{lemma}
\label{thm:vaught-open}
Let $p : X -> G^0$ be a Polish $\!G$-space and $\@A$ be a basis of open sets in $X$.  Then $\@A^{\tri \@U}$ is also a basis of open sets in $X$.
\end{lemma}
\begin{proof}
Let $B \subseteq X$ be open and $x \in B$.  Since $1_{p(x)} \cdot x = x$ and the action is continuous, there are $1_{p(x)} \in U \in \@U$ and $x \in A \in \@A$ such that $U^{-1} \cdot A \subseteq B$, as well as $1_{p(x)} \in V \in \@U$ such that $V \cdot x \subseteq A$.  Then $(U \cap V) \cdot x \subseteq A$, whence $x \in A^{\tri U} \subseteq U^{-1} \cdot A \subseteq B$.
\end{proof}

\begin{proof}[Proof of \cref{thm:borel-action-etale}]
Let $\@A, \@B$ be countable Boolean algebras of Borel sets in $X, G^0$ respectively, such that
\begin{enumerate}
\item[(i)]  each generates a Polish topology and is closed under $(-)^{\tri U}$ for each $U \in \@U$ (where for $\@B$ this refers to the trivial action of $\!G$ on $G^0$);
\item[(ii)]  $\@A$ contains a countable cover of $X$ by Borel sections (which exists by Lusin--Novikov uniformization; see \cite[18.10]{Kcdst});
\item[(iii)]  $p(A) \in \@B$ for every $A \in \@A$;
\item[(iv)]  $\@B$ contains a countable basis of open sets in $G^0$;
\item[(v)]  $p^{-1}(B) \in \@A$ for every $B \in \@B$.
\end{enumerate}
It is clear that this can be achieved via a $\omega$-length procedure like in the proof of \cite[5.2.1]{BKpol}.  Let $X', G^{\prime0}$ be $X, G^0$ with the topologies generated by $\@A^{\tri \@U}, \@B^{\tri \@U}$ respectively.  By the proof of \cite[4.1.1]{Lup}, $X', G^{\prime0}$ are Polish $\!G$-spaces.  By \cref{thm:vaught-push-pull} and (v), $p : X' -> G^{\prime0}$ is continuous.  By \cref{thm:vaught-open} and (iv), the topology of $G^{\prime0}$ is finer than that of $G^0$.  Let $\!G' = (G^{\prime0}, G^{\prime1})$ be the \defn{action groupoid} of $G^{\prime0}$, where $G^{\prime1} = G^1 \times_{G^0} G^{\prime0}$ is $G^1$ with $\partial_1^{-1}(U)$ adjoined to its topology for each open $U \subseteq G^{\prime0}$, with composition, unit, and inversion in $\!G'$ as in $\!G$.  Then $\!G'$ is an open Polish groupoid (see \cite[2.7.1]{Lup}), and we have a continuous functor $\!G' -> \!G$, namely the identity, which when composed with the action $\!G \curvearrowright X'$ gives a continuous action $\!G' \curvearrowright X'$.  This action is countable étalé, since by \cref{thm:vaught-basis,thm:vaught-section,thm:vaught-asttri} and (ii) $X'$ has a countable basis of open sets of the form $A^{\tri U}$ which are Borel sections, and by \cref{thm:vaught-push-pull} and (iii) the image of each such set is open in $G^{\prime0}$.
\end{proof}

We end this section with a general question concerning Polish groupoids:

\begin{problem}
\label{prob:polish-functor-continuity}
Let $\!G, \!H$ be (open) Polish groupoids and $F : \!G -> \!H$ be a Borel functor.  Is there necessarily a finer (open) Polish groupoid topology on $\!G$ which renders $F$ continuous?
\end{problem}

To motivate this problem, recall that every Borel homomorphism between Polish \emph{groups} is automatically continuous (see e.g., \cite[9.10]{Kcdst}).  Naive form of this statement for Borel functors between Polish groupoids are false.  For example, there is a Polish groupoid $\!G$ and a Borel endofunctor $F : \!G -> \!G$ which is the identity (hence continuous) on objects but not continuous: take $2^{\aleph_0}$ many disjoint copies of a Polish group with a nontrivial automorphism, and apply that automorphism on a Borel but non-clopen set of objects.  \Cref{prob:polish-functor-continuity} is a weaker analog of automatic continuity for Borel functors, meant to exclude such trivial counterexamples.

If \cref{prob:polish-functor-continuity} has a positive solution, then that would imply \cref{thm:borel-action-etale}, since every fiberwise countable Borel action of $\!G$ can be encoded as a Borel functor $\!G -> \!S$ where $\!S$ is the disjoint union of the symmetric groups $S_0, S_1, \dotsc, S_\infty$; see \cref{thm:borel-action-cocycle}.

\section{Imaginary sorts and definable functions}
\label{sec:imag}

In this section, we review $\@L_{\omega_1\omega}$ and the notions of countable fragments, $\omega_1$-coherent formulas, imaginary sorts, definable functions, and the syntactic $\omega_1$-pretopos of a theory.

Let $\@L$ be a first-order language.  For simplicity, we will only consider relational languages; functions may be coded via their graphs in the usual way.  Recall that the logic $\@L_{\omega_1\omega}$ is the extension of finitary first-order logic $\@L_{\omega\omega}$ with countably infinite conjunctions $\bigwedge$ and disjunctions $\bigvee$; see e.g., \cite[11.2]{Gao}.  By \defn{$\@L_{\omega_1\omega}$-formula}, we always mean a formula with finitely many free variables.

We adopt the following convention regarding formulas and free variables.  A \defn{formula-in-context} is a pair $(\vec{x}, \phi(\vec{x}))$ where $\vec{x}$ is a finite tuple of distinct variables and $\phi(\vec{x})$ is a formula with free variables among $\vec{x}$.  We identify formulas-in-context up to variable renaming, i.e., $(\vec{x}, \phi(\vec{x})) = (\vec{y}, \phi(\vec{y}))$.  \emph{By abuse of terminology, henceforth by ``formula'' we always mean ``formula-in-context''; we denote a formula-in-context $(\vec{x}, \phi(\vec{x}))$ simply by $\phi$.}

Given a formula $\phi$ with $n$ variables and an $\@L$-structure $\@M = (M, R^\@M)_{R \in \@L}$, we write $\phi^\@M \subseteq M^n$ for the \defn{interpretation of $\phi$ in $\@M$}.  For an $n$-tuple $\vec{a} \in M$, we write $\phi^\@M(\vec{a})$ or $\vec{a} \in \phi^\@M$ interchangeably.

There is a Gentzen-type proof system for $\@L_{\omega_1\omega}$, which can be found in \cite{Lop} or (in a slightly different presentation) \cite[D1.3]{Jeleph}.  By the Lopez-Escobar completeness theorem \cite{Lop}, this proof system is complete for a \emph{countable} theory $\@T$: if an $\@L_{\omega_1\omega}$-sentence $\phi$ is true in every countable model of $\@T$, then $\phi$ is provable from $\@T$.  In the following definitions, we will often refer to provability, while keeping in mind that this is equivalent to semantic implication in the case of countable theories by soundness and completeness.  In particular, when we say that two formulas are ``equivalent'' or ``$\@T$-equivalent'', we mean that they are provably so.

It is convenient to consider the following restriction of $\@L_{\omega_1\omega}$.  An \defn{$\omega_1$-coherent $\@L$-formula} is an $\@L_{\omega_1\omega}$-formula which uses only atomic formulas, finite $\wedge$ (including nullary $\top$), countable $\bigvee$, and $\exists$.  Note that every $\omega_1$-coherent formula $\phi$ is equivalent to one in the following \defn{normal form}:
\begin{align*}
\phi(\vec{x}) = \bigvee_i \exists \vec{y}_i\, (\phi_{i1}(\vec{x}, \vec{y}_i) \wedge \dotsb \wedge \phi_{ik_i}(\vec{x}, \vec{y}_i)),
\end{align*}
where the $\phi_{ij}$ are atomic.
An \defn{$\omega_1$-coherent $\@L$-axiom} is an $\@L_{\omega_1\omega}$-sentence of the form
\begin{align*}
\forall \vec{x}\, (\phi(\vec{x}) => \psi(\vec{x})),
\end{align*}
where $\phi, \psi$ are $\omega_1$-coherent $\@L$-formulas.  An \defn{$\omega_1$-coherent $\@L$-theory} is a set of $\omega_1$-coherent $\@L$-axioms.

An $\omega_1$-coherent $\@L$-theory $\@T$ is \defn{decidable} if there is an $\omega_1$-coherent $\@L$-formula $\phi(x, y)$ with two free variables which is $\@T$-equivalent to the formula $x \ne y$ (which is \emph{not} $\omega_1$-coherent).
If such a formula $\phi(x, y)$ exists, we will generally denote it by $x \ne y$, it being understood that this refers to an $\omega_1$-coherent compound formula and not the (non-$\omega_1$-coherent) negated atomic formula.

A \defn{fragment} $\@F$ of $\@L_{\omega_1\omega}$ is a set of $\@L_{\omega_1\omega}$-formulas which contains all atomic formulas and is closed under subformulas, finitary first-order logical operations $\wedge, \vee, \neg, \exists, \forall$, and variable substitutions.  An \defn{$\@F$-theory} is an $\@L_{\omega_1\omega}$-theory $\@T$ such that $\@T \subseteq \@F$.
The \defn{Morleyization} of a fragment $\@F$ is the $\omega_1$-coherent theory $\@T'$ in the language $\@L'$ consisting of $\@L$ together with a new relation symbol $R_\phi(\vec{x})$ for each $\@F$-formula $\phi(\vec{x})$, whose axioms consist of
\begin{align}
\stag{Mor}
\begin{array}{r@{}r@{\;}l@{}l}
\forall \vec{x}\, (& R_\phi(\vec{x}) &<=> \phi(\vec{x}) &) \qquad\text{for $\phi$ atomic}, \\
\forall \vec{x}\, (& R_{\bigvee_i \phi_i}(\vec{x}) &<=> \bigvee_i R_{\phi_i}(\vec{x}) &), \\
\forall \vec{x}\, (& R_{\exists y\, \phi(\vec{x}, y)}(\vec{x}) &<=> \exists y\, R_\phi(\vec{x}, y) &), \\
\forall \vec{x}\, (& R_{\phi}(\vec{x}) \wedge R_{\neg \phi}(\vec{x}) &=> \bot &), \\
\forall \vec{x}\, (& \top &=> R_\phi(\vec{x}) \vee R_{\neg \phi}(\vec{x}) &), \\
\forall \vec{x}\, (& R_{\bigwedge_i \phi_i}(\vec{x}) &=> R_{\phi_j}(\vec{x}) &), \\
\forall \vec{x}\, (& \top &=> R_{\bigwedge_i \phi_i}(\vec{x}) \vee \bigvee_i R_{\neg \phi_i}(\vec{x}) &), \\
\forall \vec{x}\, (& R_{\forall y\, \phi(\vec{x}, y)}(\vec{x}) &<=> R_{\neg \exists y\, \neg \phi(\vec{x}, y)}(\vec{x}) &)
\end{array}
\end{align}
whenever the formulas in the subscripts belong to $\@F$ (where the axioms with $<=>$ really abbreviate two $\omega_1$-coherent axioms, $<=$ and $=>$).
Note that if $\@F$ is countable, then so is $\@T'$.  The \defn{Morleyization} of an $\@F$-theory $\@T$ is defined in the same way, except that $\@T'$ also includes the axiom
\begin{align*}
R_\phi
\end{align*}
(which is a nullary relation symbol) for each axiom $\phi$ in $\@T$.  The Morleyization is a decidable theory, as witnessed by the formula $R_{\ne}(x, y)$.  For more on Morleyization, see \cite[\S2.6]{Hod} or \cite[D1.5.13]{Jeleph}.  An $\@F$-theory is ``equivalent'' to its Morleyization, in the following sense:

\begin{lemma}
\label{thm:morley}
Let $\@F$ be a fragment of $\@L_{\omega_1\omega}$, $\@T$ be an $\@F$-theory, and $(\@L', \@T')$ be its Morleyization.
\begin{enumerate}
\item[(i)]  $\@L_{\omega_1\omega}$-formulas modulo $\@T$-equivalence are in canonical bijection with $\@L'_{\omega_1\omega}$-formulas modulo $\@T'$-equivalence.
\item[(ii)]  Countable disjunctions of $\@F$-formulas modulo $\@T$-equivalence are in canonical bijection with $\omega_1$-coherent $\@L'$-formulas modulo $\@T'$-equivalence.
\item[(iii)]  Models of $\@T$ are in canonical bijection with models of $\@T'$.
\end{enumerate}
\end{lemma}
\begin{proof}
By an easy induction, $\@T' |- \forall \vec{x}\, (R_\phi(\vec{x}) <=> \phi(\vec{x}))$ for every $\phi \in \@F$.  Since $R_\phi \in \@T'$ for every $\phi \in \@T$, it follows that $\@T'$ proves every axiom in $\@T$.

For (i), if $\@T$ proves an $\@L_{\omega_1\omega}$-sentence $\phi$, then also $\@T' |- \phi$; thus if two $\@L_{\omega_1\omega}$-formulas $\phi, \psi$ are $\@T$-equivalent, then they are also $\@T'$-equivalent.  Conversely, for an $\@L'_{\omega_1\omega}$-formula $\psi$, an easy induction shows that $\@T' |- \forall \vec{x}\, (\psi <=> \psi')$ where $\psi'$ is $\psi$ with every $R_\phi$ replaced by $\phi$; thus every $\@L'_{\omega_1\omega}$-formula is $\@T'$-equivalent to an $\@L_{\omega_1\omega}$-formula.  Furthermore, for an $\@L'_{\omega_1\omega}$-sentence $\psi$ such that $\@T' |- \psi$, if we take the proof of $\psi$ from $\@T'$ (in the proof system in \cite{Lop} or \cite[D1.3]{Jeleph}) and replace every $R_\phi$ in it with $\phi$, we obtain a proof tree whose root node is $\psi'$ and whose leaves (i.e., axioms of $\@T'$) are all either tautologies (for one of the axioms \eqref{seq:Mor}) or axioms in $\@T$ (for $R_\phi \in \@T'$ where $\phi \in \@T$), whence $\@T |- \psi'$.  Thus if two $\@L'_{\omega_1\omega}$-formulas $\psi, \theta$ are $\@T'$-equivalent, then $\psi', \theta'$ are $\@T$-equivalent.

For (ii), a countable disjunction $\bigvee_i \phi_i$ of $\@F$-formulas $\phi_i$ is $\@T'$-equivalent to $\bigvee_i R_{\phi_i}$, which is a $\omega_1$-coherent $\@L'$-formula.  Conversely, an $\omega_1$-coherent $\@L'$-formula is equivalent to one in normal form $\psi(\vec{x}) = \bigvee_i \exists \vec{y}_i\, (\psi_{i1}(\vec{x}, \vec{y}_i) \wedge \dotsb \wedge \psi_{ik_i}(\vec{x}, \vec{y}_i))$, where the $\psi_{ij}$ are either equalities or some $R_{\phi_{ij}}$ which is $\@T'$-equivalent to $\phi_{ij}$; hence $\exists \vec{y}_i\, (\psi_{i1}(\vec{x}, \vec{y}_i) \wedge \dotsb \wedge \psi_{ik_i}(\vec{x}, \vec{y}_i))$ is $\@T'$-equivalent to an $\@F$-formula (since $\@F$ is closed under $\wedge, \exists$), and so $\psi$ is $\@T'$-equivalent to a countable disjunction of $\@F$-formulas.

For (iii), since $\@T' |- \@T$, the $\@L$-reduct of a model of $\@T'$ is a model of $\@T$; conversely, for a model $\@M$ of $\@T$, an easy induction shows that the unique $\@L'$-expansion of $\@M$ which satisfies $\@T'$ is given by $R_\phi^\@M := \phi^\@M$ for each $\phi \in \@F$.
\end{proof}

Let $(\@L, \@T)$ be an $\omega_1$-coherent theory.  An \defn{$\omega_1$-coherent $\@T$-imaginary sort} $A$ is a pair $A = ((\alpha_i)_{i \in I}, (\epsilon_{ij})_{i,j \in I})$ consisting of countable families of $\omega_1$-coherent formulas $\alpha_i(\vec{x}_i)$ (with possibly different numbers of free variables, say $n_i := \abs{\vec{x}_i}$) and $\epsilon_{ij}(\vec{x}_i, \vec{x}_j)$, such that $\@T$ proves the following sentences which say that ``$\bigsqcup_{i,j} \epsilon_{ij}$ is an equivalence relation on $\bigsqcup_i \alpha_i$'':
\begin{align}
\stag{Eqv}
\begin{array}{r@{}r@{\;}l@{}l}
\forall \vec{x}, \vec{y}\, (& \epsilon_{ij}(\vec{x}, \vec{y}) &=> \alpha_i(\vec{x}) \wedge \alpha_j(\vec{y}) &), \\
\forall \vec{x}\, (& \alpha_i(\vec{x}) &=> \epsilon_{ii}(\vec{x}, \vec{x}) &), \\
\forall \vec{x}, \vec{y}\, (& \epsilon_{ij}(\vec{x}, \vec{y}) &=> \epsilon_{ji}(\vec{y}, \vec{x}) &), \\
\forall \vec{x}, \vec{y}, \vec{z}\, (& \epsilon_{ij}(\vec{x}, \vec{y}) \wedge \epsilon_{jk}(\vec{y}, \vec{z}) &=> \epsilon_{ik}(\vec{x}, \vec{z}) &).
\end{array}
\end{align}
For $(\@L, \@T)$ countable, using the completeness theorem, this is easily seen to be equivalent to: in every countable model $\@M = (M, R)_{R \in \@L}$ of $\@T$, the set
\begin{align*}
\bigsqcup_{i,j} \epsilon_{ij}^\@M \subseteq \bigsqcup_{i,j} (M^{n_i} \times M^{n_j}) \cong (\bigsqcup_i M^{n_i})^2
\end{align*}
is an equivalence relation on
\begin{align*}
\bigsqcup_i \alpha_i^\@M \subseteq \bigsqcup_i M^{n_i}.
\end{align*}
The \defn{interpretation of $A$ in $\@M$} is the quotient set
\begin{align*}
A^\@M := (\bigsqcup_i \alpha_i^\@M)/(\bigsqcup_{i,j} \epsilon_{ij}^\@M).
\end{align*}
We will denote the imaginary sort $A = ((\alpha_i)_i, (\epsilon_{ij})_{i,j})$ itself suggestively by
\begin{align*}
A = (\bigsqcup_i \alpha_i)/(\bigsqcup_{i,j} \epsilon_{ij}).
\end{align*}

We identify a single formula $\alpha$ with the imaginary sort given by $\alpha$ quotiented by the equality relation (i.e., the imaginary sort $((\alpha), (\epsilon))$ where $\epsilon(\vec{x}, \vec{y}) = \alpha(\vec{x}) \wedge (\vec{x} = \vec{y})$).  Note that the notation $\phi^\@M$ means the same thing whether we regard $\phi$ as a formula or as an imaginary sort.  Likewise, for countably many formulas $\alpha_i$, we write $\bigsqcup_i \alpha_i$ for the corresponding imaginary sort where the equivalence relation is equality.

Let $A = (\bigsqcup_i \alpha_i)/(\bigsqcup_{i,j} \epsilon_{ij})$ and $B = (\bigsqcup_k \beta_k)/(\bigsqcup_{k,l} \eta_{kl})$ be two $\@T$-imaginary sorts, where $\alpha_i = \alpha_i(\vec{x}_i)$ and $\beta_k = \beta_k(\vec{y}_k)$, say.  An \defn{$\omega_1$-coherent $\@T$-definable function} $f : A -> B$ is a $\@T$-equivalence class $f = [(\phi_{ik})_{i,k}]$ of families of formulas $\phi_{ik}(\vec{x}_i, \vec{y}_k)$ such that $\@T$ proves the following sentences which say that ``$\bigsqcup_{i,k} \phi_{ik} \subseteq (\bigsqcup_i \alpha_i) \times (\bigsqcup_k \beta_k)$ is the lift of the graph of a function $A -> B$'':
\begin{align*}
\stag{Fun}
\begin{array}{r@{}r@{\;}l@{}l}
\forall \vec{x}, \vec{y}\, (& \phi_{ik}(\vec{x}, \vec{y}) &=> \alpha_i(\vec{x}) \wedge \beta_k(\vec{y}) &), \\
\forall \vec{x}, \vec{x}', \vec{y}\, (& \phi_{ik}(\vec{x}, \vec{y}) \wedge \epsilon_{ij}(\vec{x}, \vec{x}') &=> \phi_{jk}(\vec{x}', \vec{y}) &), \\
\forall \vec{x}, \vec{y}, \vec{y}'\, (& \phi_{ik}(\vec{x}, \vec{y}) \wedge \eta_{kl}(\vec{y}, \vec{y}') &=> \phi_{il}(\vec{x}, \vec{y}') &), \\
\forall \vec{x}, \vec{y}, \vec{y}'\, (& \phi_{ik}(\vec{x}, \vec{y}) \wedge \phi_{il}(\vec{x}, \vec{y}') &=> \eta_{kl}(\vec{y}, \vec{y}') &), \\
\forall \vec{x}\, (& \alpha_i(\vec{x}) &=> \bigvee_k \exists \vec{y}\, \phi_{ik}(\vec{x}, \vec{y}) &).
\end{array}
\end{align*}
Again by the completeness theorem, for $(\@L, \@T)$ countable this is equivalent to: in every countable model $\@M$ of $\@T$, $\bigsqcup_{i,k} \phi_{ik}^\@M$ is the lift of the graph of a function
\begin{align*}
f^\@M : A^\@M -> B^\@M,
\end{align*}
the \defn{interpretation of $f$ in $\@M$}.

The \defn{identity function} on $A = (\bigsqcup_i \alpha_i)/(\bigsqcup_{i,j} \epsilon_{ij})$ is
\begin{align*}
1_A := [(\epsilon_{ij})_{i,j}] : A -> A.
\end{align*}
Given also $B = (\bigsqcup_k \beta_k)/(\bigsqcup_{k,l} \eta_{kl})$ and $C = (\bigsqcup_m \gamma_m)/(\bigsqcup_{m,n} \xi_{m,n})$ and definable functions $f = [(\phi_{ik})_{i,k}] : A -> B$ and $g = [(\psi_{km})_{k,m}] : B -> C$, their \defn{composite} is $g \circ f := [(\theta_{im})_{i,m}] : A -> C$ where
\begin{align*}
\theta_{im}(\vec{x}, \vec{z}) := \bigvee_k \exists \vec{y}\, (\phi_{ik}(\vec{x}, \vec{y}) \wedge \psi_{km}(\vec{y}, \vec{z})).
\end{align*}
It is straightforward to verify (by explicitly writing down formal proofs) that $1_A$ and $g \circ f$ are definable functions and that composition is associative and unital.

The \defn{syntactic $\omega_1$-pretopos} of an $\omega_1$-coherent theory $(\@L, \@T)$ is the category of imaginary sorts and definable functions, denoted
\begin{align*}
\-{\ang{\@L \mid \@T}}_{\omega_1}.
\end{align*}
The categorical structure in the syntactic $\omega_1$-pretopos encodes the logical structure of the theory:

\begin{remark}
\label{list:pretopos-structure}
\leavevmode
\begin{itemize}

\item  There is an object $\#X \in \-{\ang{\@L \mid \@T}}_{\omega_1}$, the \defn{home sort}, given by the true formula $\top(x)$ in one variable (quotiented by the equality relation), whose interpretation in a model $\@M = (M, R^\@M)_{R \in \@L}$ is the underlying set $M$.

\item  The categorical \defn{product} $\#X^n$ of $n$ copies of the home sort $\#X$ is given by the true formula $\top(x_0, \dotsc, x_{n-1})$ in $n$ variables, with the $i$th projection $p_i : \#X^n -> \#X$ defined by the formula $\pi_i(x_0, \dotsc, x_{n-1}, y) = (x_i = y)$.

More generally, given two imaginary sorts $A = (\bigsqcup_i \alpha_i)/(\bigsqcup_{i,j} \epsilon_{ij})$ and $B = (\bigsqcup_k \beta_k)/(\bigsqcup_{k,l} \eta_{kl})$, their \defn{product} is $A \times B = (\bigsqcup_{i,k} \gamma_{ik})/(\bigsqcup_{i,k,j,l} \xi_{ikjl})$, where $\gamma_{ik}(\vec{x}, \vec{y}) = \alpha_i(\vec{x}) \wedge \beta_k(\vec{y})$ and $\xi_{ikjl}(\vec{x}, \vec{y}, \vec{z}, \vec{w}) = \epsilon_{ij}(\vec{x}, \vec{z}) \wedge \eta_{kl}(\vec{y}, \vec{w})$, so that $(A \times B)^\@M \cong A^\@M \times B^\@M$.

\item  Recall that a \defn{subobject} of $\#X^n$ is an equivalence class of monomorphisms $A `-> \#X^n$; as usual, we will abuse terminology and also refer to single monomorphisms as subobjects.  Every formula $\alpha$ with $n$ variables yields a subobject $\alpha `-> \#X^n$ given by the identity function $1_\alpha$ (as defined above, but regarded as a definable function $\alpha -> \#X^n$), with two such subobjects $\alpha, \beta$ being equal iff $\@T |- \alpha <=> \beta$.  Conversely, every subobject of $\#X^n$ is of this form; so there is an (order-preserving) bijection between subobjects of $\#X^n$ and $\@T$-equivalence classes of formulas with $n$ variables.

More generally, given an imaginary sort $A = (\bigsqcup_i \alpha_i)/(\bigsqcup_{i,j} \epsilon_{ij}) \in \-{\ang{\@L \mid \@T}}_{\omega_1}$, the \defn{subobjects} of $A$ are in bijection with ``subsorts'' or \defn{definable relations} on $A$, i.e., families of formulas $(\beta_i)_i$ such that $\@T |- \beta_i => \alpha_i$ and $\@T$ proves that ``$\bigsqcup_i \beta_i \subseteq \bigsqcup_i \alpha_i$ is $(\bigsqcup_{i,j} \epsilon_{ij})$-invariant''.

\item  Given two definable functions $f, g : A -> B$, say $A = (\bigsqcup_i \alpha_i)/(\bigsqcup_{i,j} \epsilon_{ij})$, $B = (\bigsqcup_k \beta_k)/(\bigsqcup_{k,l} \eta_{kl})$, $f = [(\phi_{ik})_{i,k}]$, and $g = [(\psi_{ik})_{i,k}]$, their \defn{equalizer} is the subsort $A' \subseteq A$ on which $f, g$ are equal, given by $A' = (\bigsqcup_i \alpha_i')/(\bigsqcup_{i,j} \epsilon_{ij}')$ where $\alpha_i'(\vec{x}) = \bigvee_k \exists \vec{y}\, (\phi_{ik}(\vec{x}, \vec{y}) \wedge \psi_{ik}(\vec{x}, \vec{y}))$ and $\epsilon_{ij}'(\vec{x}, \vec{x}') = \alpha_i'(\vec{x}) \wedge \alpha_i'(\vec{x}') \wedge \epsilon_{ij}(\vec{x}, \vec{x}')$.  In particular, the equalizer of the two projections $\#X \times \#X -> \#X$ is (equivalent to) the equality formula $x = y$.

\item  Intersection (pullback) of subobjects corresponds to taking conjunction of formulas.

\item  Union (join in subobject lattice) of subobjects corresponds to taking disjunction of formulas.

\item  For a subobject $A `-> \#X^n \times \#X$ corresponding to a formula $\phi(\vec{x}, y)$, the formula $\exists y\, \phi(\vec{x}, y)$ corresponds to the subobject of $\#X^n$ given by the image of the composite $A `-> \#X^n \times \#X -> \#X^n$, where the second map is the projection.

\end{itemize}
The proofs of these statements are straightforward syntactic calculations; see \cite[D1.4]{Jeleph}.
\end{remark}

There is an alternative, multi-step construction of the syntactic $\omega_1$-pretopos.  First, one defines the \defn{syntactic category} $\ang{\@L \mid \@T}_{\omega_1}$ in the same way as $\-{\ang{\@L \mid \@T}}_{\omega_1}$, except that instead of imaginary sorts, one considers only single $\omega_1$-coherent $\@L$-formulas (representing definable subsets); see \cite[8.1.3]{MR} or \cite[D1.4]{Jeleph}.  This category already has all of the structure encoding logical operations in the list above.  To form $\-{\ang{\@L \mid \@T}}_{\omega_1}$, one ``completes'' $\ang{\@L \mid \@T}_{\omega_1}$ by first freely adjoining countable disjoint unions of objects, and then freely adjoining quotients of equivalence relations.  This may be done either directly on the categorical level (see \cite[A1.4.5, A3.3.10]{Jeleph}), or syntactically, by considering multi-sorted $\omega_1$-coherent theories (see \cite[8.4.1]{MR}).  We have chosen to combine these steps, for the sake of brevity.

The notations $\ang{\@L \mid \@T}_{\omega_1}$ and $\-{\ang{\@L \mid \@T}}_{\omega_1}$ are meant to suggest that the syntactic category (resp., $\omega_1$-pretopos) is the category ``freely presented'' by $(\@L, \@T)$ under the categorical structures listed in \cref{list:pretopos-structure} (resp., plus countable disjoint unions and quotients of equivalence relations).  For the precise sense in which this is true, see \cite[D1.4.12]{Jeleph} or \cref{sec:interp} below.

Two $\omega_1$-coherent theories $(\@L, \@T), (\@L', \@T')$ are \defn{($\omega_1$-coherently) Morita equivalent} if their syntactic $\omega_1$-pretoposes are equivalent categories.  Intuitively, this means that the two theories have the same logical structure, modulo different presentations.

Sometimes, it is convenient to change the definition of \defn{imaginary sort} $A \in \-{\ang{\@L \mid \@T}}_{\omega_1}$ to a \emph{$\@T$-equivalence class} of pairs $((\alpha_i)_i, (\epsilon_{ij})_{i,j})$ of formulas.  Doing so results in a definition of $\-{\ang{\@L \mid \@T}}_{\omega_1}$ which is equivalent to the original one, since if two imaginary sorts $A, B$ (in the original sense) were $\@T$-equivalent, then the identity function $1_A : A -> A$ is also an isomorphism $A \cong B$.

For a fragment $\@F$ of $\@L_{\omega_1\omega}$ and an $\@F$-theory $\@T$, we define its \defn{syntactic $\omega_1$-pretopos} $\-{\ang{\@F \mid \@T}}_{\omega_1}$ to be that of its Morleyization.  By \cref{thm:morley}, we may equivalently define $\-{\ang{\@F \mid \@T}}_{\omega_1}$ in the same way as $\-{\ang{\@L \mid \@T}}_{\omega_1}$, but using only countable disjunctions of $\@F$-formulas in the definitions of both ``imaginary sort'' and ``definable function''; we call these \defn{$(\@F, \@T)$-imaginary sorts} (or simply \defn{$\@F$-imaginary sorts}) and \defn{$(\@F, \@T)$-definable functions}.  (If we quotient by $\@T$-equivalence in the definition of imaginary sort, the two definitions of $\-{\ang{\@F \mid \@T}}_{\omega_1}$ become isomorphic and not merely equivalent.)

For an $\@L_{\omega_1\omega}$-theory $\@T$, we define its \defn{syntactic Boolean $\omega_1$-pretopos} $\-{\ang{\@L \mid \@T}}^B_{\omega_1}$ to be that of $\@T$ regarded as a theory in the \emph{uncountable} fragment of all $\@L_{\omega_1\omega}$-formulas.  By \cref{thm:morley}, this is equivalent (or isomorphic if we consider sorts modulo $\@T$-equivalence) to taking the definition of $\-{\ang{\@L \mid \@T}}_{\omega_1}$ but allowing \emph{arbitrary} $\@L_{\omega_1\omega}$-formulas; we call the objects and morphisms \defn{$(\@L_{\omega_1\omega}, \@T)$-imaginary sorts} (or simply \defn{$\@L_{\omega_1\omega}$-imaginary sorts}) and \defn{$(\@L_{\omega_1\omega}, \@T)$-definable functions}.  Note that since every $\@L_{\omega_1\omega}$-formula is contained in a countable fragment, for $\@T$ countable, $\-{\ang{\@L \mid \@T}}^B_{\omega_1}$ is the direct limit of $\-{\ang{\@F \mid \@T}}_{\omega_1}$ as $\@F$ varies over all countable fragments containing $\@T$.  Two theories $(\@L, \@T)$ and $(\@L', \@T')$ are \defn{($\@L_{\omega_1\omega}$-)Morita equivalent} if their syntactic Boolean $\omega_1$-pretoposes are equivalent.

\section{The groupoid of models}
\label{sec:isogpd}

In this section, we define the space and groupoid of countable models of a theory.  We will first consider the general case of an $\omega_1$-coherent theory, and then specialize (via Morleyization) to the more familiar case of an $\@F$-theory in a countable fragment $\@F$.

Let $\@L$ be a countable relational language.  The \defn{space of countable $\@L$-structures} $\Mod(\@L)$ consists of countable $\@L$-structures $\@M = (M, R)_{R \in \@L}$ whose underlying set $M$ is an initial segment of $\#N$ (i.e., one of $0, 1, 2, \dotsc, \#N$, where as usual $n = \{0, \dotsc, n-1\}$ for $n \in \#N$), equipped with the topology generated by the subbasic open sets denoted by the following symbols:
\begin{align*}
\begin{aligned}
\den{\abs{\#X} \ge n} &:= \{\@M \in \Mod(\@L) \mid \abs{M} \ge n\} &&\text{for $n \in \#N$}, \\
\den{R(\vec{a})} &:= \{\@M \in \Mod(\@L) \mid \vec{a} \in M^n \AND R^\@M(\vec{a})\} &&\text{for $n$-ary $R \in \@L$ and $\vec{a} \in \#N^n$}
\end{aligned}
\end{align*}
(here ``$\#X$'' is thought of as the home sort).  We have a homeomorphism
\begin{align*}
\Mod(\@L) &\cong \left\{(x, y_R)_{R \in \@L} \relmiddle| \begin{aligned}
&\forall a \in \#N\, (x(a+1) = 1 \implies x(a) = 1) \AND \\
&\forall \text{$n$-ary $R \in \@L$},\, \vec{a} \in \#N^n\, (y_R(\vec{a}) = 1 \implies x(a_0) = \dotsb = x(a_{n-1}) = 1)
\end{aligned}\right\} \\
&\subseteq \#S^\#N \times \prod_{\text{$n$-ary $R$} \in \@L} \#S^{\#N^n}
\end{align*}
to a $\*\Pi^0_2$ subset of a countable power of $\#S$, whence $\Mod(\@L)$ is a quasi-Polish space.

For an $\@L_{\omega_1\omega}$-formula $\phi$ with $n$ variables and $\vec{a} \in \#N^n$, we define
\begin{align*}
\den{\phi(\vec{a})} &:= \{\@M \in \Mod(\@L) \mid \vec{a} \in M^n \AND \phi^\@M(\vec{a})\}.
\end{align*}
Note that the subbasic open set $\den{\abs{\#X} \ge n}$ above can also be written as $\den{\top(0, \dotsc, n-1)}$ (i.e., we consider the true formula $\top$ with $n$ variables).  Let us say that a \defn{basic formula} is a finite conjunction of atomic relations $R(\vec{x})$.  Thus, a countable basis of open sets in $\Mod(\@L)$ consists of $\den{\phi(\vec{a})}$ for basic formulas $\phi$.

By the usual induction on $\phi$ (see \cite[16.7]{Kcdst}), $\den{\phi(\vec{a})}$ is a Borel subset of $\Mod(\@L)$.  Moreover if $\phi$ is $\omega_1$-coherent, then it is easily seen that $\den{\phi(\vec{a})}$ is open.  It follows that for two $\omega_1$-coherent formulas $\phi, \psi$, $\den{\phi(\vec{a}) => \psi(\vec{a})}$ is the union of a closed set and an open set, and hence that for an $\omega_1$-coherent axiom $\phi$, $\den{\phi}$ is $\*\Pi^0_2$.  For a countable $\omega_1$-coherent $\@L$-theory $\@T$, put
\begin{align*}
\Mod(\@L, \@T) := \bigcap_{\phi \in \@T} \den{\phi} \subseteq \Mod(\@L).
\end{align*}
This is also a quasi-Polish space, the \defn{space of countable models of $\@T$}.  We will continue to denote $\den{\phi(\vec{a})} \cap \Mod(\@L, \@T) \subseteq \Mod(\@L)$ by $\den{\phi(\vec{a})}$; similarly with $\den{\abs{\#X} \ge n}$.

For a countable fragment $\@F$ of $\@L_{\omega_1\omega}$ and a countable $\@F$-theory $\@T$, we define the \defn{space of countable models of $\@T$ with topology induced by $\@F$} to be
\begin{align*}
\Mod(\@F, \@T) := \Mod(\@L', \@T')
\end{align*}
where $(\@L', \@T')$ is the Morleyization of $\@T$.  Using \cref{thm:morley}, it is easily seen that $\Mod(\@F, \@T)$ is equivalently the set of countable models of $\@T$ on initial segments of $\#N$, equipped with the topology generated by the sets
\begin{align*}
\den{\phi(\vec{a})} \quad\text{for $\phi \in \@F$}.
\end{align*}
Since $\@F$ is closed under $\neg$, the topology is zero-dimensional, hence regular, hence Polish.  This is the usual definition of the topology induced by a countable fragment; see \cite[11.4]{Gao}.

For a countable $\@L_{\omega_1\omega}$-theory $\@T$, we define its standard Borel \defn{space of countable models} as
\begin{align*}
\Mod(\@L, \@T) := \Mod(\@F, \@T)
\end{align*}
for any countable fragment $\@F$ containing $\@T$.
Since for two countable fragments $\@F' \supseteq \@F \supseteq \@T$, the Polish topology induced by $\@F'$ is clearly finer than that induced by $\@F$, the standard Borel structure on $\Mod(\@L, \@T)$ does not depend on which countable fragment we take.  Moreover if $\@T$ happens to be $\omega_1$-coherent, the standard Borel structure on $\Mod(\@F, \@T)$ is induced by the topology on $\Mod(\@L, \@T)$.

We now turn to isomorphisms between models.  Let $\Iso(\@L)$ denote the set of triples $(\@N, g, \@M)$ where $\@M, \@N \in \Mod(\@L)$ and $g : \@M \cong \@N$ is an isomorphism.  Let
\begin{align*}
\begin{aligned}
\partial_0 : \Iso(\@L) &--> \Mod(\@L) &
(\@N, g, \@M) &|--> \@N, \\
\partial_1 : \Iso(\@L) &--> \Mod(\@L) &
(\@N, g, \@M) &|--> \@M, \\
\iota : \Mod(\@L) &--> \Iso(\@L) &
\@M &|--> (\@M, 1_\@M, \@M), \\
\mu : \Iso(\@L) \times_{\Mod(\@L)} \Iso(\@L) &--> \Iso(\@L) &
((\@P, h, \@N), (\@N, g, \@M)) &|--> (\@P, h \circ g, \@M), \\
\nu : \Iso(\@L) &--> \Iso(\@L) &
(\@N, g, \@M) &|--> (\@M, g^{-1}, \@N).
\end{aligned}
\end{align*}
Equipped with these maps, $\MOD(\@L) := (\Mod(\@L), \Iso(\@L))$ is a groupoid.  We usually denote its morphisms by $g$ instead of $(\@N, g, \@M)$ when $\@M, \@N$ are clear from context.  We equip $\Iso(\@L)$ with the topology generated by the subbasic open sets
\begin{gather*}
\partial_1^{-1}(U) \quad\text{for $U \subseteq \Mod(\@L)$ (subbasic) open}, \\
\den{a |-> b} := \{(\@N, g, \@M) \mid a, b \in M \AND g(a) = b\} \quad\text{for $a, b \in \#N$}.
\end{gather*}
It is easily verified that the maps $\partial_0, \partial_1, \iota, \mu, \nu$ are continuous,
and that $\partial_1$ is open.  Thus, $\MOD(\@L)$ is an open quasi-Polish groupoid, the \defn{quasi-Polish groupoid of countable $\@L$-structures.}

We note that the space $\Iso(\@L)$ can alternatively be regarded as consisting of pairs $(g, \@M)$ where $\@M \in \Mod(\@L)$ is a countable structure and $g \in S_M$ (the symmetric group on $M$) is a permutation of its underlying set.  This definition of $\Iso(\@L)$ can be regarded as a subspace of $S_\infty \times \Mod(\@L)$ (consisting of the $(g, \@M) \in S_\infty \times \Mod(\@L)$ such that $g$ is the identity outside of $M$), with the subspace topology.

For a countable $\omega_1$-coherent $\@L$-theory $\@T$, we define the \defn{quasi-Polish groupoid of countable models of $\@T$}
\begin{align*}
\MOD(\@L, \@T) = (\Mod(\@L, \@T), \Iso(\@L, \@T)) \subseteq \MOD(\@L)
\end{align*}
to be the full subgroupoid on $\Mod(\@L, \@T) \subseteq \Mod(\@L)$; clearly it is also an open quasi-Polish groupoid.  For a countable theory $\@T$ in a countable fragment $\@F$ of $\@L_{\omega_1\omega}$, we define the \defn{Polish groupoid of countable models of $\@T$ with topology induced by $\@F$}
\begin{align*}
\MOD(\@F, \@T) = (\Mod(\@F, \@T), \Iso(\@F, \@T)) := \MOD(\@L', \@T')
\end{align*}
where $(\@L', \@T')$ is the Morleyization of $\@T$.  For a countable $\@L_{\omega_1\omega}$-theory $\@T$, we define the \defn{standard Borel groupoid of countable models of $\@T$}
\begin{align*}
\MOD(\@L, \@T) = (\Mod(\@L, \@T), \Iso(\@L, \@T)) := \MOD(\@F, \@T)
\end{align*}
for any countable fragment $\@F \supseteq \@T$; again, the Borel structure does not depend on the fragment $\@F$, and is consistent with the topology in case $\@T$ is $\omega_1$-coherent.

\section{Interpretations of imaginary sorts}
\label{sec:den}

In this section, we define the interpretation functor $\den{-}$ taking imaginary sorts to $\MOD(\@L, \@T)$-spaces.  As before, we begin with the general case of an $\omega_1$-coherent theory.

Let $\@L$ be a countable relational language and $\@T$ be a countable $\omega_1$-coherent $\@L$-theory.  For an imaginary sort $A \in \-{\ang{\@L \mid \@T}}_{\omega_1}$, we put
\begin{align*}
\den{A} := \{(\@M, a) \mid \@M \in \Mod(\@L, \@T) \AND a \in A^\@M\},
\end{align*}
the disjoint union the interpretations $A^\@M$ in all models $\@M \in \Mod(\@L, \@T)$, equipped with the projection $\pi : \den{A} -> \Mod(\@L, \@T)$; we may call $\den{A}$ simply the \defn{interpretation of $A$}.  We have the following alternative definition of $\den{A}$ which is uniform over all models, which will yield the topology on $\den{A}$.

For a single $\omega_1$-coherent $\@L$-formula $\alpha$ with $n$ variables, regarded as an imaginary sort, put
\begin{align*}
\den{\alpha} := \{(\@M, \vec{a}) \mid \@M \in \Mod(\@L, \@T) \AND \vec{a} \in M^n \AND \alpha^\@M(\vec{a})\}.
\end{align*}
There is an obvious countable étalé (over $\Mod(\@L, \@T)$) topology on $\den{\alpha}$, with a cover by disjoint open sections of the form
\begin{align*}
\den{\alpha}_{\vec{a}} := \{(\@M, \vec{a}) \mid \@M \in \Mod(\@L, \@T) \AND \vec{a} \in M^n \AND \alpha^\@M(\vec{a})\}, 
\end{align*}
each of which is a section over $\den{\alpha(\vec{a})} \subseteq \Mod(\@L, \@T)$.  Thus, a countable basis for $\den{\alpha}$ consists of sets of the form
\begin{align*}
\den{\alpha}_{\vec{a}} \cap \pi^{-1}(\den{\phi(\vec{b})})
\end{align*}
with $\den{\phi(\vec{b})} \subseteq \Mod(\@L, \@T)$ a basic open set, i.e., $\phi$ a basic formula.  Note that when $\alpha = \top$, so that $\alpha$ as an imaginary sort is a power $\#X^n$ of the home sort $\#X$,
\begin{align*}
\den{\#X^n} = \{(\@M, \vec{a}) \mid \@M \in \Mod(\@L, \@T) \AND \vec{a} \in M^n\};
\end{align*}
and for general $\alpha$ with $n$ variables, we have $\den{\alpha} \subseteq \den{\#X^n}$.

When $n = 0$, the notation $\den{\alpha}$ agrees with the notation $\den{\alpha(\vec{a})} \subseteq \Mod(\@L, \@T)$ defined earlier (for $\vec{a}$ the empty tuple).  For future use, we introduce the following common generalization of both notations: for a formula $\phi$ with $m+n$ variables and $\vec{a} \in \#N^m$, put
\begin{align*}
\den{\phi(\vec{a}, -)} := \{(\@M, \vec{b}) \mid \@M \in \Mod(\@L, \@T) \AND \vec{a} \in M^m \AND \vec{b} \in M^n \AND \phi^\@M(\vec{a}, \vec{b})\}.
\end{align*}
When $n = 0$, this reduces to $\den{\phi(\vec{a})}$; when $m = 0$, this reduces to $\den{\phi}$.

For countably many $\omega_1$-coherent formulas $\alpha_i$, we put
\begin{align*}
\den{\bigsqcup_i \alpha_i} := \bigsqcup_i \den{\alpha_i}
\end{align*}
with the disjoint union topology.  Finally, for an arbitrary imaginary sort $A = (\bigsqcup_i \alpha_i)/(\bigsqcup_{i,j} \epsilon_{ij})$, where $\alpha_i$ has $n_i$ variables, $\den{\bigsqcup_{i,j} \epsilon_{ij}} \subseteq \bigsqcup_{i,j} \den{\#X^{n_i+n_j}} \cong (\bigsqcup_i \den{\#X^{n_i}}) \times_{\Mod(\@L, \@T)} (\bigsqcup_j \den{\#X^{n_j}})$ is fiberwise (over $\Mod(\@L, \@T)$) an equivalence relation on $\den{\bigsqcup_i \alpha_i} \subseteq \bigsqcup_i \den{\#X^{n_i}}$ by \eqref{seq:Eqv} (and soundness); we define $\den{A}$ to be the corresponding quotient
\begin{align*}
\den{(\bigsqcup_i \alpha_i)/(\bigsqcup_{i,j} \epsilon_{ij})} := \den{\bigsqcup_i \alpha_i}/\den{\bigsqcup_{i,j} \epsilon_{ij}}
\end{align*}
with the quotient topology.  By \cref{thm:etale}(v,vi), the quotient of a countable étalé space by an étalé equivalence relation is countable étalé; a countable basis of open sections in $\den{\bigsqcup_i \alpha_i}/\den{\bigsqcup_{i,j} \epsilon_{ij}}$ is given by the images of basic open sections in $\den{\bigsqcup_i \alpha_i}$.

We let the groupoid $\MOD(\@L, \@T)$ act on $\den{A}$ in the obvious way, via application.  That is, for a single formula $\alpha$, we put
\begin{align*}
g \cdot (\@M, \vec{a}) := (\@N, g(\vec{a}))
\end{align*}
for $(\@M, \vec{a}) \in \den{\alpha}$ and $g : \@M \cong \@N$.  This action is continuous, since
\begin{align*}
g \cdot (\@M, \vec{a}) \in \den{\alpha}_{\vec{b}}
&\iff g(\vec{a}) = \vec{b} \AND \alpha^\@N(\vec{b}) \\
&\iff \exists \vec{a}'\, (g \in \bigcap_i \den{a'_i |-> b_i} \AND (\@M, \vec{a}) \in \den{\alpha}_{\vec{a}'}).
\end{align*}
For countably many formulas $\alpha_i$, equip $\den{\bigsqcup_i \alpha_i}$ with the disjoint union of the actions.  For a general imaginary sort $A = (\bigsqcup_i \alpha_i)/(\bigsqcup_{i,j} \epsilon_{ij})$, equip $\den{A}$ with the quotient action.  Thus for every imaginary sort $A \in \-{\ang{\@L \mid \@T}}_{\omega_1}$, we have defined a countable étalé $\MOD(\@L, \@T)$-space $\den{A}$.

For a definable function $f : A -> B$, we let
\begin{align*}
\den{f} : \den{A} -> \den{B}
\end{align*}
be given fiberwise by $f^\@M : A^\@M -> B^\@M$ for each $\@M$.  Again there is a uniform definition, which shows that $\den{f}$ is continuous.  Let $A = (\bigsqcup_i \alpha_i)/(\bigsqcup_{i,j} \epsilon_{ij})$, $B = (\bigsqcup_k \beta_k)/(\bigsqcup_{k,l} \eta_{kl})$, and $f = [(\phi_{ik})_{i,k}]$.  By \eqref{seq:Fun} (and soundness), the sub-countable étalé space $\den{\bigsqcup_{i,k} \phi_{ik}} \subseteq \den{\bigsqcup_i \alpha_i} \times_{\Mod(\@L, \@T)} \den{\bigsqcup_k \beta_k}$ is invariant with respect to the equivalence relation $\den{\bigsqcup_{i,j} \epsilon_{ij}} \times_{\Mod(\@L, \@T)} \den{\bigsqcup_{k,l} \eta_{kl}}$, and its image in the quotient $\den{A} \times_{\Mod(\@L, \@T)} \den{B}$ is fiberwise the graph of a function $\den{f} : \den{A} -> \den{B}$, which is continuous because its fiberwise graph is open.  It is clear that $\den{f}$ is $\MOD(\@L, \@T)$-equivariant, and that $\den{-}$ preserves identity and composition, so that we have defined a functor
\begin{align*}
\den{-} : \-{\ang{\@L \mid \@T}}_{\omega_1} &--> \Act_{\omega_1}(\MOD(\@L, \@T))
\end{align*}
(recall from \cref{sec:groupoid} that $\Act_{\omega_1}(\!G)$ denotes the category of countable étalé $\!G$-spaces).

For a countable fragment $\@F$ of $\@L_{\omega_1\omega}$ and a countable $\@F$-theory $\@T$, we define
\begin{align*}
\den{-} : \-{\ang{\@F \mid \@T}}_{\omega_1} &--> \Act_{\omega_1}(\MOD(\@F, \@T))
\end{align*}
by taking the Morleyization.  Note that since every countable étalé action is Borel, we have $\Act_{\omega_1}(\MOD(\@F, \@T)) \subseteq \Act^B_{\omega_1}(\MOD(\@F, \@T)) = \Act^B_{\omega_1}(\MOD(\@L, \@T))$.  For a countable $\@L_{\omega_1\omega}$-theory $\@T$, we define
\begin{align*}
\den{-} : \-{\ang{\@L \mid \@T}}^B_{\omega_1} &--> \Act^B_{\omega_1}(\MOD(\@L, \@T))
\end{align*}
by regarding $\-{\ang{\@L \mid \@T}}^B_{\omega_1}$ as the direct limit $\injlim_{\@F} \-{\ang{\@F \mid \@T}}_{\omega_1}$ over countable fragments $\@F \supseteq \@T$, i.e., $\den{A}$ for an imaginary sort $A$ is defined as above for any countable fragment $\@F$ containing all of the formulas in $A$, and similarly for definable functions.   Both of these definitions are the same as if we had repeated the definition in the $\omega_1$-coherent case, except that we do not have to re-check that the actions are continuous/Borel.

\section{The Lopez-Escobar theorem}
\label{sec:lopez-escobar}

In this section, we present what is essentially Vaught's proof \cite[3.1]{Vau} of Lopez-Escobar's theorem (see also \cite[16.9]{Kcdst} or \cite[11.3.5]{Gao}).  We do so for the sake of completeness, since we are working in a slightly more general context (we allow finite models), and since later we will need precise statements of some intermediate parts of the proof.

Let $\@L$ be a countable relational language and $\@T$ be a countable \emph{decidable} $\omega_1$-coherent $\@L$-theory.  Recall (\cref{sec:imag}) that this means that the formula $x \ne y$ is $\@T$-equivalent to an $\omega_1$-coherent formula, which by abuse of notation we also write as $x \ne y$.

Recall also the subbasic open sets $\den{a |-> b} \subseteq \Iso(\@L, \@T)$ from \cref{sec:isogpd}, consisting of isomorphisms taking $a$ to $b$.  We say that two tuples $\vec{a}, \vec{b} \in \#N^n$ have the same \defn{equality type}, written $\vec{a} \equiv \vec{b}$, if $a_i = a_j \iff b_i = b_j$ for all $i, j$.  For two such tuples, put $\den{\vec{a} |-> \vec{b}} := \bigcap_i \den{a_i |-> b_i}$.

Let $n \in \#N$, $\vec{a} \in \#N^n$, and $\vec{x}$ be an $n$-tuple of variables.  We introduce the following notational abbreviations for certain $\omega_1$-coherent formulas we will use repeatedly:
\begin{align*}
(\abs{\#X} \ge n) &:= \exists y_0, \dotsc, y_{n-1}\, \bigwedge_{i \ne j} (y_i \ne y_j), \\
(\vec{a} \in \#X^n) &:= (\abs{\#X} \ge \max_i (a_i+1)), \\
(\vec{x} \equiv \vec{a}) &:= \bigwedge_{a_i = a_j} (x_i = x_j) \wedge \bigwedge_{a_i \ne a_j} (x_i \ne x_j), \\
(S_\#X \cdot \vec{x} \ni \vec{a}) &:= (\vec{a} \in \#X^n) \wedge (\vec{x} \equiv \vec{a}).
\end{align*}
These have the expected interpretations in models $\@M \in \Mod(\@L, \@T)$: for $\vec{b} \in M^n$,
\begin{align*}
(\abs{\#X} \ge n)^\@M &\iff \abs{M} \ge n, \\
(\vec{a} \in \#X^n)^\@M &\iff \vec{a} \in M^n, \\
(\vec{b} \equiv \vec{a})^\@M &\iff \vec{b} \equiv \vec{a}, \\
(S_\#X \cdot \vec{b} \ni \vec{a})^\@M &\iff S_M \cdot \vec{b} \ni \vec{a}
\end{align*}
(where $S_M$ is the symmetric group on $M$).  In particular, note that $\den{\abs{\#X} \ge n} \subseteq \Mod(\@L, \@T)$ is as defined in \cref{sec:isogpd}.

\begin{lemma}
\label{lm:lopez-escobar-vaught-open}
Let $\vec{b} \in \#N^k$ and let $U \subseteq \den{\#X^n}$ be open.  Then there is an $\omega_1$-coherent formula $\phi$ with $k+n$ variables such that for all $\vec{a} \equiv \vec{b}$,
\begin{align*}
\den{\vec{a} |-> \vec{b}}^{-1} \cdot U = \den{\phi(\vec{a}, -)}.
\end{align*}
\end{lemma}
\begin{proof}
We may assume that $U$ is a basic open set, i.e., $U = \den{\#X^n}_{\vec{d}} \cap \pi^{-1}(\den{\psi(\vec{f})})$ for some basic formula $\psi$ (say with $l$ variables) and tuples $\vec{d} \in \#N^n$ and $\vec{f} \in \#N^l$.  So
\begin{align*}
U &= \{(\@M, \vec{d}) \mid \@M \in \Mod(\@L, \@T) \AND \vec{d} \in M^n \AND \psi^\@M(\vec{f})\}, \\
\den{\vec{a} |-> \vec{b}}^{-1} \cdot U
&= \{(\@M, \vec{c}) \in \den{\#X^n} \mid \exists g \in S_M\, (g(\vec{a}, \vec{c}) = (\vec{b}, \vec{d}) \AND \psi^\@M(g^{-1}(\vec{f})))\} \\
&= \{(\@M, \vec{c}) \in \den{\#X^n} \mid \exists \vec{e} \in M^l\, (S_M \cdot (\vec{a}, \vec{c}, \vec{e}) \ni (\vec{b}, \vec{d}, \vec{f}) \AND \psi^\@M(\vec{e}))\} \\
&= \den{\phi(\vec{a}, -)},
\end{align*}
where $\phi$ is the formula
\begin{align*}
\phi(x_0, \dotsc, x_{k+n-1}) &= \exists x_{k+n}, \dotsc, x_{k+n+l-1}\, ((S_\#X \cdot \vec{x} \ni (\vec{b}, \vec{d}, \vec{f})) \wedge \psi(x_{k+n}, \dotsc, x_{k+n+l-1})).
\qedhere
\end{align*}
\end{proof}

\begin{corollary}
\label{thm:lopez-escobar-vaught-open}
If $U \subseteq \den{\#X^n}$ is open and $\MOD(\@L, \@T)$-invariant, then there is an $\omega_1$-coherent formula $\phi$ with $n$ variables such that $U = \den{\phi}$.  \qed
\end{corollary}

Now replace $\@T$ above with the Morleyization of a countable theory $\@T$ in a countable fragment $\@F$, so that $\MOD(\@F, \@T)$ is an open Polish groupoid and so we may talk about Vaught transforms.

\begin{lemma}
\label{lm:lopez-escobar-vaught-borel}
Let $\vec{b} \in \#N^k$ and let $B \subseteq \den{\#X^n}$ be Borel.  Then there is an $\@L_{\omega_1\omega}$-formula $\phi$ with $k+n$ variables such that for all $\vec{a} \equiv \vec{b}$,
\begin{align*}
B^{\tri \den{\vec{a} |-> \vec{b}}} = \den{\phi(\vec{a}, -)}.
\end{align*}
\end{lemma}
\begin{proof}
By induction on the complexity of $B$.  For $B$ open, this is \cref{lm:lopez-escobar-vaught-open}.  For a countable union $B = \bigcup_i B_i$, let for each $i$ the corresponding formula for $B_i$ be $\phi_i$, then put $\phi := \bigvee_i \phi_i$.  For a complement $B = \neg C$, using the ``$\partial_1$-fiberwise weak basis for $\den{\vec{a} |-> \vec{b}}$'' (see \cref{thm:vaught}) consisting of $\den{\vec{c} |-> \vec{d}}$ for $\vec{a} \subseteq \vec{c} \equiv \vec{d} \supseteq \vec{b}$, we have
\begin{align*}
B^{\tri \den{\vec{a} |-> \vec{b}}}
&= \bigcup_{\vec{a} \subseteq \vec{c} \equiv \vec{d} \supseteq \vec{b}} B^{\oast \den{\vec{c} |-> \vec{d}}} \\
&= \bigcup_{\vec{a} \subseteq \vec{c} \equiv \vec{d} \supseteq \vec{b}} (B^{*\den{\vec{c} |-> \vec{d}}} \cap p^{-1}(\partial_1(\den{\vec{c} |-> \vec{d}}))) \\
&= \bigcup_{\vec{a} \subseteq \vec{c} \equiv \vec{d} \supseteq \vec{b}} (\neg C^{\tri \den{\vec{c} |-> \vec{d}}} \cap p^{-1}(\den{\vec{c} \in \#X^{\abs{\vec{c}}}} \cap \den{\vec{d} \in \#X^{\abs{\vec{d}}}}));
\end{align*}
let for each $\vec{d} \supseteq \vec{b}$ the formula $\psi_{\vec{d}}$ be such that $C^{\tri \den{\vec{c} |-> \vec{d}}} = \den{\psi_{\vec{d}}(\vec{c}, -)}$ for all $\vec{c} \equiv \vec{d}$, whence
\begin{align*}
B^{\tri \den{\vec{a} |-> \vec{b}}}
&= \bigcup_{\vec{a} \subseteq \vec{c} \equiv \vec{d} \supseteq \vec{b}} (\neg \den{\psi_{\vec{d}}(\vec{c}, -)} \cap p^{-1}(\den{\vec{c} \in \#X^{\abs{\vec{c}}}} \cap \den{\vec{d} \in \#X^{\abs{\vec{d}}}})) \\
&= \{(\@M, \vec{e}) \mid \vec{e} \in M^n \AND \exists \vec{a} \subseteq \vec{c} \equiv \vec{d} \supseteq \vec{b}\, (\vec{c}, \vec{d} \in M^{\abs{\vec{d}}} \AND \neg \psi_{\vec{d}}^\@M(\vec{c}, \vec{e}))\} \\
&= \den{\phi(\vec{a}, -)}
\end{align*}
where
\begin{align*}
\phi(x_0, \dotsc, x_{k-1}, y_0, \dotsc, y_{n-1})
&= \bigvee_{\vec{d} \supseteq \vec{b}} \exists x_k, \dotsc, x_{\abs{\vec{d}}-1}\, ((S_\#X \cdot \vec{x} \ni \vec{d}) \wedge \neg \psi_{\vec{d}}(\vec{x}, \vec{y})).
\qedhere
\end{align*}
\end{proof}

\begin{corollary}
\label{thm:lopez-escobar-vaught-borel}
If $B \subseteq \den{\#X^n}$ is Borel and $\MOD(\@L, \@T)$-invariant, then there is an $\@L_{\omega_1\omega}$-formula $\phi$ with $n$ variables such that $B = \den{\phi}$.  \qed
\end{corollary}

The usual statement of Lopez-Escobar's theorem is the case $n = 0$.

\section{Naming countable étalé actions}
\label{sec:etale-interp-equiv}

In this section, we give a direct proof of the following generalization of \cref{thm:etale-interp-equiv} (which follows from it via Morleyization).  The proof is analogous to that of a similar result of Awodey--Forssell \cite[\S1.4]{AF} for $\@L_{\omega\omega}$-theories.  Later in \cref{sec:joyal-tierney} we will give a more abstract proof of this result, by extracting it from the Joyal--Tierney representation theorem.

\begin{theorem}
\label{thm:etale-coherent-interp-equiv}
Let $\@L$ be a countable relational language and $\@T$ be a countable decidable $\omega_1$-coherent $\@L$-theory.  Then the interpretation functor
\begin{align*}
\den{-} : \-{\ang{\@L \mid \@T}}_{\omega_1} &--> \Act_{\omega_1}(\MOD(\@L, \@T))
\end{align*}
is an equivalence of categories.
\end{theorem}
\begin{proof}
We will prove that the functor is conservative, full on subobjects, and essentially surjective.  This is enough, since by standard category theory, a finite limit-preserving functor between categories with finite limits is an equivalence iff it is conservative, full on subobjects, and essentially surjective (see e.g., \cite[D3.5.6]{Jeleph}),
and $\den{-}$ is easily seen to preserve finite limits (using the explicit constructions of finite products and equalizers of imaginary sorts in \cref{list:pretopos-structure}).

Conservative means that given an imaginary sort $A = (\bigsqcup_i \alpha_i)/(\bigsqcup_{i,j} \epsilon_{ij}) \in \-{\ang{\@L \mid \@T}}_{\omega_1}$ and a subsort (i.e., definable relation) $B \subseteq A$, if $\den{B} = \den{A}$, then $B$ and $A$ are provably equivalent.  Recall (\cref{list:pretopos-structure}) that $B$ is given by a family of formulas $(\beta_i)_i$ such that $\@T$ proves that ``$\bigsqcup_i \beta_i \subseteq \bigsqcup_i \alpha_i$ is $(\bigsqcup_{i,j} \epsilon_{ij})$-invariant''.  If $\den{B} = \den{A}$, then clearly $\den{\beta_i} = \den{\alpha_i}$ for every $i$, i.e., $\beta_i^\@M = \alpha_i^\@M$ for every countable model $\@M$ of $\@T$.  By the completeness theorem, $\@T |- \beta_i <=> \alpha_i$, as desired.

Full on subobjects means that given an imaginary sort $A = (\bigsqcup_i \alpha_i)/(\bigsqcup_{i,j} \epsilon_{ij}) \in \-{\ang{\@L \mid \@T}}_{\omega_1}$ and a sub-countable étalé $\MOD(\@L, \@T)$-space $Y \subseteq \den{A}$, there is a subsort $B \subseteq A$ such that $\den{B} = Y$.  Since $Y$ is étalé, $Y \subseteq \den{A}$ is open.  Let $q : \den{\bigsqcup_i \alpha_i} -> \den{A}$ be the quotient map.  Then for each $i$, $\den{\alpha_i} \cap q^{-1}(Y) \subseteq \den{\alpha_i}$ is open and $\MOD(\@L, \@T)$-invariant, hence by \cref{thm:lopez-escobar-vaught-open} equal to $\den{\beta_i}$ for some $\omega_1$-coherent formula $\beta_i$.  Since $\den{\beta_i} \subseteq \den{\alpha_i}$, by the completeness theorem (or conservativity as above applied to $\alpha_i \wedge \beta_i \subseteq \beta_i$), $\beta_i \subseteq \alpha_i$.  Similarly, since $q^{-1}(Y) = \bigsqcup_i \den{\beta_i} \subseteq \den{\bigsqcup_i \alpha_i}$ is $\den{\bigsqcup_{i,j} \epsilon_{ij}}$-invariant, $\@T$ proves that ``$\bigsqcup_i \beta_i \subseteq \bigsqcup_i \alpha_i$ is $(\bigsqcup_{i,j} \epsilon_{ij})$-invariant''.  So the desired subsort $B$ is given by $(\beta_i)_i$.

We now come to the heart of the proof: essentially surjective means that given a countable étalé $\MOD(\@L, \@T)$-space $p : X -> \Mod(\@L, \@T)$, there is an imaginary sort $A \in \-{\ang{\@L \mid \@T}}_{\omega_1}$ such that $X \cong \den{A}$.  Since $X$ is countable étalé, it has a countable basis $\@U$ of open sections $U \subseteq X$, which we may assume to be over basic open sets $\den{\phi(\vec{a})} \subseteq \Mod(\@L, \@T)$.  We claim that we may choose these so that
\begin{align*}
\den{\vec{a} |-> \vec{a}} \cdot U \subseteq U.  \tag{$*$}
\end{align*}

\begin{proof}[Proof of claim][(\textit{Claim})\openbox]
Let $V_i \subseteq X$ be any countable basis of open sections.  Let $x \in X$ and let $W \subseteq X$ be any open section containing $x$.  Since $1_{p(x)} \cdot x = x$, by continuity of the action, there is a basic open section $x \in V_i \subseteq W$ and a basic open set $1_{p(x)} \in \den{\vec{b} |-> \vec{b}'} \cap \partial_1^{-1}(\den{\psi(\vec{c})}) \subseteq \Iso(\@L, \@T)$ such that $(\den{\vec{b} |-> \vec{b}'} \cap \partial_1^{-1}(\den{\psi(\vec{c})})) \cdot V_i \subseteq W$.  That $1_{p(x)} \in \den{\vec{b} |-> \vec{b}'} \cap \partial_1^{-1}(\den{\psi(\vec{c})})$ means that $\vec{b} = \vec{b}'$ and that the model $p(x) \in \Mod(\@L, \@T)$ contains $\vec{b}$ and $\vec{c}$ and satisfies $\psi^{p(x)}(\vec{c})$.  Since $p(x) \in p(V_i)$, there is a basic open set $p(x) \in \den{\theta(\vec{d})} \subseteq p(V_i)$.  Putting $\vec{a} := (\vec{b}, \vec{c}, \vec{d})$ and $\phi(\vec{x}, \vec{y}, \vec{z}) := \psi(\vec{y}) \wedge \theta(\vec{z})$, we have $p(x) \in \den{\phi(\vec{a})} \subseteq \den{\theta(\vec{d})} \subseteq p(V_i)$.  Thus
\begin{align*}
U := V_i \cap p^{-1}(\den{\phi(\vec{a})})
\end{align*}
is an open section over $\den{\phi(\vec{a})}$ containing $x$, such that
\begin{align*}
\den{\vec{a} |-> \vec{a}} \cdot U
= (\den{\vec{a} |-> \vec{a}} \cap \partial_1^{-1}(\den{\phi(\vec{a})})) \cdot U
\subseteq (\den{\vec{b} |-> \vec{b}} \cap \partial_1^{-1}(\den{\psi(\vec{c})})) \cdot V_i \subseteq W.
\end{align*}
This implies $\den{\vec{a} |-> \vec{a}} \cdot U \subseteq U$, since $U \subseteq W$, $W$ is a section, and $\den{\vec{a} |-> \vec{a}} \cdot p(U) = \den{\vec{a} |-> \vec{a}} \cdot \den{\phi(\vec{a})} \subseteq \den{\phi(\vec{a})} = p(U)$.  So we may take $\@U$ to consist of all $U = V_i \cap p^{-1}(\den{\phi(\vec{a})})$ with $\phi$ a basic formula, $\den{\phi(\vec{a})} \subseteq p(V_i)$, and $\den{\vec{a} |-> \vec{a}} \cdot U \subseteq U$.
\end{proof}

Now having found such a basis $\@U$, fix some $U \in \@U$ and associated $\phi, \vec{a}$ satisfying ($*$), say with $\abs{\vec{a}} = n$.  Put
\begin{align*}
\alpha(\vec{x}) := (S_\#X \cdot \vec{x} \ni \vec{a}) \wedge \phi(\vec{x}).
\end{align*}
We claim that we have a $\MOD(\@L, \@T)$-equivariant continuous map $f : \den{\alpha} -> X$ given by
\begin{align*}
f(\@M, \vec{b}) &= x \iff p(x) = \@M \AND \exists g \in S_M\, (g(\vec{b}) = \vec{a} \AND g \cdot x \in U)
\end{align*}
whose image contains $U$.

\begin{proof}[Proof of claim][(\textit{Claim})\openbox]
First, we must check that $f$ so defined is a function.  For $(\@M, \vec{b}) \in \den{\alpha}$, by definition of $\alpha$, there is some isomorphism $g : \@M \cong \@N$ such that $g(\vec{b}) = \vec{a}$, whence $\@N \in \den{\phi(\vec{a})} = p(U)$; letting $y \in U \cap p^{-1}(\@N)$ be the unique element, $x := g^{-1} \cdot y$ is one value for $f(\@M, \vec{b})$.  If $x, x' \in p^{-1}(\@M)$ and there are two isomorphisms $g : \@M \cong \@N$ and $g' : \@M \cong \@N'$ with $g(\vec{b}) = \vec{a} = g'(\vec{b})$ and $g \cdot x, g' \cdot x' \in U$, then $g' \circ g^{-1} \in \den{\vec{a} |-> \vec{a}}$, whence (by ($*$)) $g' \cdot x = (g' \circ g^{-1}) \cdot (g \cdot x) \in U$ with $p(g' \cdot x) = \@N' = p(g' \cdot x')$, whence $g' \cdot x = g' \cdot x'$ since $U$ is a section, whence $x = x'$; thus $f(\@M, \vec{b})$ is unique.

It is straightforward that $f$ is $\MOD(\@L, \@T)$-equivariant.  Furthermore, for $x \in U$, clearly $f(p(x), \vec{a}) = x$ as witnessed by $1_{p(x)} \in S_{p(x)}$; hence the image of $f$ contains $U$.


Finally, we must check that $f$ is continuous.  Since $\den{\alpha} = \bigsqcup_{\vec{b}} \den{\alpha}_{\vec{b}}$, it suffices to check that each $f|\den{\alpha}_{\vec{b}}$ is continuous.  For $\vec{b} \not\equiv \vec{a}$, we have $\den{\alpha}_{\vec{b}} = \emptyset$.  For $\vec{b} \equiv \vec{a}$, put $m := \max_i (a_i+1, b_i+1)$, and fix some $g \in S_\infty$ which is the identity on $\#N \setminus m$ such that $g(\vec{b}) = \vec{a}$.  The map
\begin{align*}
\den{\abs{\#X} \ge m} &--> \Iso(\@L, \@T) \\
\@M &|--> (g|M : \@M \cong g(\@M))
\end{align*}
is easily seen to be continuous, whence for any continuous $\MOD(\@L, \@T)$-space $q : Y -> \Mod(\@L, \@T)$, we have a continuous map (which we denote simply by $g$)
\begin{align*}
g : q^{-1}(\den{\abs{\#X} \ge m}) &--> Y \\
y &|--> (g|q(y)) \cdot y;
\end{align*}
similarly, we have a map $g^{-1} : q^{-1}(\den{\abs{\#X} \ge m}) -> Y$.  Then $f|\den{\alpha}_{\vec{b}}$ factors as the composite
\begin{align*}
\den{\alpha}_{\vec{b}} --->[\cong]{\pi} \den{\alpha(\vec{b})} --->[\cong]{g} \den{\alpha(\vec{a}) \wedge (\abs{\#X} \ge m)} --->[\cong]{p^{-1}} U \cap p^{-1}(\den{\#X \ge m}) --->{g^{-1}} X,
\end{align*}
which is continuous.
\end{proof}

We can now finish the proof of essential surjectivity via a standard covering argument.  For every basic open section $U \in \@U$, we have found a formula $\alpha_U$ and an equivariant continuous map $f_U : \den{\alpha_U} -> X$, defined as above, whose image contains $U$.  Combining these yields an equivariant continuous surjection $f : \den{\bigsqcup_{U \in \@U} \alpha_U} -> X$.  The kernel $\den{\bigsqcup_{U \in \@U} \alpha_U} \times_X \den{\bigsqcup_{U \in \@U} \alpha_U}$ of $f$ is a sub-countable étalé $\MOD(\@L, \@T)$-space of $\den{(\bigsqcup_{U \in \@U} \alpha_U)^2}$ (by \cref{thm:etale}(iv)), hence since (as shown above) $\den{-}$ is full on subobjects, is given by $\den{\bigsqcup_{U,V \in \@U} \epsilon_{UV}}$ for some family of formulas $\epsilon_{ij}$, such that $\@T$ proves that $\bigsqcup_{U,V} \epsilon_{UV}$ is an equivalence relation on $\bigsqcup_U \alpha_U$ by conservativity of $\den{-}$ (or completeness).  Since $f$ is an étalé surjection, the quotient of its kernel is $X$ (by \cref{thm:etale}(i,v,vi)), i.e., putting $A := (\bigsqcup_U \alpha_U)/(\bigsqcup_{U,V} \epsilon_{UV})$, we have $X \cong \den{A}$, as desired.
\end{proof}

\section{Naming fiberwise countable Borel actions}
\label{sec:borel-interp-equiv}

In this section, we prove \cref{thm:borel-interp-equiv} using \cref{thm:etale-interp-equiv} and (the proof of) \cref{thm:borel-action-etale}.

\begin{proof}[Proof of \cref{thm:borel-interp-equiv}]
As in the preceding section, we prove that $\den{-}$ is conservative, full on subobjects, and essentially surjective.  The proofs of conservativity and fullness on subobjects are the same as before, except that for the latter we use \cref{thm:lopez-escobar-vaught-borel} instead of \cref{thm:lopez-escobar-vaught-open}.  We now give the proof of essential surjectivity, i.e., of \cref{thm:borel-interp-eso}.

Let $\@F \supseteq \@T$ be any countable fragment, so that $\MOD(\@F, \@T)$ is an open Polish groupoid whose underlying standard Borel groupoid is $\MOD(\@L, \@T)$.  Let $\@U$ be the open basis for $\Iso(\@F, \@T)$ consisting of sets of the form
\begin{align*}
U = \partial_0^{-1}(C) \cap \den{\vec{a} |-> \vec{b}}
\end{align*}
where $C \subseteq \Mod(\@F, \@T)$ is basic open and $\vec{a} \equiv \vec{b} \in \#N^n$ (recall the definition of $\den{\vec{a} |-> \vec{b}}$ from \cref{sec:lopez-escobar}).  Note that for such $U$ and for any Borel $B \subseteq \Mod(\@F, \@T)$, we have
\begin{align*}
B^{\tri U} = \den{\phi(\vec{a})}  \tag{$*$}
\end{align*}
for some $\@L_{\omega_1\omega}$-formula $\phi$; this follows from \cref{lm:lopez-escobar-vaught-borel} (with $n = 0$) and the observation that $B^{\tri U} = (B \cap \partial_0^{-1}(C))^{\tri \den{\vec{a} |-> \vec{b}}}$.

Now let $p : X -> \Mod(\@L, \@T)$ be a fiberwise countable Borel $\MOD(\@L, \@T)$-space, equivalently a fiberwise countable Borel $\MOD(\@F, \@T)$-space.  We modify the proof of \cref{thm:borel-action-etale} in \cref{sec:borel-action-etale} for $\MOD(\@F, \@T)$ and $X$, using the above basis $\@U$ for $\Iso(\@F, \@T)$, by imposing further conditions on the countable Boolean algebras $\@A, \@B$ used in that proof, while simultaneously keeping track of a countable fragment $\@F' \supseteq \@F$, as follows:
\begin{enumerate}
\item[(vi)]  for each $B \in \@B$ and $U \in \@U$, there exist $\phi, \vec{a}$ so that ($*$) holds and such that $\phi \in \@F'$;
\item[(vii)]  $\@B$ contains each basic open set in $\Mod(\@F', \@T)$.
\end{enumerate}
Clearly these can be achieved by enlarging $\@A, \@B, \@F'$ $\omega$-many times.  The new topology on $\Mod(\@F, \@T)$ produced by \cref{thm:borel-action-etale} is then that of $\Mod(\@F', \@T)$ by ($*$) and \cref{thm:vaught-open}, so that we have turned $X$ into a countable étalé $\MOD(\@F', \@T)$-space.  By \cref{thm:etale-interp-equiv}, $X$ is isomorphic to $\den{A}$ for some $\@F'$-imaginary sort $A \in \-{\ang{\@F' \mid \@T}}_{\omega_1} \subseteq \-{\ang{\@L \mid \@T}}^B_{\omega_1}$.
\end{proof}

\section{Interpretations between theories}
\label{sec:interp}

In this section, we consider the ($2$-)functorial aspects of the passage from theories to their groupoids of models; in particular, we define the notion of an interpretation \emph{between theories}.  This will enable us, in the next section, to compare our results to the various known strong conceptual completeness theorems mentioned in the Introduction, as well as to the main result of \cite{HMM}.  We would like to warn the reader that these two sections involve some rather technical 2-categorical notions.

For the basic theory of $2$-categories, see \cite[B1.1]{Jeleph} or \cite[I~Ch.~7]{Bor}.  In every $2$-category we will consider, all $2$-cells will be invertible.

The following definitions make precise the idea that the syntactic $\omega_1$-pretopos $\-{\ang{\@L \mid \@T}}_{\omega_1}$ of an $\omega_1$-coherent theory $(\@L, \@T)$ is a category with algebraic structure ``presented by $(\@L, \@T)$''; more precisely, it is the ``free category with the structures listed in \cref{list:pretopos-structure}, generated by an object $\#X$, together with subobjects $R \subseteq \#X^n$ for each $n$-ary $R \in \@L$, and satisfying the relations in $\@T$''.

\begin{definition}
\label{defn:spretop}
An \defn{$\omega_1$-pretopos} is a category $\!C$ with the following three kinds of structure (see \cite[A1.3--4]{Jeleph}, \cite[\S3.4]{MR}, \cite{CLW}, \cite[II~Ch.~2]{Bor}):
\begin{itemize}

\item  Finite limits (equivalently, finite products and equalizers) exist.

\item  Every countable family of objects $X_i \in \!C$ has a coproduct $\bigsqcup_i X_i$ which is disjoint and pullback-stable.  \defn{Disjoint} means that the injections $X_i -> \bigsqcup_i X_i$ are monomorphisms and have pairwise empty intersections, i.e., the pullback $X_i \times_{\bigsqcup_k X_k} X_j$ is $X_i$ for $i = j$ and the initial object for $i \ne j$.  \defn{Pullback-stable} means that for every morphism $f : Y -> \bigsqcup_i X_i$, the pullback of the injections $X_i -> \bigsqcup_j X_j$ along $f$ exhibits $Y$ as the coproduct $\bigsqcup_i (X_i \times_{\bigsqcup_j X_j} Y)$.  Such a coproduct is also called a \defn{disjoint union}.

\item  Every equivalence relation $E \subseteq X^2$ has a coequalizer (of the two projections) $E \rightrightarrows X -> X/E$ which is effective and pullback-stable.  An \defn{equivalence relation} $E \subseteq X^2$ is a subobject which is ``reflexive'', ``symmetric'', and ``transitive'', as expressed internally in $\!C$, i.e., $E$ contains the diagonal subobject $X \subseteq X^2$ and is invariant under the ``twist'' automorphism $X^2 -> X^2$, and the pullback of $E$ along the projection $X^3 -> X^2$ omitting the middle coordinate contains $E \times_X E \subseteq X^3$.  \defn{Effective} means that $E$ is (via the two projections) the kernel of $X -> X/E$, i.e., the pullback of $X -> X/E$ with itself.  \defn{Pullback-stable} means that the coequalizer diagram $E \rightrightarrows X -> X/E$ is still a coequalizer diagram after pullback along any morphism $f : Y -> X/E$.  The coequalizer $X/E$ is then also called a \defn{quotient}.

\end{itemize}
An \defn{$\omega_1$-coherent functor} $F : \!C -> \!D$ between two $\omega_1$-pretoposes is a functor preserving these operations.  By combining these operations, every $\omega_1$-pretopos also has the following (and they are also preserved by every $\omega_1$-coherent functor):
\begin{itemize}

\item  The \defn{image} $\im(f)$ of a morphism $f : X -> Y$ in $\!C$ is the quotient of the kernel of $f$, and yields a factorization of $f$ into a regular epimorphism $X ->> \im(f)$ followed by a monomorphism $\im(f) `-> Y$.  This factorization is pullback-stable (along morphisms $Z -> Y$).

\item  Given countably many subobjects $A_i \subseteq X$, their \defn{union} $\bigcup_i A_i \subseteq X$ is the image of the induced map from the disjoint union $\bigsqcup_i A_i -> X$.  Unions are pullback-stable.

%

\end{itemize}
We denote the $2$-category of (small) $\omega_1$-pretoposes, $\omega_1$-coherent functors, and natural \emph{isomorphisms} by $\&{\omega_1PTop}$.  Thus, given two $\omega_1$-pretoposes $\!C, \!D$, the groupoid of $\omega_1$-coherent functors $\!C -> \!D$ and natural isomorphisms between them is denoted
\begin{align*}
\&{\omega_1PTop}(\!C, \!D).
\end{align*}
(We restrict to isomorphisms because we are only considering isomorphisms between models.)
\end{definition}

A typical example of an $\omega_1$-pretopos is the syntactic $\omega_1$-pretopos $\-{\ang{\@L \mid \@T}}_{\omega_1}$ of an $\omega_1$-coherent theory $(\@L, \@T)$.  A simpler example is the full subcategory
\begin{align*}
\!{Count} := \{0, 1, 2, \dotsc, \#N\} \subseteq \!{Set},
\end{align*}
which is a skeleton of the category of countable sets.

\begin{definition}
\label{defn:struct-pretop}
Let $\@L$ be a countable relational language and $\!C$ be an $\omega_1$-pretopos.  An \defn{$\@L$-structure in $\!C$}, $\@M = (M, R^\@M)_{R \in \@L}$, consists of an \defn{underlying object} $M \in \!C$ together with subobjects $R^\@M \subseteq M^n$ for each $n$-ary $R \in \@L$.  An \defn{isomorphism} between $\@L$-structures $f : \@M -> \@N$ is an isomorphism $f : M -> N$ in $\!C$ such that for each $n$-ary $R \in \@L$, $f^n : M^n -> N^n$ restricts to an isomorphism $R^\@M -> R^\@N$.  The groupoid of $\@L$-structures in $\!C$ and isomorphisms is denoted
\begin{align*}
\MOD_\!C(\@L).
\end{align*}
Let $\@M$ be an $\@L$-structure.  For each $\omega_1$-coherent $\@L$-formula $\phi$ with $n$ variables, we define its \defn{interpretation in $\@M$}, $\phi^\@M \subseteq M^n$, by induction on $\phi$ in the expected manner (see \cite[D1.2]{Jeleph}):
\begin{itemize}
\item  For $\phi(\vec{x}) = R(\vec{x})$, $\phi^\@M := R^\@M$.
\item  For $\phi(\vec{x}) = (x_i = x_j)$, $\phi^\@M := M^n$ if $i = j$, otherwise $\phi^\@M := M^{n-1} \subseteq M^n$ is the diagonal which duplicates the $i$th coordinate into the $j$th.
\item  For $\phi = \psi \wedge \theta$, $\phi^\@M$ is the intersection (i.e., pullback) of $\psi^\@M, \theta^\@M \subseteq M^n$.  For $\phi = \top$, $\phi^\@M := M^n$.
\item  For $\phi = \bigvee_i \psi_i$, $\phi^\@M$ is the union $\bigcup_i \psi_i^\@M$.
\item  For $\phi(\vec{x}) = \exists y\, \psi(\vec{x}, y)$, $\phi^\@M$ is the image of the composite $\psi^\@M \subseteq M^{n+1} -> M^n$ (where the second map is the projection onto the first $n$ coordinates).
\end{itemize}
By the usual inductions, interpretations are sound with respect to provability (see \cite[D1.3.2]{Jeleph}), and isomorphisms preserve interpretations of formulas (see \cite[D1.2.9]{Jeleph}).
An $\omega_1$-coherent axiom $\forall \vec{x}\, (\phi(\vec{x}) => \psi(\vec{x}))$ (where $n = \abs{\vec{x}}$) is \defn{satisfied by $\@M$} if $\phi^\@M \subseteq \psi^\@M$ (as subobjects of $M^n$).  For an $\omega_1$-coherent $\@L$-theory $\@T$, we say that $\@M$ is a \defn{model of $\@T$} if $\@M$ satisfies every axiom in $\@T$.  The groupoid of models of $\@T$ in $\!C$ and isomorphisms is the full subgroupoid
\begin{align*}
\MOD_\!C(\@L, \@T) \subseteq \MOD_\!C(\@L).
\end{align*}
Given a model $\@M$ of $\@T$ in $\!C$, we may also interpret imaginary sorts and definable functions in $\@M$, exactly as expected:
\begin{itemize}
\item  For an imaginary sort $A = (\bigsqcup_i \alpha_i)/(\bigsqcup_{i,j} \epsilon_{ij})$, $\bigsqcup_{i,j} \epsilon_{ij}^\@M \subseteq (\bigsqcup_i \alpha_i^\@M)^2$ is an equivalence relation (by soundness and \eqref{seq:Eqv}); the quotient object is $A^\@M$.
\item  For a definable function $f = [(\phi_{ik})_{i,k}] : A -> B$, $f^\@M : A^\@M -> B^\@M$ is the unique morphism whose graph, when pulled back along $(\bigsqcup_i \alpha_i^\@M) \times (\bigsqcup_k \beta_k^\@M) -> A^\@M \times B^\@M$, is $\bigsqcup_{i,k} \phi_{ik}$.
\end{itemize}
\end{definition}

When $\!C = \!{Set}$, $\MOD_\!{Set}(\@L, \@T)$ is the usual groupoid of set-theoretic models of $\@T$.
When $\!C = \!{Count}$, we recover $\MOD(\@L, \@T)$ as defined before (\cref{sec:imag}):
\begin{align*}
\MOD_\!{Count}(\@L, \@T) \cong \MOD(\@L, \@T).
\end{align*}

The general notion of model of $\@T$ in $\!C$ formalizes that of ``an object in $\!C$ equipped with subobjects for each $R \in \@L$ which satisfy the relations in $\@T$''.  By comparing the definition of $\phi^\@M$ with \cref{list:pretopos-structure}, we see that we have a model $\@X$ of $(\@L, \@T)$ in $\-{\ang{\@L \mid \@T}}_{\omega_1}$, called the \defn{universal model}, with underlying object $\#X$ and
\begin{align*}
\phi^\@X = \phi \subseteq \#X^n
\end{align*}
for all $\omega_1$-coherent $\@L$-formulas $\phi$ with $n$ variables.

Let $F : \!C -> \!D$ be an $\omega_1$-coherent functor between $\omega_1$-pretoposes.  Since the definitions of ``model of $\@T$ in $\!C$'' and ``isomorphism between models'' use only the categorical structure found in an arbitrary $\omega_1$-pretopos, which is preserved by an $\omega_1$-coherent functor, we get a functor
\begin{align*}
F_* : \MOD_\!C(\@L, \@T) &--> \MOD_\!D(\@L, \@T) \\
(M, R^\@M)_{R \in \@L} &|--> (F(M), F(R^\@M))_{R \in \@L}.
\end{align*}
We are finally ready to state the universal property of the syntactic $\omega_1$-pretopos:

\begin{proposition}
\label{thm:synptop-free}
For any $\omega_1$-coherent theory $(\@L, \@T)$, the syntactic $\omega_1$-pretopos $\-{\ang{\@L \mid \@T}}_{\omega_1}$ is the free $\omega_1$-pretopos containing a model $\@X$ of $\@T$: for any other $\omega_1$-pretopos $\!C$, we have an equivalence of groupoids
\begin{align*}
\&{\omega_1PTop}(\-{\ang{\@L \mid \@T}}_{\omega_1}, \!C) &--> \MOD_\!C(\@L, \@T) \\
F &|--> F_*(\@X)
\end{align*}
between $\omega_1$-coherent functors $\-{\ang{\@L \mid \@T}}_{\omega_1} -> \!C$ and models of $\@T$ in $\!C$.
\end{proposition}
\begin{proof}
An inverse equivalence takes a model $\@M$ to the functor $\-{\ang{\@L \mid \@T}}_{\omega_1} -> \!C$ which takes imaginary sorts and definable functions to their interpretations in $\@M$.
For details, see \cite[D1.4.7, D1.4.12]{Jeleph} (which deals with finitary logic, but generalizes straightforwardly).
\end{proof}

We define an \defn{($\omega_1$-coherent) interpretation} $F : (\@L, \@T) -> (\@L', \@T')$ between two $\omega_1$-coherent theories to mean an $\omega_1$-coherent functor between their syntactic $\omega_1$-pretoposes $F : \-{\ang{\@L \mid \@T}}_{\omega_1} -> \-{\ang{\@L' \mid \@T'}}_{\omega_1}$.  By \cref{thm:synptop-free}, an interpretation is equivalently a model of $\@T$ in $\-{\ang{\@L' \mid \@T'}}_{\omega_1}$, which can be rephrased in more familiar terms:
\begin{itemize}
\item  an imaginary sort $F(\#X) \in \-{\ang{\@L' \mid \@T'}}_{\omega_1}$;
\item  for each $n$-ary $R \in \@L$, a subsort $F(R) \subseteq F(\#X)^n$;
\item  such that for each axiom $\forall \vec{x}\, (\phi(\vec{x}) => \psi(\vec{x}))$ in $\@T$, the corresponding inclusion of subsorts $F(\phi) \subseteq F(\psi)$ is $\@T'$-provable (where $F(\phi), F(\psi)$ are defined by induction in the obvious way).
\end{itemize}

We have analogous notions for $\@L_{\omega_1\omega}$-theories.  An $\omega_1$-pretopos $\!C$ is \defn{Boolean} if for every object $X \in \!C$, the lattice of subobjects of $X$ is a Boolean algebra.  Clearly, an $\omega_1$-coherent functor automatically preserves complements of subobjects when they exist.  We denote the 2-category of Boolean $\omega_1$-pretoposes (a full sub-2-category of $\&{\omega_1PTop}$) by
\begin{align*}
\&{B\omega_1PTop}.
\end{align*}
Given an $\@L$-structure $\@M$ in a Boolean $\omega_1$-pretopos $\!C$, we may define the interpretation $\phi^\@M$ not just for $\omega_1$-coherent $\@L$-formulas $\phi$, but for all $\@L_{\omega_1\omega}$-formulas $\phi$; thus, we may speak of $\@M$ being a \defn{model of $\@T$} for an arbitrary $\@L_{\omega_1\omega}$-theory $\@T$, meaning that $\phi^\@M = 1$ (the terminal object) for each $\phi \in \@T$.
The \defn{universal model} $\@X \in \MOD_{\-{\ang{\@L \mid \@T}}^B_{\omega_1}}(\@L, \@T)$ is defined in the same way as before.  An easy generalization of \cref{thm:morley} shows

\begin{lemma}
\label{thm:morley-catmod}
Let $\@T$ be an $\@L_{\omega_1\omega}$-theory and $(\@L', \@T')$ be its Morleyization in the (uncountable) fragment of all $\@L_{\omega_1\omega}$-formulas.  Then for any Boolean $\omega_1$-pretopos $\!C$, we have an isomorphism of groupoids $\MOD_\!C(\@L', \@T') \cong \MOD_\!C(\@L, \@T)$, which takes a model of $\@T'$ to its $\@L$-reduct.  \qed
\end{lemma}

The universal property of the syntactic Boolean $\omega_1$-pretopos is analogous to \cref{thm:synptop-free}:

\begin{proposition}
\label{thm:synbptop-free}
For an $\@L_{\omega_1\omega}$-theory $\@T$, the syntactic Boolean $\omega_1$-pretopos $\-{\ang{\@L \mid \@T}}^B_{\omega_1}$ is the free Boolean $\omega_1$-pretopos containing a model $\@X$ of $\@T$: for any other Boolean $\omega_1$-pretopos $\!C$, we have an equivalence of groupoids
\begin{align*}
\&{\omega_1PTop}(\-{\ang{\@L \mid \@T}}^B_{\omega_1}, \!C) &--> \MOD_\!C(\@L, \@T) \\
F &|--> F_*(\@X).
\end{align*}
\end{proposition}
\begin{proof}
By \cref{thm:synptop-free}, \cref{thm:morley-catmod}, and the definition $\-{\ang{\@L \mid \@T}}^B_{\omega_1} = \-{\ang{\@L' \mid \@T'}}_{\omega_1}$ (for $(\@L', \@T')$ the Morleyization of $\@T$ in the fragment of all $\@L_{\omega_1\omega}$-formulas).
Alternatively, we may directly define the inverse as in the proof of \cref{thm:synptop-free} (i.e., \cite[D1.4.7]{Jeleph}).
\end{proof}

For an $\@L_{\omega_1\omega}$-theory $\@T$ and an $\@L'_{\omega_1\omega}$-theory $\@T'$, an \defn{($\@L'_{\omega_1\omega}$-)interpretation} $F : (\@L, \@T) -> (\@L', \@T')$ is an $\omega_1$-coherent functor $\-{\ang{\@L \mid \@T}}^B_{\omega_1} -> \-{\ang{\@L' \mid \@T'}}^B_{\omega_1}$, or equivalently a model of $\@T$ in $\-{\ang{\@L' \mid \@T'}}^B_{\omega_1}$; this may be spelled out explicitly as with $\omega_1$-coherent interpretations above.  Let
\begin{align*}
\&{\omega_1\omega Thy}_{\omega_1}
\end{align*}
denote the $2$-category of countable $\@L_{\omega_1\omega}$-theories, interpretations, and natural isomorphisms.  Thus $\&{\omega_1\omega Thy}_{\omega_1}$ is equivalent, via $\-{\ang{-}}^B_{\omega_1}$, to a full sub-$2$-category of $\&{B\omega_1PTop}$.

Given an interpretation $F : (\@L, \@T) -> (\@L', \@T')$, precomposition with $F$ yields a functor
\begin{align*}
F^* : \&{\omega_1PTop}(\-{\ang{\@L' \mid \@T'}}^B_{\omega_1}, \!C) --> \&{\omega_1PTop}(\-{\ang{\@L \mid \@T}}^B_{\omega_1}, \!C)
\end{align*}
for any Boolean $\omega_1$-pretopos $\!C$.  By \cref{thm:synbptop-free}, this is equivalently a functor
\begin{align*}
F^* : \MOD_\!C(\@L', \@T') --> \MOD_\!C(\@L, \@T);
\end{align*}
in other words, $F$ gives a uniform way of defining a model of $\@T$ from a model of $\@T'$.
However, this latter functor is canonically defined only up to isomorphism: given a model $\@M$ of $\@T'$ in $\!C$, turning it into an $\omega_1$-coherent functor $\-{\ang{\@L' \mid \@T'}}^B_{\omega_1} -> \!C$ (i.e., computing the inverse image of $\@M$ under \cref{thm:synbptop-free}) requires a choice of representatives for the disjoint unions and quotients involved in the interpretation of $\@T'$-imaginary sorts in $\@M$.

We now specialize to the case $\!C = \!{Count}$, so that $\MOD_\!C(\@L, \@T) = \MOD(\@L, \@T)$.  We will show that in this case, the functor $F^* : \MOD(\@L', \@T') -> \MOD(\@L, \@T)$ above can always be taken to be Borel; moreover, the assignment $F |-> F^*$ can be made (pseudo)functorial ``in a Borel way''.  This is conceptually straightforward, although the details (which involve coding functions) are quite messy.

We regard $\!{Count}$ as a standard Borel category by equipping its space of morphisms $\bigsqcup_{M, N \in \{0, 1, \dotsc, \#N\}} N^M$ with the obvious standard Borel structure.  Recall that $\!{Count}$ is also a Boolean $\omega_1$-pretopos.  The following lemma says that the Boolean $\omega_1$-pretopos operations may be taken to be Borel:

\begin{lemma}
\label{thm:count-pretop-borel}
There are Borel coding maps implementing the Boolean $\omega_1$-pretopos operations on $\!{Count}$, i.e., which
\begin{enumerate}
\item[(a)]  given two objects $M, N \in \!{Count}$, yields an object $P \in \!{Count}$ together with a bijection $(p, q) : P \cong M \times N$;
\item[(b)]  given two objects $M, N \in \!{Count}$ and two morphisms $f, g : M -> N$, yields an object $E \in \!{Count}$ together with an injection $h : E -> M$ whose image is the equalizer $\{m \in M \mid f(m) = g(m)\}$;
\item[(c)]  given three objects $L, M, N$ and morphisms $f : M -> L$ and $g : N -> L$, yields an object $P$ together with a bijection $(p, q) : P -> M \times_L N$;
\item[(d)]  given objects $(N_i)_{i \in \#N}$, yields an object $S$ together with injections $j_i : N_i -> S$ forming a bijection $\bigsqcup_i N_i \cong S$;
\item[(e)]  given objects $M, N$ and maps $p, q : M -> N$ such that $(p, q) : M -> N^2$ is injective with image an equivalence relation on $N$, yields an object $Q$ together with a surjection $r : N -> Q$ exhibiting $Q$ as the quotient;
\item[(f)]  given objects $M, N$ and a map $f : M -> N$, yields an object $I$, a surjection $g : M -> I$, and an injection $h : I -> N$, such that $f = h \circ g$;
\item[(g)]  given an object $M$, objects $(N_i)_{i \in \#N}$, and injections $f_i : N_i -> M$, yields an object $U$ and an injection $g : U -> M$ with image the union of the images of the $f_i$;
\item[(h)]  given objects $M, N$ and an injection $f : M -> N$, yields an object $C$ and an injection $g : C -> N$ whose image is the complement of that of $f$.
\end{enumerate}
\end{lemma}
\begin{proof}
All straightforward.
\end{proof}

We now prove that for $\!C = \!{Count}$, the functor in \cref{thm:synbptop-free} is an equivalence ``in a Borel way'' (which doesn't literally make sense, as $\&{\omega_1PTop}(\-{\ang{\@L \mid \@T}}^B_{\omega_1}, \!{Count})$ is not standard Borel):

\begin{lemma}
\label{thm:isogpd-retract}
Let $\@L$ be a countable relational language, $\@T$ be a countable $\@L_{\omega_1\omega}$-theory.  Let
\begin{align*}
L : \&{\omega_1PTop}(\-{\ang{\@L \mid \@T}}^B_{\omega_1}, \!{Count}) &--> \MOD(\@L, \@T)
\end{align*}
be the functor $F |-> F^*(\@X)$ from \cref{thm:synbptop-free}.  There is a functor
\begin{align*}
K : \MOD(\@L, \@T) &--> \&{\omega_1PTop}(\-{\ang{\@L \mid \@T}}^B_{\omega_1}, \!{Count})
\end{align*}
and a natural isomorphism
\begin{align*}
\zeta : K \circ L &--> 1_{\&{\omega_1PTop}(\-{\ang{\@L \mid \@T}}^B_{\omega_1}, \!{Count})}
\end{align*}
such that
\begin{itemize}
\item[(i)]  $L \circ K = 1$ (on the nose) and $\zeta_{K(\@M)} = 1_{K(\@M)}$, so that $K, L$ form an adjoint equivalence;
\item[(ii)]  for each imaginary sort $A \in \-{\ang{\@L \mid \@T}}^B_{\omega_1}$, the functor $K(-)(A) : \MOD(\@L, \@T) -> \!{Count}$ is Borel;
\item[(iii)]  for each definable function $f : A -> B$, the natural transformation $K(-)(f) : K(-)(A) -> K(-)(B)$ is Borel;
\item[(iv)]  for each imaginary sort $A \in \-{\ang{\@L \mid \@T}}^B_{\omega_1}$, there is a countable subcategory $\!S_A \subseteq \-{\ang{\@L \mid \@T}}^B_{\omega_1}$ such that the morphism $\zeta_{F,A} : K(L(F))(A) -> F(A)$ (in $\!{Count}$), as $F \in \&{\omega_1PTop}(\-{\ang{\@L \mid \@T}}^B_{\omega_1}, \!{Count})$ varies, depends only on $F|\!S_A$ and in a Borel way.
\end{itemize}
\end{lemma}
\begin{proof}
Recall from the proof of \cref{thm:synptop-free} (i.e., \cite[D1.4.7]{Jeleph}) that an inverse equivalence of $L$ is defined by sending a model $\@M \in \MOD(\@L, \@T)$ to the functor $\-{\ang{\@L \mid \@T}}^B_{\omega_1} -> \!{Count}$ which takes imaginary sorts and definable functions to their interpretations in $\@M$.  This is how we will define $K$, except that when interpreting, we use the Borel $\omega_1$-pretopos operations from \cref{thm:count-pretop-borel}.

Put $K(\@M)(\#X) := M$.  For each $n$, using \cref{thm:count-pretop-borel}(a) (repeatedly), let $K(\@M)(\#X^n)$ be an $n$th power (in $\!{Count}$) of $M$ equipped with projections to $M$ which depend in a Borel way on $\@M$.  For $n$-ary $R \in \@L$, let $K(\@M)(R) := \abs{R^\@M}$, equipped with a monomorphism $K(\@M)(R) `-> K(\@M)(\#X^n)$ which is Borel in $\@M$ and such that the image of $K(\@M)(R) `-> K(\@M)(\#X^n) \cong M^n$ is $R^\@M \subseteq M^n$ (for example, the unique order-preserving such monomorphism).

Next, for each $\@L_{\omega_1\omega}$-formula $\alpha$ with $n$ variables, define $K(\@M)(\alpha)$ equipped with a monomorphism $K(\@M)(\alpha) `-> K(\@M)(\#X^n)$ to be the interpretation $\alpha^\@M$ as in \cref{defn:struct-pretop}, but using \cref{thm:count-pretop-borel}(a--c,f--h) for the products, pullbacks, unions, etc., in that definition; then both $K(\@M)(\alpha)$ and the monomorphism to $K(\@M)(\#X^n)$ are Borel in $\@M$.  Similarly define $K(\@M)(A) := A^\@M$ using \cref{thm:count-pretop-borel}(d,e) for $A = (\bigsqcup_i \alpha_i)/(\bigsqcup_{i,j} \epsilon_{ij})$, equipped with maps $K(\@M)(\alpha_i) -> K(\@M)(A)$ which are Borel in $\@M$; and define $K(\@M)(f) := f^\@M : K(\@M)(A) -> K(\@M)(B)$, Borel in $\@M$, for a definable function $f : A -> B$.  For an isomorphism $g : \@M \cong \@N$ in $\MOD(\@L, \@T)$, $K(g)(A) : K(\@M)(A) \cong K(\@N)(A)$ is defined by induction on the structure of $A$ in the obvious way.  This completes the definition of $K$; (ii) and (iii) are immediate.

From the definition of $K(\@M)(R)$, we have $L(K(\@M)) = \@M$; likewise, $L(K(g)) = K(g)(\#X) = g$, which proves the first part of (i).

As in \cite[D1.4.7]{Jeleph}, we define $\zeta_{F,A} : K(L(F))(A) -> F(A)$ by induction on the structure of $A$ in the obvious way.  That is, $\zeta_{F,\#X} : K(L(F))(\#X) = F(\#X) -> F(\#X)$ is the identity; $\zeta_{F,\#X^n} : K(L(F))(\#X^n) = F(\#X)^n -> F(\#X^n)$ is the comparison isomorphism, using that $F$ preserves finite products; for $n$-ary $R \in \@L$, $\zeta_{F,R} : K(L(F))(R) -> F(R)$ is such that
\begin{equation*}
\begin{tikzcd}
K(L(F))(R) \dar[hook] \rar["\zeta_{F,R}"] & F(R) \dar[hook] \\
K(L(F))(\#X^n) = F(\#X)^n \rar["\zeta_{F,\#X^n}","\cong"'] & F(\#X^n)
\end{tikzcd}
\end{equation*}
commutes; $\zeta_{F,\alpha} : K(L(F))(\alpha) -> F(\alpha)$ is defined by induction on $\alpha$, using that $F$ preserves the $\omega_1$-pretopos operations used to interpret $\alpha$; and for $A = (\bigsqcup_i \alpha_i)/(\bigsqcup_{i,j} \epsilon_{ij})$, $\zeta_{F,A} : K(L(F))(A) = (\bigsqcup_i K(L(F))(\alpha_i))/(\bigsqcup_{i,j} K(L(F))(\epsilon_{ij})) -> F(A)$ is the comparison isomorphism, using that $F$ preserves countable $\bigsqcup$ and quotients.  When $F = K(\@M)$, it is easily verified by induction that $\zeta_{K(\@M),A}$ is the identity morphism for all $A$ (e.g., when $A = \#X^n$, the comparison $K(\@M)(\#X)^n -> K(\@M)(\#X^n)$ is the identity, since $K(\@M)(\#X^n)$ is by definition $K(\@M)(\#X)^n$); this proves (i).

Finally, for (iv), we take $\!S_A$ to consist of the various limits and colimits used to define the comparison isomorphisms in the definition $\zeta_{F,A}$.  That is, for $A = \#X$, we take $\!S_\#X = \{\#X\}$; for $A = \#X^n$, we take $\!S_{\#X^n}$ to be $\#X$, $\#X^n$, and the projections $\#X^n -> \#X$; for $n$-ary $R \in \@L$, we take $\!S_R$ to be $\!S_{\#X^n}$ together with the inclusion $R `-> \#X^n$; for $\alpha = \phi \wedge \psi$ with $n$ variables, we take $\!S_\alpha$ to be $\!S_\phi$, $\!S_\psi$ (which contain the inclusions $\phi `-> \#X^n$ and $\psi `-> \#X^n$) together with the inclusions $\alpha `-> \phi$ and $\alpha `-> \psi$; etc.  It is straightforward to check that this works.
\end{proof}

Recall that a \defn{pseudofunctor} $F : \&C -> \&D$ between two $2$-categories $\&C, \&D$ consists of an object $F(X) \in \&D$ for each object $X \in \&C$, a morphism $F(f) : F(X) -> F(Y)$ for each morphism $f : X -> Y$ in $\&C$, and a $2$-cell $F(\alpha) : F(f) -> F(g)$ for each $2$-cell $\alpha : f -> g$ in $\&C$, together with unnamed (but specified) isomorphisms $F(g) \circ F(f) \cong F(g \circ f)$ and $1_{F(X)} \cong F(1_X)$ for $f : X -> Y$ and $g : Y -> Z$ in $\&C$, which are required to obey certain coherence conditions.  A \defn{pseudonatural transformation} $\tau : F -> G$ between two pseudofunctors $F, G : \&C -> \&D$ consists of a morphism $\tau_X : F(X) -> G(X)$ for each object $X \in \&C$ and an invertible $2$-cell $\tau_f : G(f) \circ \tau_X -> \tau_Y \circ F(f)$ for each morphism $f : X -> Y$ in $\&C$, subject to certain coherence conditions.  A \defn{modification} $\Theta : \sigma -> \tau$ between two pseudonatural transformations $\sigma, \tau : F -> G$ consists of a $2$-cell $\Theta_X : \sigma_X -> \tau_X$ for each object $X \in \&C$, subject to certain conditions.  See \cite[1.1.2]{Jeleph} or \cite[II~7.5.1--3]{Bor} for details.

Let $\&{Gpd}$ denote the $2$-category of small groupoids, functors, and natural isomorphisms, and $\&{BorGpd}$ denote the $2$-category of standard Borel groupoids, Borel functors, and Borel natural isomorphisms; we have a forgetful $2$-functor $\&{BorGpd} -> \&{Gpd}$.  From above, we have a $2$-functor
\begin{align*}
\&{\omega_1PTop}(\-{\ang{-}}^B_{\omega_1}, \!{Count}) : \&{\omega_1\omega Thy}_{\omega_1}^\op &--> \&{Gpd} \\
(\@L, \@T) &|--> \&{\omega_1PTop}(\-{\ang{\@L \mid \@T}}^B_{\omega_1}, \!{Count}) \\
(F : (\@L, \@T) -> (\@L', \@T')) &|--> F^* : \&{\omega_1PTop}(\-{\ang{\@L' \mid \@T'}}^B_{\omega_1}, \!{Count}) -> \&{\omega_1PTop}(\-{\ang{\@L \mid \@T}}^B_{\omega_1}, \!{Count}).
\end{align*}

\begin{proposition}
\label{thm:isogpd-psfunc}
The assignment $(\@L, \@T) |-> \MOD(\@L, \@T)$ extends to a pseudofunctor
\begin{align*}
\MOD : \&{\omega_1\omega Thy}_{\omega_1}^\op --> \&{BorGpd}
\end{align*}
which is a ``lifting of $\&{\omega_1PTop}(\-{\ang{-}}^B_{\omega_1}, \!{Count})$'', in that there are pseudonatural transformations $K, L$ as in the following diagram (given componentwise by $K, L$ from \cref{thm:isogpd-retract}) which form an adjoint equivalence between $\&{\omega_1PTop}(\-{\ang{-}}^B_{\omega_1}, \!{Count})$ and the composite $\&{\omega_1\omega Thy}_{\omega_1}^\op --->{\MOD} \&{BorGpd} -> \&{Gpd}$.
\begin{equation*}
\begin{tikzcd}
&[5em] \&{BorGpd} \dar \\
\&{\omega_1\omega Thy}_{\omega_1}^\op \urar["\MOD",""{coordinate,pos=.65,name=us}] \rar["{\&{\omega_1PTop}(\-{\ang{-}}^B_{\omega_1}, \!{Count})}"',""{coordinate,pos=.65,name=ds}] & \&{Gpd}
\ar[draw=none,to path={(us)--(ds)\tikztonodes},"\Down K \;\; \Up L"{anchor=center}]
\end{tikzcd}
\end{equation*}
\end{proposition}
\begin{proof}
For two theories $(\@L, \@T), (\@L', \@T')$, we define $\MOD$ on the hom-category between them by
\begin{align*}
\&{\omega_1\omega Thy}_{\omega_1}((\@L, \@T), (\@L', \@T')) &--> \&{BorGpd}(\MOD(\@L', \@T'), \MOD(\@L, \@T)) \\
(F : (\@L, \@T) -> (\@L', \@T')) &|--> L_{\@T} \circ F^* \circ K_{\@T'},
\end{align*}
where $K, L$ are as in \cref{thm:isogpd-retract}.  To check that this lands in $\&{BorGpd}$: for a model $\@M \in \MOD(\@L', \@T')$, the model $\@N := L_{\@T}(F^*(K_{\@T'}(\@M))) = (F^*(K_{\@T'}(\@M)))_*(\@X) \in \MOD(\@L, \@T)$ has underlying set $N = F^*(K_{\@T'}(\@M))(\#X) = K_{\@T'}(\@M)(F(\#X))$, and for each $R \in \@L$, $R^\@N = F^*(K_{\@T'}(\@M))(R) = K_{\@T'}(\@M)(F(R))$; these are Borel in $\@M$ by \cref{thm:isogpd-retract}(ii).  Similarly, for an isomorphism of models $g : \@M \cong \@M'$ in $\MOD(\@L', \@T')$, the isomorphism $L_{\@T}(F^*(K_{\@T'}(g))) : \@N \cong \@N'$ (where $\@N' := L_{\@T}(F^*(K_{\@T'}(\@M')))$) is given by $K_{\@T'}(g)(\#X) : N \cong N'$, which is Borel in $g$ by \cref{thm:isogpd-retract}(ii).  Thus $L_{\@T} \circ F^* \circ K_{\@T'}$ is a Borel functor.  Furthermore, for a natural transformation between two interpretations $f : F -> G : (\@L, \@T) -> (\@L', \@T')$, the natural transformation $L_{\@T} \circ f^* \circ K_{\@T'} : L_{\@T} \circ F^* \circ K_{\@T'} -> L_{\@T} \circ G^* \circ K_{\@T'}$ is Borel: its component at $\@M \in \MOD(\@L', \@T')$ is $L_{\@T}(f^*_{K_{\@T'}(\@M)}) = K_{\@T'}(\@M)(f_\#X) : K_{\@T'}(\@M)(F(\#X)) -> K_{\@T'}(\@M)(G(\#X))$, which is Borel in $\@M$ by \cref{thm:isogpd-retract}(iii).  Thus the above definition of $\MOD$ on each hom-category lands in $\&{BorGpd}$.

The isomorphism $1_{\MOD(\@L, \@T)} \cong \MOD(1_{(\@L, \@T)}) = L_\@T \circ K_\@T$ is the identity, using \cref{thm:isogpd-retract}(i).  For interpretations $F : (\@L, \@T) -> (\@L', \@T')$ and $G : (\@L', \@T') -> (\@L'', \@T'')$, the isomorphism
\begin{align*}
\MOD(F) \circ \MOD(G) = L_\@T \circ F^* \circ K_{\@T'} \circ L_{\@T'} \circ G^* \circ K_{\@T''} \cong L_\@T \circ F^* \circ G^* \circ K_{\@T''} = \MOD(G \circ F)
\end{align*}
is given by $L_\@T \circ F^* \circ \zeta_{\@T'} \circ G^* \circ K_{\@T''}$, where $\zeta_{\@T'} : K_{\@T'} \circ L_{\@T'} -> 1$ is from \cref{thm:isogpd-retract}.  The coherence conditions are straightforward (the one corresponding to unitality uses the triangle identities, \cref{thm:isogpd-retract}(i)).  To complete the definition of $\MOD$, we need only verify that the natural isomorphism $\MOD(F) \circ \MOD(G) \cong \MOD(G \circ F)$ is Borel.  For a model $\@M \in \MOD(\@L'', \@T'')$, the component of the isomorphism at $\@M$ is $L_\@T(F^*(\zeta_{\@T',G^*(K_{\@T''}(\@M))})) = \zeta_{\@T',K_{\@T''}(\@M) \circ G,F(\#X)}$; letting $\!S_{F(\#X)} \subseteq \-{\ang{\@L' \mid \@T'}}^B_{\omega_1}$ be the countable subcategory given by \cref{thm:isogpd-retract}(iv), $\zeta_{\@T',K_{\@T''}(\@M) \circ G,F(\#X)}$ is Borel in $K_{\@T''}(\@M) \circ G|\!S_{F(\#X)}$, which is Borel in $\@M$ by \cref{thm:isogpd-retract}(ii,iii), as desired.

The components of $K, L$ on objects are given by \cref{thm:isogpd-retract}.  On a morphism $F : (\@L, \@T) -> (\@L', \@T')$, $K_F : F^* \circ K_{\@T'} -> K_\@T \circ \MOD(F) = K_\@T \circ L_\@T \circ F^* \circ K_{\@T'}$ is given by $\zeta_\@T^{-1} \circ F^* \circ K_{\@T'}$, while $L_F : \MOD(F) \circ L_{\@T'} = L_\@T \circ F^* \circ K_{\@T'} \circ L_{\@T'} -> L_\@T \circ F^*$ is given by $L_\@T \circ F^* \circ \zeta_{\@T'}$.  The coherence conditions are again straightforward (again using \cref{thm:isogpd-retract}(i) for units).  We have $L \circ K = 1$ by \cref{thm:isogpd-retract}(i) (both parts), as well as a modification $\zeta : K \circ L -> 1$ given componentwise by \cref{thm:isogpd-retract}; by \cref{thm:isogpd-retract}(i), these make $K, L$ into an adjoint equivalence.
\end{proof}

Henceforth we will denote $\MOD(F) : \MOD(\@L', \@T') -> \MOD(\@L, \@T)$ also by $F^*$, whenever there is no risk of confusion with $F^* : \&{\omega_1PTop}(\-{\ang{\@L' \mid \@T'}}^B_{\omega_1}, \!{Count}) -> \&{\omega_1PTop}(\-{\ang{\@L \mid \@T}}^B_{\omega_1}, \!{Count})$.

\section{Stone duality}
\label{sec:stone}

In this section, we explain how \cref{thm:borel-interp-equiv} may be viewed as one half of a Stone-type duality, yielding a ``strong conceptual completeness'' theorem for $\@L_{\omega_1\omega}$.  We then use this viewpoint to deduce the Borel version of the main result of \cite{HMM}.  This section depends on the previous section, and like it, involves some tedious $2$-categorical technicalities.

First, we briefly recall the abstract setup of the original Stone duality between the categories $\!{Bool}$ of Boolean algebras and $\!{KZHaus}$ of compact Hausdorff zero-dimensional spaces; our point of view here can be found in e.g., \cite[VI~\S4]{Jstone}.  We have a \defn{dualizing object}, the set $2 = \{0, 1\}$, which is both a Boolean algebra and a compact Hausdorff zero-dimensional space; and these two types of structure commute, meaning that the Boolean operations (e.g., $\wedge : 2 \times 2 -> 2$) are continuous.  As a consequence, for every other $A \in \!{Bool}$ and $X \in \!{KZHaus}$, the set of Boolean homomorphisms $\!{Bool}(A, 2)$ inherits the pointwise $\!{KZHaus}$-topology from $2$, and the set of continuous maps $\!{KZHaus}(X, 2)$ inherits the pointwise Boolean structure from $2$; and we have natural bijections
\begin{align*}
\!{Bool}(A, \!{KZHaus}(X, 2)) \cong (\!{Bool}, \!{KZHaus})(A \times X, 2) \cong \!{KZHaus}(X, \!{Bool}(A, 2))
\end{align*}
where $(\!{Bool}, \!{KZHaus})(A \times X, 2)$ denotes the set of  \defn{bihomomorphisms} $A \times X -> 2$, i.e., maps which are continuous for each fixed $a \in A$ and Boolean homomorphisms for each fixed $x \in X$.  This yields a contravariant adjunction between the functors
\begin{align*}
\!{Bool}(-, 2) : \!{Bool}^\op &--> \!{KZHaus}, &
\!{KZHaus}(-, 2) : \!{KZHaus}^\op &--> \!{Bool},
\end{align*}
whose adjunction units are the ``evaluation'' maps
\begin{align*}
A &--> \!{KZHaus}(\!{Bool}(A, 2), 2) &
X &--> \!{Bool}(\!{KZHaus}(X, 2), 2) \\
a &|--> (x |-> x(a)) &
x &|--> (a |-> a(x)).
\end{align*}
The Stone duality theorem states that these maps are isomorphisms, i.e., the adjunction is an adjoint equivalence $\!{Bool}^\op \cong \!{KZHaus}$.

In the case where $A = \ang{\@L \mid \@T}$ is the Lindenbaum--Tarski algebra of a finitary propositional theory $(\@L, \@T)$, $\!{Bool}(A, 2)$ is the space $\mathrm{Mod}(\@L, \@T)$ of models of $\@T$; and the half of Stone duality asserting that the unit at $A$ is an isomorphism is the completeness theorem for the theory $\@T$, plus the ``definability'' theorem that every clopen set of models is named by some formula.  The significance of the dualizing object $2$ is that the syntax of propositional logic (i.e., propositional formulas) is to be interpreted as elements of $2$.

In first-order logic, the syntax is assigned values of sets, functions, and relations; thus, the dualizing object for a first-order analog of Stone duality is naturally taken to be some variant of the category $\!{Set}$.  We listed several such (half-)duality theorems in the Introduction, notably those of Makkai \cite{Multra,Mscc} who introduced the term \defn{strong conceptual completeness} for this kind of logical interpretation of duality theorems.  Here, our goal is to interpret \cref{thm:borel-interp-equiv} as such a (half-)duality.  We take the dualizing object to be $\!{Count}$ (from \cref{sec:interp}, equivalent to the category of countable sets), equipped with the structure of a Boolean $\omega_1$-pretopos as well as that of a standard Borel groupoid (by forgetting the non-isomorphisms); \cref{thm:count-pretop-borel} can be seen as showing that these two kinds of structure commute.  However, there are some technical difficulties in directly copying the setup of Stone duality.

Let $(\@L, \@T)$ be an $\@L_{\omega_1\omega}$-theory and $\!G = (G^0, G^1)$ be a standard Borel groupoid.  We would like to say, on the basis of the two commuting structures on $\!{Count}$, that we have a $2$-adjunction
\begin{align*}
\text{``}\,\&{\omega_1PTop}(\-{\ang{\@L \mid \@T}}^B_{\omega_1}, \&{BorGpd}(\!G, \!{Count})) \cong \&{BorGpd}(\!G, \&{\omega_1PTop}(\-{\ang{\@L \mid \@T}}^B_{\omega_1}, \!{Count}))\,\text{''}
\end{align*}
as in Stone duality.  It is easily seen that $\&{BorGpd}(\!G, \!{Count})$ is a Boolean $\omega_1$-pretopos (with the pointwise operations from $\!{Count}$).  The problem is that the groupoid $\&{\omega_1PTop}(\-{\ang{\@L \mid \@T}}^B_{\omega_1}, \!{Count})$ is not standard Borel.  Instead, we must replace it with the equivalent standard Borel groupoid $\MOD(\@L, \@T)$, using \cref{thm:isogpd-psfunc}.

We proceed as follows.  Forgetting for now the Borel structure on $\!{Count}$, we have an isomorphism
\begin{align*}
\&{\omega_1PTop}(\-{\ang{\@L \mid \@T}}^B_{\omega_1}, \&{Gpd}(\!G, \!{Count}))
\cong (\&{\omega_1PTop}, \&{Gpd})(\-{\ang{\@L \mid \@T}}^B_{\omega_1} \times \!G, \!{Count})
\cong \&{Gpd}(\!G, \&{\omega_1PTop}(\-{\ang{\@L \mid \@T}}^B_{\omega_1}, \!{Count}))
\end{align*}
as in Stone duality, where the middle denotes the category of ``bihomomorphisms'' $\-{\ang{\@L \mid \@T}}^B_{\omega_1} \times \!G -> \!{Count}$, i.e., functors which are $\omega_1$-coherent for each fixed object $x \in G^0$; this isomorphism is clearly (strictly) natural in $(\@L, \@T)$ and $\!G$.  Composing this isomorphism with the postcomposition functors
\begin{align*}
K_* &: \&{Gpd}(\!G, \MOD(\@L, \@T)) --> \&{Gpd}(\!G, \&{\omega_1PTop}(\-{\ang{\@L \mid \@T}}^B_{\omega_1}, \!{Count})), \\
L_* &: \&{Gpd}(\!G, \&{\omega_1PTop}(\-{\ang{\@L \mid \@T}}^B_{\omega_1}, \!{Count})) --> \&{Gpd}(\!G, \MOD(\@L, \@T))
\end{align*}
induced by the functors $K, L$ from \cref{thm:isogpd-retract} yields an adjoint equivalence consisting of
\begin{align*}
&\begin{aligned}
\Phi : \&{Gpd}(\!G, \MOD(\@L, \@T)) &--> \&{\omega_1PTop}(\-{\ang{\@L \mid \@T}}^B_{\omega_1}, \&{Gpd}(\!G, \!{Count})) \\
F &|--> (A |-> (x |-> K(F(x))(A))),
\end{aligned} \\
&\begin{aligned}
\Psi : \&{\omega_1PTop}(\-{\ang{\@L \mid \@T}}^B_{\omega_1}, \&{Gpd}(\!G, \!{Count})) &--> \&{Gpd}(\!G, \MOD(\@L, \@T)) \\
G &|--> (x |-> L(G(-)(x))),
\end{aligned}
\end{align*}
such that $\Psi \circ \Phi = 1$, and a natural isomorphism $\xi : \Phi \circ \Psi -> 1$ satisfying the triangle identities (induced by $\zeta : K \circ L -> 1$ from \cref{thm:isogpd-retract}), given by
\begin{align*}
\xi_{G,A,x} = \zeta_{G(-)(x),A} : \Phi(\Psi(G))(A)(x) = K(L(G(-)(x)))(A) --> G(A)(x).
\end{align*}
Since $K, L$ are pseudonatural and $\zeta$ is a modification by \cref{thm:isogpd-psfunc}, this adjoint equivalence remains natural in $\!G$ and pseudonatural $(\@L, \@T)$.

\begin{lemma}
\label{thm:borel-duality-adjoint-cocycle}
$\Phi, \Psi, \xi$ restrict to a pseudonatural adjoint equivalence
\begin{align*}
\&{\omega_1PTop}(\-{\ang{\@L \mid \@T}}^B_{\omega_1}, \&{BorGpd}(\!G, \!{Count})) \cong \&{BorGpd}(\!G, \MOD(\@L, \@T)).
\end{align*}
\end{lemma}
\begin{proof}
First, fix $\!G, \@T$.  To check that $\Psi$ restricts, let $G, G' \in \&{\omega_1PTop}(\-{\ang{\@L \mid \@T}}^B_{\omega_1}, \&{BorGpd}(\!G, \!{Count}))$ and $\gamma : G -> G'$ with $\gamma_A$ Borel for each $A$; we must check that $\Psi(G), \Psi(G'), \Psi(\gamma)$ are Borel.  $\Psi(G)(x) = L(G(-)(x))$ is the model with underlying set $G(\#X)(x)$ and with $R^{\Psi(G)(x)} = $ the image of $G(R)(x) `-> G(\#X^n)(x)$ for $n$-ary $R \in \@L$; thus $\Psi(G)(x)$ is Borel in $x$.  For a morphism $g : x -> y$ in $\!G$, we have $\Psi(G)(x) = G(\#X)(g)$, which is Borel in $g$.  Thus, $\Psi(G)$ is Borel; similarly, $\Psi(G')$ is Borel.  And $\Psi(\gamma)(x) = \gamma_{\#X,x}$; so $\Psi(\gamma)$ is Borel.

To check that $\Phi$ restricts, let $F, F' \in \&{BorGpd}(\!G, \MOD(\@L, \@T))$ and $\phi : F -> F'$ be Borel; we must check that $\Phi(F), \Phi(F'), \Phi(\phi)$ are pointwise Borel.  For $A \in \-{\ang{\@L \mid \@T}}^B_{\omega_1}$, we have $\Phi(F)(A)(x) = K(F(x))(A)$ which is Borel in $x$ (and similarly when $x$ is replaced by $g : x -> y$) by \cref{thm:isogpd-retract}(ii) and Borelness of $F$; similarly, for a definable function $f : A -> B$, $\Phi(F)(f)(x) = K(F(x))(f)$ is Borel in $x$ by \cref{thm:isogpd-retract}(iii).  Similarly, $\Phi(\phi)(A)(x) = K(\phi_x)(A)$ which is Borel in $x$ by \cref{thm:isogpd-retract}(ii) and Borelness of $\phi$.

To check that $\xi$ restricts, let $G \in \&{\omega_1PTop}(\-{\ang{\@L \mid \@T}}^B_{\omega_1}, \&{BorGpd}(\!G, \!{Count}))$; then $\xi_{G,A,x} = \zeta_{G(-)(x),A}$ is Borel in $G(-)(x)|\!S_A$ which is Borel in $x$, where $\!S_A \subseteq \-{\ang{\@L \mid \@T}}^B_{\omega_1}$ is given by \cref{thm:isogpd-retract}(iv).

Finally, we must check that the pseudonaturality isomorphisms for $\Phi, \Psi$ as $(\@L, \@T)$ varies are Borel (there is nothing to check as $\!G$ varies, since $\Phi, \Psi$ are natural in $\!G$).  Let $H : (\@L, \@T) -> (\@L', \@T')$ be an interpretation.  From the proof of \cref{thm:isogpd-psfunc}, it is easily seen that the pseudonaturality isomorphism $\Psi_H : \MOD(H)_* \circ \Psi_{\@T'} -> \Psi_\@T \circ H^*$ (induced by that for $L$) is given by
\begin{align*}
\Psi_{H,G,x} = \zeta_{\@T',G(-)(x),H(\#X)} : K_{\@T'}(L_{\@T'}(G(-)(x)))(H(\#X)) --> G(H(\#X))(x)
\end{align*}
which is Borel in $x$ using \cref{thm:isogpd-retract}(iv) as above, while the pseudonaturality isomorphism $\Phi_H : H^* \circ \Phi_{\@T'} -> \Phi_\@T \circ \MOD(H)_*$ (induced by that for $K$) is given by
\begin{align*}
\Phi_{H,F,A,x} = \zeta_{\@T,K_{\@T'}(F(x)) \circ H,A}^{-1} : K_{\@T'}(F(x))(H(A)) --> K_\@T(L_\@T(K_{\@T'}(F(x)) \circ H))(A)
\end{align*}
which is Borel in $x$ using \cref{thm:isogpd-retract}(ii--iv).
\end{proof}

Thus, in place of the adjunction in Stone duality, we have a ``relative pseudoadjunction'' between the pseudofunctors
\begin{align*}
\MOD : \&{\omega_1\omega Thy}_{\omega_1}^\op &--> \&{BorGpd}, &
\&{BorGpd}(-, \!{Count}) : \&{BorGpd}^\op &--> \&{B\omega_1PTop};
\end{align*}
``relative'' means that $\MOD$ is not defined on all of $\&{B\omega_1PTop}$, but only on $\&{\omega_1\omega Thy}_{\omega_1}$ (which, recall, is equivalent to a full sub-$2$-category of the former).  See \cite{LMV} for basic facts on pseudoadjunctions (also called biadjunctions), and \cite{Ulmer} for relative adjunctions.  We still have one adjunction unit, namely the transpose across \cref{thm:borel-duality-adjoint-cocycle} of the identity $\MOD(\@L, \@T) -> \MOD(\@L, \@T)$:
\begin{align*}
\eta_\@T = \Phi(1_{\MOD(\@L, \@T)}) : \-{\ang{\@L \mid \@T}}^B_{\omega_1} &--> \&{BorGpd}(\MOD(\@L, \@T), \!{Count}) \\
A &|--> K_\@T(-)(A).
\end{align*}
We next verify that this unit functor is none other than the interpretation functor $\den{-}$ from \cref{thm:borel-interp-equiv}, under the following standard identification.

A Borel functor $F : \!G -> \!{Count}$ determines a fiberwise countable Borel $\!G$-space, namely
\begin{align*}
\Sigma(F) := \{(x, a) \mid x \in G^0 \AND a \in F(x)\} \subseteq G^0 \times \#N,
\end{align*}
with the projection $p : \Sigma(F) -> G^0$ and the obvious action of $\!G$: $g \cdot (x, a) := (y, F(g)(a))$ for $g : x -> y$.  Given another Borel functor $G : \!G -> \!{Count}$ and a Borel natural transformation $f : F -> G$, we have the $\!G$-equivariant map $\Sigma(f) : \Sigma(F) -> \Sigma(G)$ given fiberwise by the components of $f$.  Thus, we have a functor
\begin{align*}
\Sigma = \Sigma_\!G : \&{BorGpd}(\!G, \!{Count}) --> \Act^B_{\omega_1}(\!G).
\end{align*}
Note that $\Act^B_{\omega_1}(-)$ is contravariantly pseudofunctorial: given a Borel functor $F : \!G -> \!H$, we may pull back $\!H$-spaces $p : X -> H^0$ along $F$ to obtain $\!G$-spaces $F^*(X) = G^0 \times_{H^0} X$ (temporarily denote this by $G^0 \times_{H^0}^F X$);
and given a Borel natural isomorphism $\theta : F \cong F' : \!G -> \!H$ and a fiberwise countable Borel $\!H$-space $p : X -> H^0$, we have a Borel $\!G$-equivariant isomorphism
\begin{align*}
\theta^*(X) : G^0 \times_{H^0}^F X &\cong G^0 \times_{H^0}^{F'} X \\
(y, x) &|-> (y, \theta_y \cdot x).
\end{align*}

\begin{lemma}
\label{thm:borel-action-cocycle}
$\Sigma : \&{BorGpd}(-, \!{Count}) -> \Act^B_{\omega_1}(-)$ is a pseudonatural equivalence between pseudofunctors $\&{BorGpd}^\op -> \&{B\omega_1PTop}$.
\end{lemma}
\begin{proof}
First, we check that for fixed $\!G$, $\Sigma_\!G$ is an equivalence.  Faithfulness and fullness are clear from the definition of $\Sigma_\!G(f)$.  For essential surjectivity, given an arbitrary fiberwise countable Borel $\!G$-space $p : X -> G^0$, we may use Lusin--Novikov to enumerate each fiber $p^{-1}(x)$ in a Borel way, yielding bijections $e_x : p^{-1}(x) \cong \abs{p^{-1}(x)}$ such that $e_x$ and $\abs{p^{-1}(x)}$ are Borel in $x$; then defining $F : \!G -> \!{Count}$ by $F(x) := \abs{p^{-1}(x)}$ and $F(g) := e_y \circ g \circ e_x^{-1}$ for $g : x -> y$ in $\!G$, we clearly have $X \cong \Sigma(F)$.  (This was noted at the end of \cref{sec:borel-action-etale}.)

For a Borel functor $F : \!G -> \!H$, the isomorphism $\Sigma_F : F^* \circ \Sigma_\!H \cong \Sigma_\!G \circ F^*$ is given by
\begin{align*}
\Sigma_{F,G} : F^*(\Sigma_\!H(G)) = G^0 \times_{H^0}^F \Sigma_\!H(G) &--> \Sigma_\!G(G \circ F) = \Sigma_\!G(F^*(G)) \\
(x, (y, a)) &|--> (x, a)
\end{align*}
for $G : \!H -> \!{Count}$.  It is straightforward to check that this works.
\end{proof}

\begin{lemma}
\label{thm:borel-duality-unit-cocycle}
For a countable $\@L_{\omega_1\omega}$-theory $(\@L, \@T)$, we have $\Sigma_{\MOD(\@L, \@T)} \circ \eta_\@T \cong \den{-} : \-{\ang{\@L \mid \@T}}^B_{\omega_1} -> \Act^B_{\omega_1}(\MOD(\@L, \@T))$.
\end{lemma}
\begin{proof}
By \cref{thm:borel-action-cocycle} and the definition of $\eta_\@T$, $\Sigma_{\MOD(\@L, \@T)} \circ \eta_\@T$ is an $\omega_1$-coherent functor; from the definition of $\den{-}$, it is easy to check that $\den{-}$ is also an $\omega_1$-coherent functor.  Thus, by \cref{thm:synbptop-free}, it suffices to check that the models of $\@T$ in $\Act^B_{\omega_1}(\MOD(\@L, \@T))$ corresponding to $\Sigma_{\MOD(\@L, \@T)} \circ \eta_\@T$ and $\den{-}$ are isomorphic.  The former has underlying object
\begin{align*}
\Sigma_{\MOD(\@L, \@T)}(\eta_\@T(\#X)) = \Sigma_{\MOD(\@L, \@T)}(K_\@T(-)(\#X)) = \{(\@M, a) \mid \@M \in \Mod(\@L, \@T) \AND a \in M\} = \den{\#X};
\end{align*}
and the interpretation of $n$-ary $R \in \@L$ is (the image of) $\Sigma_{\MOD(\@L, \@T)}(\eta_\@T(R)) `-> \Sigma_{\MOD(\@L, \@T)}(\eta_\@T(\#X^n)) \cong \Sigma_{\MOD(\@L, \@T)}(\eta_\@T(\#X))^n$, which from the definition of $K_\@T$ (\cref{thm:isogpd-retract}) is easily seen to be $\den{R} \subseteq \den{\#X^n}$.  Thus, the two models are the same.
\end{proof}

Combining \cref{thm:borel-duality-adjoint-cocycle,thm:borel-action-cocycle,thm:borel-duality-unit-cocycle,thm:borel-interp-equiv}, we get

\begin{proposition}
\label{thm:borel-duality}
We have a contravariant relative pseudoadjunction
\begin{align*}
\&{\omega_1PTop}(\-{\ang{\@L \mid \@T}}^B_{\omega_1}, \Act^B_{\omega_1}(\!G)) \cong \&{BorGpd}(\!G, \MOD(\@L, \@T))
\end{align*}
between the pseudofunctors
\begin{align*}
\MOD : \&{\omega_1\omega Thy}_{\omega_1}^\op &--> \&{BorGpd}, &
\Act^B_{\omega_1} : \&{BorGpd}^\op &--> \&{B\omega_1PTop};
\end{align*}
and the adjunction unit $\-{\ang{\@L \mid \@T}}^B_{\omega_1} -> \Act^B_{\omega_1}(\MOD(\@L, \@T))$ is (isomorphic to $\den{-}$, and hence) an equivalence for every countable $\@L_{\omega_1\omega}$-theory $(\@L, \@T)$.  \qed
\end{proposition}

This is our promised interpretation of \cref{thm:borel-interp-equiv} as a half-duality.  One consequence is the following reformulation of \cref{thm:borel-interp-equiv}, which contains \cref{thm:2interp-eso}:

\begin{corollary}
\label{thm:2interp-equiv}
For any two theories $(\@L, \@T), (\@L', \@T')$, the functor
\begin{align*}
\MOD_{\@T,\@T'} : \&{\omega_1\omega Thy}_{\omega_1}((\@L, \@T), (\@L', \@T') &--> \&{BorGpd}(\MOD(\@L', \@T'), \MOD(\@L, \@T)) \\
(F : (\@L, \@T) -> (\@L', \@T')) &|--> \MOD(F) = F^* : \MOD(\@L', \@T') -> \MOD(\@L, \@T)
\end{align*}
is an equivalence of groupoids.
\end{corollary}
\begin{proof}
This follows from a version (for relative pseudoadjunctions) of the standard fact that a left adjoint is full and faithful iff the unit is a natural isomorphism.  See e.g., \cite[1.3]{LMV}.

(In more detail, $\MOD_{\@T,\@T'}$ is easily seen to be isomorphic to the composite
\begin{align*}
\&{\omega_1PTop}(\-{\ang{\@L \mid \@T}}^B_{\omega_1}, \-{\ang{\@L' \mid \@T'}}^B_{\omega_1})
\cong \&{\omega_1PTop}(\-{\ang{\@L \mid \@T}}^B_{\omega_1}, \Act^B_{\omega_1}(\MOD(\@L', \@T')))
\cong \&{BorGpd}(\MOD(\@L', \@T'), \MOD(\@L, \@T))
\end{align*}
where the second equivalence is given by \cref{thm:borel-duality} and the first equivalence is induced by the adjunction unit (i.e., $\den{-}$) for $\@T'$.)
\end{proof}

Finally in this section, we explain how the Borel version of the main result of Harrison-Trainor--Miller--Montalbán \cite[Theorem~9]{HMM} can be viewed as a special case of \cref{thm:2interp-eso}.

Recall (see e.g., \cite[\S12.1]{Gao}) that for a countable $\@L$-structure $\@M$, the \defn{Scott sentence} of $\@M$ is an $\@L_{\omega_1\omega}$-sentence $\sigma_\@M$ whose countable models are precisely the isomorphic copies of $\@M$.

Let $\@M \in \MOD(\@L)$ and $\@N \in \MOD(\@L')$ be countable structures (on initial segments of $\#N$) in possibly different languages $\@L, \@L'$.  According to \cite{HMM}, an \defn{interpretation} $\@I$ of $\@M$ in $\@N$ consists of:
\begin{enumerate}
\item[(i)]  a subset $\@Dom_\@M^\@N \subseteq N^{<\omega}$, \defn{definable (without parameters)} in $\@N$, i.e., for each $n$ we have $N^n \cap \@Dom_\@M^\@N = \phi^\@N$ for some $\@L'_{\omega_1\omega}$-formula with $n$ variables;
\item[(ii)]  a definable equivalence relation $\sim$ on $\@Dom_\@M^\@N$;
\item[(iii)]  for each $n$-ary $R \in \@L$, a $\sim$-invariant definable subset $R^\@I \subseteq (\@Dom_\@M^\@N)^n$;
\item[(iv)]  an isomorphism of $\@L$-structures $g_\@I : (\@Dom_\@M^\@N/{\sim}, R^\@I/{\sim})_{R \in \@L} \cong \@M$.
\end{enumerate}

We may rephrase this in our terminology as follows.  By completeness and the defining property of $\sigma_\@N$, a definable subset $S \subseteq N^n$ is defined by a unique $\@L'_{\omega_1\omega}$-formula modulo $\sigma_\@N$-equivalence.  Thus, definable subsets of $N^n$ are in bijection with subobjects (i.e., subsorts) of $\#X^n$ in $\-{\ang{\@L' \mid \sigma_\@N}}^B_{\omega_1}$, and similarly for definable subsets of $N^{<\omega}$, etc.  Furthermore, the conditions on the definable sets $\sim$ and $R^\@I$ imposed by (ii--iv) above are equivalent to the corresponding syntactic conditions on the defining formulas being $\sigma_\@N$-provable.  Using this, it is easily seen that an interpretation $\@I$ of $\@M$ in $\@N$ is equivalently given by
\begin{enumerate}
\item[(i$'$)]  a subsort $D_\@I \subseteq \bigsqcup_{n \in \#N} \#X^n$ in $\-{\ang{\@L' \mid \sigma_\@N}}^B_{\omega_1}$;
\item[(ii$'$)]  an equivalence relation $E_\@I$ on $D_\@I$;
\item[(iii$'$)]  an interpretation $F_\@I : (\@L, \sigma_\@M) -> (\@L', \sigma_\@N)$ (in our sense), i.e., model of $\sigma_\@M$ in $\-{\ang{\@L' \mid \sigma_\@N}}^B_{\omega_1}$, with underlying object $F_\@I(\#X) = D_\@I/E_\@I$;
\item[(iv$'$)]  an isomorphism $g_\@I : F_\@I^*(\@N) \cong \@M$ in $\MOD(\@L, \sigma_\@M)$.
\end{enumerate}

\begin{remark}
Note that $D_\@I, E_\@I, g_\@I$ are in some sense irrelevant.  Indeed, for any imaginary sort $A = (\bigsqcup_i \alpha_i)/(\bigsqcup_{i,j} \epsilon_{ij}) \in \-{\ang{\@L \mid \@T}}^B_{\omega_1}$ (in any theory $(\@L, \@T)$), where $\alpha_i$ has $n_i$ variables, we may express $A$ as a quotient of a subsort of $\bigsqcup_{n \in \#N} \#X^n$, by picking $n_0' < n_1' < \dotsb$ with $n_i' \ge n_i$ and then embedding $\alpha_i \subseteq \#X^{n_i}$ into $\#X^{n_i'}$ diagonally, so that $\bigsqcup_i \alpha_i \subseteq \bigsqcup_i \#X^{n_i'} \subseteq \bigsqcup_n \#X^n$.  Thus, given $F_\@I : (\@L, \sigma_\@M) -> (\@L', \sigma_\@N)$, we can always find $D_\@I$ and $E_\@I$ as in (i$'$,ii$'$) such that $F_\@I(\#X) \cong D_\@I/E_\@I$.  And since all countable models of $\sigma_\@M$ are isomorphic to $\@M$, we can always find $g_\@I$ as in (iv$'$).
\end{remark}

Given an interpretation $\@I$ of $\@M$ in $\@N$, \cite{HMM} defines the induced functor $\MOD(\@L', \sigma_\@N) -> \MOD(\@L, \sigma_\@M)$ to take an isomorphic copy of $\@N$ to the isomorphic copy of $\@M$ given by $\@I$, with domain replaced by (an initial segment of) $\#N$ via some canonical coding of quotients of subsets of $\#N^{<\omega}$.  This is also how we defined $F_\@I^* : \MOD(\@L', \sigma_\@N) -> \MOD(\@L, \sigma_\@M)$ in \cref{thm:isogpd-psfunc}, with the coding given by \cref{thm:count-pretop-borel}.  Note that in accordance with the above remark, $D_\@I, E_\@I, g_\@I$ are not used here.

The Borel version of \cite[Theorem~9]{HMM} states that every Borel functor $\MOD(\@L', \sigma_\@N) -> \MOD(\@L, \sigma_\@M)$ is induced by some interpretation of $\@M$ in $\@N$.  By the above, this is equivalent to \cref{thm:2interp-eso} in the case where $\@T, \@T'$ are both (equivalent to) Scott sentences, i.e., when they both have a single countable model up to isomorphism.


\section{$\kappa$-coherent frames and locales}
\label{sec:kloc}

In the rest of this paper, we sketch a proof of a generalization of \cref{thm:etale-coherent-interp-equiv} (itself a generalization of \cref{thm:etale-interp-equiv}) using the Joyal--Tierney representation theorem for Grothendieck toposes.  In this section, we review some concepts from locale theory; see \cite{Jstone}, \cite[C1]{Jeleph}, or \cite{JT}.

In this and the following sections, let $\kappa$ be a regular \emph{uncountable} cardinal or the symbol $\infty$ (bigger than all cardinals).  By \defn{$\kappa$-ary} we mean of size less than $\kappa$.

A \defn{frame} is a poset with finite meets and arbitrary joins, the former distributing over the latter.  A \defn{locale} $X$ is the same thing as a frame $\@O(X)$, except that we think of $X$ as a generalized topological space whose \defn{frame of opens} is $\@O(X)$.  A \defn{continuous map} or \defn{locale morphism} $f : X -> Y$ between locales is a frame homomorphism $f^* : \@O(Y) -> \@O(X)$.  Thus, the category $\!{Loc}$ of locales is the opposite of the category $\!{Frm}$ of frames.  A topological space $X$ is regarded as the locale with $\@O(X) = \{\text{open sets in $X$}\}$; thus we have a forgetful functor $\!{Top} -> \!{Loc}$.  A locale $X$ is \defn{spatial} if it is isomorphic to a topological space.  The \defn{spatialization} $\Sp(X)$ of a locale $X$ is the space of all locale morphisms $1 -> X$ or \defn{points} (where $\@O(1) = \{0 < 1\}$), with topology consisting of
\begin{align*}
[U] := \{x \in \Sp(X) \mid x^*(U) = 1\}
\end{align*}
for $U \in \@O(X)$; we have a canonical locale morphism $\epsilon : \Sp(X) -> X$ given by $\epsilon^* = [-] : \@O(X) -> \@O(\Sp(X))$, which is an isomorphism iff $X$ is spatial.


A \defn{sublocale} $Y \subseteq X$ is given by a quotient frame $\@O(X) ->> \@O(Y)$.  The \defn{image sublocale} of a locale morphism $f : X -> Y$ is given by the image of the corresponding frame homomorphism $f^*$.  Every open $U \in \@O(X)$ gives rise to the \defn{open sublocale} $U \subseteq X$ corresponding to the quotient frame $U \wedge (-) : \@O(X) ->> \down U =: \@O(U)$.  A locale morphism $f : X -> Y$ is \defn{open} if the image of every open $U \subseteq X$ is an open sublocale $f_+(U) \subseteq Y$, and \defn{étalé} if $X$ (i.e., the top element $\top \in \@O(X)$) is the union of open sublocales $U \subseteq X$ to which the restriction $f|U : U -> Y$ is an open sublocale inclusion (with image $f_+(U)$; we call such $U$ an \defn{open section over $f_+(U)$}).

A \defn{sheaf} on a frame $L$ is a functor $L^\op -> \!{Set}$ preserving limits (i.e., colimits in $L$) of the form $\bigvee A$ where $A \subseteq L$ is downward-closed; the category of sheaves on $L$ is denoted $\Sh(L)$.  A \defn{sheaf} on a locale $X$ is a sheaf on $\@O(X)$; we also put $\Sh(X) := \Sh(\@O(X))$.  By standard sheaf theory (see \cite[C1.3]{Jeleph}), a sheaf on $X$ is equivalently given by an étalé locale over $X$.

A \defn{$\kappa$-frame} is a poset with finite meets and $\kappa$-ary joins, the former distributing over the latter.  For a $\kappa$-frame $K$, the \defn{$\kappa$-ideal completion} $\Idl_\kappa(K)$ is a frame; a frame $L$ (or locale $X$) is \defn{$\kappa$-coherent} if $L$ ($\@O(X)$) is isomorphic to $\Idl_\kappa(K)$ for some $\kappa$-frame $K$.  The \defn{$\kappa$-compact} elements of $L$, denoted $L_\kappa \subseteq L$, are those $u \in L$ such that whenever $u \le \bigvee_i v_i$ then $u \le \bigvee_j v_{i_j}$ for a $\kappa$-ary subfamily $\{v_{i_j}\}_j$; for a $\kappa$-frame $K$ we have $\Idl_\kappa(K)_\kappa \cong K$, hence for a $\kappa$-coherent frame $L$ we have $L \cong \Idl_\kappa(L_\kappa)$.  For a locale $X$, put $\@O_\kappa(X) := \@O(X)_\kappa$.  A locale morphism $f : X -> Y$ between $\kappa$-coherent $X, Y$ is \defn{$\kappa$-coherent} if $f^*(\@O_\kappa(Y)) \subseteq \@O_\kappa(X)$, i.e., $f^*$ is the join-preserving extension of a $\kappa$-frame homomorphism $\@O_\kappa(Y) -> \@O_\kappa(X)$.  Thus, the category $\!{\kappa Loc}$ of $\kappa$-coherent locales and $\kappa$-coherent locale morphisms is equivalent to the opposite of the category $\!{\kappa Frm}$ of $\kappa$-frames.  Because of this, we also refer to a $\kappa$-coherent locale (morphism) simply as a \defn{$\kappa$-locale (morphism)}.

Let $X, Y$ be $\kappa$-locales.  A locale morphism $f : X -> Y$ is \defn{$\kappa$-étalé} if $X$ is a $\kappa$-ary union of $\kappa$-compact open sections.  It is easy to see that a $\kappa$-étalé morphism is automatically $\kappa$-coherent.  Note that when $\kappa = \omega_1$, the notion of $\omega_1$-étalé locale morphism is not quite analogous to the notion of countable étalé map from \cref{sec:groupoid}: the present notion requires the open sections to be $\omega_1$-compact.  (By \cref{thm:qpol-sloc} below, the two notions agree when restricted to quasi-Polish spaces.)  Nonetheless, we have analogs of the basic properties in \cref{thm:etale} (with ``$\kappa$-étalé'' in place of ``countable étalé''), proved in exactly the same way.  For a $\kappa$-locale $Y$, we write $\Sh_\kappa(Y) \subseteq \Sh(Y)$ for the subcategory of $\kappa$-étalé locales over $Y$ (identified with sheaves).

A ($\kappa$-)frame $L$ is \defn{$\kappa$-presented} if it has a $\kappa$-ary presentation, i.e., there are $<\kappa$-many $u_i \in L$ and $<\kappa$-many equations between ($\kappa$-)frame terms involving the $u_i$ such that $L$ is the free ($\kappa$-)frame generated by the $u_i$ subject to these equations.  Let $\!{Frm}_\kappa \subseteq \!{Frm}$ (resp., $\!{\kappa Frm}_\kappa \subseteq \!{\kappa Frm}$) denote the full subcategory of $\kappa$-presented ($\kappa$-)frames.  The following is straightforward:

\begin{lemma}
\label{thm:kfrmk}
For a $\kappa$-presented $\kappa$-frame $K$, the principal ideal embedding $\down : K -> \Idl_\kappa(K)$ is an isomorphism; and $\Idl_\kappa : \!{\kappa Frm} -> \!{Frm}$ restricts to an equivalence of categories $\!{\kappa Frm}_\kappa \cong \!{Frm}_\kappa$.  \qed
\end{lemma}

We call a locale $X$ \defn{$\kappa$-copresented} if $\@O(X)$ is $\kappa$-presented as a frame, or equivalently as a $\kappa$-frame.  By \cref{thm:kfrmk}, a $\kappa$-copresented locale is $\kappa$-coherent, with $\@O_\kappa(X) = \@O(X)$.

In the next lemma (a generalization of \cref{thm:etale-qpol}, by \cref{thm:qpol-sloc}), we adopt the point of view featured prominently in \cite{JT}, where a ($\kappa$-)frame is viewed as analogous to a commutative ring; thus, a $\kappa$-frame homomorphism $f : K -> L$ exhibits $L$ as a ``$K$-algebra''.

\begin{lemma}
\label{thm:etale-kloc}
Let $f : X -> Y$ be a $\kappa$-étalé locale morphism, and let $\@U \subseteq \@O(X)$ be $<\kappa$-many $\kappa$-compact open sections covering $X$ and closed under binary meets.  Then $f^* : \@O_\kappa(Y) -> \@O_\kappa(X)$ exhibits $\@O_\kappa(X)$ as the $\@O_\kappa(Y)$-algebra presented by the generators $U \in \@U$ and the relations
\begin{align*}
\begin{aligned}
U &= f^*(f_+(U)) \wedge V &&\text{for $U \le V \in \@U$}, \\
U \wedge V &= W &&\text{for $U, V \in \@U$ and $W = U \wedge V$}, \\
\top &= \bigvee_{U \in \@U} U.
\end{aligned}
\end{align*}
Thus, if $Y$ is $\kappa$-copresented, then so is $X$.
\end{lemma}
\begin{proof}
Given another $\@O_\kappa(Y)$-algebra (i.e., $\kappa$-frame homomorphism) $g : \@O_\kappa(Y) -> L$, and a map $h : \@U -> L$ such that the above relations (with $f^*$ replaced by $g$) hold after applying $h$ to $\@U$, the unique $\kappa$-frame homomorphism $h' : \@O_\kappa(X) -> L$ extending $h$ such that $h' \circ f^* = g$ is given by
\begin{align*}
h'(W) := \bigvee_{U \in \@U} (g(f_+(U \wedge W)) \wedge h(U)).
\end{align*}
It is straightforward to check that this works.
\end{proof}

A basic intuition regarding ($\kappa$-)locales is that they are (quasi-)Polish spaces generalized by removing countability requirements.  This is made precise by the following.  It can be found in \cite{Hec}, who states that similar results have been proved before by various authors.

\begin{proposition}
\label{thm:qpol-sloc}
The forgetful functor $\!{Top} -> \!{Loc}$ restricts to an equivalence $\!{QPol} -> \!{Loc}_{\omega_1}$ between the category of quasi-Polish spaces and the category of $\omega_1$-copresented locales.
\end{proposition}

\section{Locally $\kappa$-presentable categories}
\label{sec:lkpcat}

This section collects some basic facts we will need on locally $\kappa$-presentable categories; see \cite{ARlpac}.

Let $\!C$ be a category with (small) $\kappa$-filtered colimits.  An object $X \in \!C$ is \defn{$\kappa$-presentable} if the representable functor $\!C(X, -) : \!C -> \!{Set}$ preserves $\kappa$-filtered colimits.  Let $\!C_\kappa \subseteq \!C$ denote the full subcategory of $\kappa$-presentable objects.  $\!C$ is \defn{$\kappa$-accessible} if $\!C_\kappa$ is essentially small and generates $\!C$ under $\kappa$-filtered colimits, and \defn{locally $\kappa$-presentable} if it is furthermore cocomplete (equivalently, $\!C_\kappa$ has $\kappa$-ary colimits).

An arbitrary category $\!C$ has a \defn{$\kappa$-ind-completion} $\Ind_\kappa(\!C)$, which is the free cocompletion of $\!C$ under (small) $\kappa$-filtered colimits; see \cite[2.26]{ARlpac}.  When $\!C$ is small and has $\kappa$-ary colimits, $\Ind_\kappa(\!C)$ can be constructed as the full subcategory of $\!{Set}^{\!C^\op}$ on the functors preserving $\kappa$-ary limits.  For a small category $\!K$, $\Ind_\kappa(\!K)$ is $\kappa$-accessible, with $\Ind_\kappa(\!K)_\kappa$ equivalent to the Cauchy completion of $\!K$; conversely, for a $\kappa$-accessible category $\!C$, we have $\!C \cong \Ind_\kappa(\!C_\kappa)$.

\begin{lemma}
\label{thm:lkpcat-product}
A $\kappa$-ary product $\prod_i \!C_i$ of locally $\kappa$-presentable categories $\!C_i$ is locally $\kappa$-presentable, with $(\prod_i \!C_i)_\kappa = \prod_i (\!C_i)_\kappa$.
\end{lemma}
\begin{proof}
See \cite[2.67]{ARlpac} (which is missing the hypothesis that the product must be $\kappa$-ary).
\end{proof}


\begin{lemma}
\label{thm:lkpcat-equifier}
Let $F, G : \!C -> \!D$ be cocontinuous functors between locally $\kappa$-presentable categories such that $F(\!C_\kappa), G(\!C_\kappa) \subseteq \!D_\kappa$, and let $\alpha, \beta : F -> G$ be natural transformations.  Then the \defn{equifier} of $\alpha, \beta$, i.e., the full subcategory $\Eq(\alpha, \beta) \subseteq \!C$ of those $X \in \!C$ for which $\alpha_X = \beta_X$, is locally $\kappa$-presentable, with $\Eq(\alpha, \beta)_\kappa = \Eq(\alpha, \beta) \cap \!C_\kappa$.
\end{lemma}
\begin{proof}
See \cite[2.76]{ARlpac}.  Alternatively, here is a direct proof.
Since $F, G$ are cocontinuous, $\Eq(\alpha, \beta) \subseteq \!C$ is closed under colimits.  Since $\Eq(\alpha, \beta) \subseteq \!C$ is full, $\Eq(\alpha, \beta) \cap \!C_\kappa \subseteq \Eq(\alpha, \beta)_\kappa$.  So it suffices to show that every $X \in \Eq(\alpha, \beta)$ is a $\kappa$-filtered colimit of objects in $\Eq(\alpha, \beta) \cap \!C_\kappa$; for this, it suffices to show that every morphism $f : Y -> X$ with $Y \in \!C_\kappa$ factors through some $g : Z -> X$ with $Z \in \Eq(\alpha, \beta) \cap \!C_\kappa$.  We have $G(f) \circ \alpha_Y = \alpha_X \circ F(f) = \beta_X \circ F(f) = G(f) \circ \beta_Y : F(Y) -> G(X) = \injlim_{\!C_\kappa \ni Z -> X} G(Z)$, so since $F(Y) \in \!D_\kappa$, $f$ factors as $Y =: Z_0 --->{g_0} Z_1 --->{h_1} X$ with $Z_1 \in \!C_\kappa$ such that $G(g_0) \circ \alpha_{Z_0} = G(g_0) \circ \beta_{Z_0}$.  Similarly, $h_1$ factors as $Z_1 --->{g_1} Z_2 --->{h_2} X$ with $Z_2 \in \!C_\kappa$ such that $G(g_1) \circ \alpha_{Z_1} = G(g_1) \circ \beta_{Z_1}$.  Continue finding $Z_0 --->{g_0} Z_1 --->{g_1} Z_2 --->{g_2} \dotsb --->{h_i} X$ in this way, then put $Z := \injlim_i Z_i$, to get $\alpha_Z = \beta_Z$.  Since $\kappa$ is uncountable, $Z \in \!C_\kappa$.
\end{proof}

\begin{lemma}
\label{thm:lkpcat-inserter}
Let $F, G : \!C -> \!D$ be cocontinuous functors between locally $\kappa$-presentable categories such that $F(\!C_\kappa), G(\!C_\kappa) \subseteq \!D_\kappa$.  Then the \defn{inserter} category $\Ins(F, G)$, whose objects are pairs $(X, \alpha)$ where $X \in \!C$ and $\alpha : F(X) -> G(X)$, is locally $\kappa$-presentable, with $\Ins(F, G)_\kappa$ consisting of those $(X, \alpha)$ with $X \in \!C_\kappa$.
\end{lemma}
\begin{proof}
This result is almost certainly well-known, although we could not find a precise reference for it.  Here is a proof sketch.

First, one shows that under the hypotheses of \cref{thm:lkpcat-equifier}, the \defn{inverter} $\Inv(\alpha) \subseteq \!C$, i.e., the full subcategory of $X$ such that $\alpha_X$ is invertible, is locally $\kappa$-presentable, with $\Inv(\alpha)_\kappa = \Inv(\alpha) \cap \!C_\kappa$.  This is done by an $\omega$-step construction as in the proof of \cref{thm:lkpcat-equifier} (see also \cite[B3.4.9]{Jeleph}).

Now consider $(1_\!C, F), (1_\!C, G) : \!C -> \!C \times \!D$.  By \cite[2.43]{ARlpac}, the comma category $\!(1_\!C, F) \down (1_\!C, G) = \{(X, Y,\, f : X -> Y,\, g : F(X) -> G(Y))\}$ is locally $\kappa$-presentable, with an object $(X, Y, f, g)$ locally presentable iff $X, Y \in \!C_\kappa$.  Clearly, $\Ins(F, G)$ is equivalent to the inverter of the natural transformation $\phi$ between the two projections $(1_\!C, F) \down (1_\!C, G) -> \!C$ given by $\phi_{(X, Y, f, g)} := f$.
\end{proof}

\section{$\kappa$-coherent theories}
\label{sec:kcohthy}

In this section, we briefly define the $\kappa$-ary analogs of the concepts from \cref{sec:imag,sec:isogpd,sec:den}.

Let $\@L$ be a first-order (relational) language.  Recall that $\@L_{\kappa\omega}$ is the extension of finitary first-order logic with $\kappa$-ary conjunctions $\bigwedge$ and disjunctions $\bigvee$.  A proof system for $\@L_{\kappa\omega}$ may be found in \cite[D1.3]{Jeleph}.  Note that this proof system is \emph{not} complete with respect to set-theoretic models.

The notions of \defn{$\kappa$-coherent formula}, \defn{$\kappa$-coherent axiom}, \defn{$\kappa$-coherent theory}, \defn{$\kappa$-coherent imaginary sort}, and \defn{$\kappa$-coherent definable function} are defined as in \cref{sec:imag}, with $\kappa$-ary disjunctions/disjoint unions replacing countable ones throughout.  The $\kappa$-coherent imaginary sorts and definable functions of a $\kappa$-coherent theory $(\@L, \@T)$ form the \defn{syntactic $\kappa$-pretopos}, denoted
\begin{align*}
\-{\ang{\@L \mid \@T}}_\kappa,
\end{align*}
which is the free \defn{$\kappa$-pretopos} (defined as in \cref{defn:spretop} but with $\kappa$-ary disjoint unions; functors preserving the $\kappa$-pretopos structure are called \defn{$\kappa$-coherent}) containing a \defn{model of $\@T$} (defined as in \cref{defn:struct-pretop}).  An \defn{interpretation} between two $\kappa$-coherent theories $(\@L, \@T), (\@L', \@T')$ is a $\kappa$-coherent functor $\-{\ang{\@L \mid \@T}}_\kappa -> \-{\ang{\@L' \mid \@T'}}_\kappa$; the theories are \defn{($\kappa$-coherently) Morita equivalent} if their syntactic $\kappa$-pretoposes are equivalent.

In the case $\kappa = \infty$, ``$\infty$-coherent'' is better known as \defn{geometric}, while the syntactic $\infty$-pretopos is better known as the \defn{classifying topos} (and usually denoted by $\!{Set}[\@T]$ instead of $\-{\ang{\@L \mid \@T}}_\infty$; see e.g., \cite[D3]{Jeleph}).  Note that when we speak of an $\infty$-coherent (i.e., geometric) theory $(\@L, \@T)$, we still mean that $\@L, \@T$ form sets (and not proper classes).

A $\kappa$-coherent theory $(\@L, \@T)$ may also be regarded as a geometric theory.  The link between the syntactic $\kappa$-pretopos and the classifying topos is provided by

\begin{lemma}
\label{thm:synptop-clstop}
For a $\kappa$-coherent theory $(\@L, \@T)$, we have $\-{\ang{\@L \mid \@T}}_\infty \cong \Ind_\kappa(\-{\ang{\@L \mid \@T}}_\kappa)$ (more precisely, the inclusion $\-{\ang{\@L \mid \@T}}_\kappa \subseteq \-{\ang{\@L \mid \@T}}_\infty$ exhibits the latter as the $\kappa$-ind-completion of the former).
\end{lemma}
\begin{proof}
Recall that $\Ind_\kappa(\-{\ang{\@L \mid \@T}}_\kappa)$ can be taken as the $\kappa$-ary limit-preserving functors $\-{\ang{\@L \mid \@T}}_\kappa^\op -> \!{Set}$.  On the other hand, by the theory of syntactic sites (see \cite[D3.1]{Jeleph}), $\-{\ang{\@L \mid \@T}}_\infty$ can be taken as the functors $\-{\ang{\@L \mid \@T}}_\kappa^\op -> \!{Set}$ preserving limits (i.e., colimits in $\-{\ang{\@L \mid \@T}}_\kappa$) of the form $(\bigsqcup_i A_i)/(\bigsqcup_{i,j} E_{ij})$ where $A_i \in \-{\ang{\@L \mid \@T}}_\kappa$ are $<\kappa$-many objects and $\bigsqcup_{i,j} E_{ij} \subseteq (\bigsqcup_i A_i)^2$ is an equivalence relation.  Clearly this includes $\kappa$-ary coproducts in $\-{\ang{\@L \mid \@T}}_\kappa$.  Since $\kappa$ is uncountable, we may compute general coequalizers in $\-{\ang{\@L \mid \@T}}_\kappa$ by imitating the usual procedure in $\!{Set}$ (to compute the coequalizer of $f, g : A -> B$, take the quotient of the equivalence relation generated $(f, g) : A -> B^2$; see \cite[A1.4.19]{Jeleph} for details).  This reduces coequalizers in $\-{\ang{\@L \mid \@T}}_\kappa$ to colimits of the form $(\bigsqcup_i A_i)/(\bigsqcup_{i,j} E_{ij})$.
\end{proof}


There is a more general notion of a \defn{multi-sorted $\kappa$-coherent theory} $(\@S, \@L, \@T)$, where $\@S$ is a set of sorts, $\@L$ consists of $\@S$-sorted relation symbols, and formulas have $\@S$-sorted variables and quantifiers; see \cite[D1.1]{Jeleph}.  So far we have been considering the case of a single-sorted theory, where $\@S = \{\#X\}$.  The definitions of syntactic $\omega_1$-pretopos, etc., have obvious multi-sorted generalizations.
Other than single-sorted theories, we will consider $0$-sorted or \defn{propositional $\kappa$-coherent theories} $(\@L, \@T)$, where $\@L$ is a set of proposition symbols (i.e., $0$-ary relation symbols) and $\@T$ is a set of implications between \defn{propositional $\kappa$-coherent $\@L$-formulas} (i.e., $\@L$-formulas built with finite $\wedge$ and $\kappa$-ary $\bigvee$).  The \defn{Lindenbaum--Tarski algebra} $\ang{\@L \mid \@T}_\kappa$ of a propositional theory $(\@L, \@T)$ is the $\kappa$-frame presented by $(\@L, \@T)$; it is also the \defn{syntactic category} of $(\@L, \@T)$, as defined in \cref{sec:imag}.  Recall from there that the syntactic $\kappa$-pretopos is a certain completion of the syntactic category; for propositional theories, this takes the form 
\begin{align*}
\-{\ang{\@L \mid \@T}}_\kappa \cong \Sh_\kappa(\ang{\@L \mid \@T}_\kappa).
\end{align*}

Let $(\@L, \@T)$ be a (single-sorted) geometric theory.  We define the \defn{locale of countable $\@L$-structures} $\Mod(\@L)$ by formally imitating the definition in \cref{sec:isogpd}.  That is, we take the frame $\@O(\Mod(\@L))$ to be freely generated by the symbols
\begin{align*}
\den{\abs{\#X} \ge n} \quad\text{for $n \in \#N$}, &&
\den{R(\vec{a})} \quad\text{for $n$-ary $R \in \@L$ and $\vec{a} \in \#N^n$},
\end{align*}
subject to the relations (compare with the proof in \cref{sec:isogpd} that $\Mod(\@L)$ is quasi-Polish)
\begin{align*}
\top \le \den{\abs{\#X} \ge 0}, &&
\den{\abs{\#X} \ge n+1} \le \den{\abs{\#X} \ge n}, &&
\den{R(\vec{a})} \le \den{\abs{\#X} \ge \max_i (a_i+1)}.
\end{align*}
Clearly, a point of $\Mod(\@L)$ (i.e., a frame homomorphism $\@O(\Mod(\@L)) -> 2$) is the same thing as an $\@L$-structure on an initial segment of $\#N$; thus the spatialization $\Sp(\Mod(\@L))$ is just the space of countable $\@L$-structures, as defined in \cref{sec:isogpd}.  For each geometric $\@L$-formula $\phi$ with $n$ variables and $\vec{a} \in \#N^n$, we define $\den{\phi(\vec{a})} \in \@O(\Mod(\@L))$ by induction on $\phi$ in the obvious manner:
\begin{align*}
\den{\top} &:= \den{\abs{\#X} \ge \max_i a_i}, \\
\den{a_i = a_j} &:= \text{$\den{\abs{\#X} \ge \max_i a_i}$ if $a_i = a_j$, else $\bot$}, \\
\den{\phi(\vec{a}) \wedge \psi(\vec{a})} &:= \den{\phi(\vec{a})} \wedge \den{\psi(\vec{a})}, \\
\den{\bigvee_i \phi_i(\vec{a})} &:= \bigvee_i \den{\phi_i(\vec{a})}, \\
\den{\exists x\, \phi(\vec{a}, x)} &:= \bigvee_{b \in \#N} \den{\phi(\vec{a}, b)}.
\end{align*}
We define the \defn{locale of countable models of $\@T$} to be the sublocale $\Mod(\@L, \@T) \subseteq \Mod(\@L)$ determined by the relations (i.e., $\@O(\Mod(\@L, \@T))$ is the quotient of $\@O(\Mod(\@L))$ by these relations)
\begin{align*}
\den{\phi(\vec{a})} \le \den{\psi(\vec{a})} \quad\text{for an axiom $\forall \vec{x}\, (\phi(\vec{x}) => \psi(\vec{x}))$ in $\@T$, where $n := \abs{\vec{x}}$, and $\vec{a} \in \#N^n$}.
\end{align*}

We similarly define the locale $\Iso(\@L)$ ``of pairs $(g, \@M)$ where $\@M \in \Mod(\@L)$ and $g \in S_M$'' by imitating \cref{sec:isogpd}: $\@O(\Iso(\@L))$ is the frame generated by the symbols
\begin{align*}
\partial_1^*(U) \quad\text{for (a generator) $U \in \@O(\Mod(\@L))$}, &&
\den{a |-> b} \quad\text{for $a, b \in \#N$},
\end{align*}
subject to the relations in $\@O(\Mod(\@L))$ between generators of the first kind, and the relations
\begin{gather*}
\den{a |-> b} \le \partial_1^*(\den{\abs{\#X} \ge a+1}) \wedge \partial_1^*(\den{\abs{\#X} \ge b+1}), \\
(\den{a |-> b} \wedge \den{a |-> c}) \vee (\den{b |-> a} \wedge \den{c |-> a}) \le \bot \quad\text{for $b \ne c$}, \\
\partial_1^*(\den{\abs{\#X} \ge a+1}) \le (\bigvee_b \den{a |-> b}) \wedge (\bigvee_b \den{b |-> a})
\end{gather*}
which say that ``$g \in S_M$''.  Clearly, $\Sp(\Iso(\@L))$ is the space of isomorphisms as defined in \cref{sec:isogpd}.  We clearly have a locale morphism $\partial_1 : \Iso(\@L) -> \Mod(\@L)$; we also have $\partial_0 : \Iso(\@L) -> \Mod(\@L)$, $\iota : \Mod(\@L) -> \Iso(\@L)$, $\mu : \Iso(\@L) \times_{\Mod(\@L)} \Iso(\@L) -> \Iso(\@L)$, and $\nu : \Iso(\@L) -> \Iso(\@L)$, given by
\begin{align*}
\partial_0^*(\den{\abs{\#X} \ge n}) &:= \den{\abs{\#X} \ge n}, &
\iota^*(\partial_1^*(U)) &:= U, \\
\partial_0^*(\den{R(\vec{a})}) &:= \bigvee_{\vec{b}} (\bigwedge_i \den{b_i |-> a_i} \wedge \partial_1^*(\den{R(\vec{b})}), &
\iota^*(\den{a |-> b}) &:= \text{$\den{\abs{\#X} \ge a+1}$ if $a = b$, else $\bot$}, \\[1.2ex]
\nu^*(\partial_1^*(U)) &:= \partial_0^*(U), &
\mu^*(\partial_1^*(U)) &:= \pi_1^*(\partial_1^*(U)), \\
\nu^*(\den{a |-> b}) &:= \den{b |-> a}, &
\mu^*(\den{a |-> b}) &:= \bigvee_c (\pi_1^*(\den{a |-> c}) \wedge \pi_0^*(\den{c |-> b})),
\end{align*}
where $\pi_0, \pi_1 : \Iso(\@L) \times_{\Mod(\@L)} \Iso(\@L) -> \Iso(\@L)$ are the two projections.
It is straightforward to check that these are well-defined and form the structure maps of a localic groupoid $\MOD(\@L)$, the \defn{localic groupoid of countable $\@L$-structures}.  For a geometric $\@L$-theory $\@T$, the \defn{localic groupoid of countable models of $\@T$}, $\MOD(\@L, \@T)$, is the full subgroupoid on $\Mod(\@L, \@T) \subseteq \Mod(\@L)$.

Note that when $(\@L, \@T)$ is a $\kappa$-coherent theory, $\Mod(\@L, \@T)$, $\Iso(\@L, \@T)$, and the structure maps are $\kappa$-coherent.  If furthermore $(\@L, \@T)$ is $\kappa$-ary (i.e., $\abs{\@L}, \abs{\@T} < \kappa$), then $\Mod(\@L, \@T)$ and $\Iso(\@L, \@T)$ are $\kappa$-presented.  When $\kappa = \omega_1$, for a countable $\omega_1$-coherent theory $(\@L, \@T)$, it is easily verified that the definition of $\MOD(\@L, \@T)$ given here corresponds under \cref{thm:qpol-sloc} to that given in \cref{sec:isogpd}.

For $\kappa$-coherent $(\@L, \@T)$, we let $\Act_\kappa(\MOD(\@L, \@T))$ denote the \defn{category of $\kappa$-étalé actions of $\MOD(\@L, \@T)$}, i.e., $\kappa$-étalé locales $X -> \Mod(\@L, \@T)$ over $\Mod(\@L, \@T)$ equipped with an action $\Iso(\@L, \@T) \times_{\Mod(\@L, \@T)} X -> X$.  Note that $\Act_\kappa(\MOD(\@L, \@T)) \subseteq \Act_\infty(\MOD(\@L, \@T))$.

Finally, for a geometric theory $(\@L, \@T)$, we define the \defn{interpretation} $\den{A}$ of an imaginary sort $A \in \-{\ang{\@L \mid \@T}}_\infty$ by imitating \cref{sec:den}.  That is, for a geometric formula $\alpha$ with $n$ variables, we define $\den{\alpha}$ to be the disjoint union of open sublocales $\den{\alpha}_{\vec{a}} \subseteq \den{\alpha}$ for $\vec{a} \in \#N^n$, where each $\den{\alpha}_{\vec{a}}$ is an isomorphic copy of $\den{\alpha(\vec{a})} \subseteq \Mod(\@L, \@T)$, equipped with the étalé morphism $\pi : \den{\alpha} = \bigsqcup_{\vec{a}} \den{\alpha}_{\vec{a}} -> \Mod(\@L, \@T)$ induced by the inclusions $\den{\alpha}_{\vec{a}} \cong \den{\alpha(\vec{a})} \subseteq \Mod(\@L, \@T)$.
Thus, $\@O(\den{\alpha})$ is generated (as a frame) by
\begin{align*}
\den{\alpha}_{\vec{a}} \quad\text{for $\vec{a} \in \#N^n$}, &&
\pi^*(U) \quad\text{for $U \in \@O(\Mod(\@L, \@T))$ (a generator)}.
\end{align*}
We let $\MOD(\@L, \@T)$ act on $\den{\alpha}$ via $\rho_\alpha : \Iso(\@L, \@T) \times_{\Mod(\@L, \@T)} \den{\alpha} -> \den{\alpha}$, given by
\begin{align*}
\rho_\alpha^*(\den{\alpha}_{\vec{a}}) &:= \bigvee_{\vec{b}} (\pi_0^*(\bigwedge_i \den{b_i |-> a_i}) \wedge \pi_1^*(\den{\alpha}_{\vec{b}}))
\end{align*}
where $\pi_0 : \Iso(\@L, \@T) \times_{\Mod(\@L, \@T)} \den{\alpha} -> \Iso(\@L, \@T)$ and $\pi_1 : \Iso(\@L, \@T) \times_{\Mod(\@L, \@T)} \den{\alpha} -> \den{\alpha}$ are the projections.  We then extend this definition to $\den{A}$ for an arbitrary imaginary sort $A \in \-{\ang{\@L \mid \@T}}_\infty$, as well as $\den{f} : \den{A} -> \den{B}$ for a definable function $f : A -> B$, exactly as in \cref{sec:den}.  This defines a geometric functor
\begin{align*}
\den{-} : \-{\ang{\@L \mid \@T}}_\infty --> \Act_\infty(\MOD(\@L, \@T)).
\end{align*}

When $(\@L, \@T)$ is $\kappa$-coherent, $\den{A}$ is $\kappa$-étalé (over $\MOD(\@L, \@T)$) for $\kappa$-coherent $A \in \-{\ang{\@L \mid \@T}}_\kappa \subseteq \-{\ang{\@L \mid \@T}}_\infty$; thus $\den{-}$ restricts to a $\kappa$-coherent functor
\begin{align*}
\den{-} : \-{\ang{\@L \mid \@T}}_\kappa --> \Act_\kappa(\MOD(\@L, \@T)).
\end{align*}
For $\kappa = \omega_1$ and $(\@L, \@T)$ countable, we recover via \cref{thm:qpol-sloc} the definition in \cref{sec:den}.

\section{The Joyal--Tierney theorem for decidable theories}
\label{sec:joyal-tierney}

In this section, we sketch a proof of the following generalization of \cref{thm:etale-coherent-interp-equiv} using the Joyal--Tierney representation theorem.

As in \cref{sec:imag}, we call a $\kappa$-coherent theory $(\@L, \@T)$ \defn{decidable} if there is a $\kappa$-coherent $\@L$-formula with two variables (denoted $x \ne y$) which $\@T$ proves is the negation of equality.

\begin{theorem}
Let $\@L$ be a relational language and $\@T$ be a decidable $\kappa$-coherent $\@L$-theory.  Then
\begin{align*}
\den{-} : \-{\ang{\@L \mid \@T}}_\kappa --> \Act_\kappa(\MOD(\@L, \@T))
\end{align*}
is an equivalence of categories.
\end{theorem}
\begin{proof}
We begin with the case $\kappa = \infty$, which is a straightforward variant of the usual proof of the Joyal--Tierney theorem; see \cite{JT} or \cite[C5.2]{Jeleph}.

Let $(\@L_d, \@T_d)$ be the (single-sorted) \defn{theory of decidable sets}, where $\@L_d$ consists of a single binary relation symbol $\ne$ and $\@T_d$ consists of the ($\omega$-coherent) axioms
\begin{align*}
\forall x\, (x \ne x => \bot), &&
\forall x, y\, (\top => (x = y) \vee (x \ne y)).
\end{align*}
That $(\@L, \@T)$ is decidable means that we have an interpretation
\begin{align*}
F_d : (\@L_d, \@T_d) -> (\@L, \@T)
\end{align*}
given by $F_d(\#X) := \#X$ and $F_d({\ne}) := ({\ne})$ (where $({\ne}) \subseteq \#X^2 \in \-{\ang{\@L \mid \@T}}_\infty$ is the geometric formula with $2$ variables witnessing decidability).  The classifying topos $\-{\ang{\@L_d \mid \@T_d}}_\infty$ is computed in \cite[D3.2.7]{Jeleph}: it is the presheaf topos $\!{Set}^{\!{Set}_{fm}}$ where $\!{Set}_{fm}$ is the category of finite sets and injections, with the home sort $\#X \in \!{Set}^{\!{Set}_{fm}}$ given by the inclusion.

Let $(\@L_i, \@T_i)$ be the \defn{propositional theory of initial segments of $\#N$}, where $\@L_i$ consists of the proposition symbols $(\abs{\#X} \ge n)$ for each $n \in \#N$ (here $\#X$ is merely part of the notation, and does \emph{not} denote a home sort, as the theory is propositional), and $\@T_i$ consists of the axioms
\begin{align*}
\top => (\abs{\#X} \ge 0), &&
(\abs{\#X} \ge n+1) => (\abs{\#X} \ge n).
\end{align*}
The Lindenbaum--Tarski algebra $\ang{\@L_i \mid \@T_i}_\infty$ is the frame
\begin{align*}
\ang{\@L_i \mid \@T_i}_\infty = \{\bot < \dotsb < [\abs{\#X} \ge 2] < [\abs{\#X} \ge 1] < [\abs{\#X} \ge 0] = \top\}.
\end{align*}
Thus the classifying topos $\-{\ang{\@L_i \mid \@T_i}}_\infty = \Sh(\ang{\@L_i \mid \@T_i}_\infty)$ is the presheaf topos $\!{Set}^\#N$, where $[\abs{\#X} \ge n] \in \-{\ang{\@L_i \mid \@T_i}}_\infty$ is identified with the functor $\#N -> \!{Set}$ which is $1$ on $m \ge n$ and $0$ on $m < n$.

We have a geometric functor $\!{Set}^{\!{Set}_{fm}} -> \!{Set}^\#N$ induced by the inclusion $\#N -> \!{Set}_{fm}$ (mapping $m \le n$ to the inclusion $m \subseteq n$).  Using \cite[A4.2.7(b), C3.1.2]{Jeleph}, this geometric functor is easily seen to be the inverse image part of a surjective open geometric morphism; surjectivity follows from essential surjectivity of the inclusion $\#N -> \!{Set}_{fm}$, while for openness, given $U \in \#N$, $V \in \!{Set}_{fm}$, and $b : U `-> V$ as in \cite[C3.1.2]{Jeleph}, let $U' := \abs{V}$, $r : U' \cong V$ be any bijection such that $r|U = b$, and $i := r^{-1}$, so that $r \circ i = 1_V$ and $i \circ b$ is the inclusion $U \subseteq U'$, as required by \cite[C3.1.2]{Jeleph}.  This functor is the interpretation
\begin{align*}
F_i : (\@L_d, \@T_d) &--> (\@L_i, \@T_i) \\
\#X &|--> \bigsqcup_{n \in \#N} (\abs{\#X} \ge n+1).
\end{align*}
In terms of models (in arbitrary Grothendieck toposes), this interpretation takes a model of $\@T_i$, i.e., an initial segment of $\#N$, to the model of $\@T_d$, i.e., decidable set, given by that initial segment.

We now compute the ``(2-)pushout of theories'' $(\@L', \@T')$ of the two interpretations $F_d, F_i$ (in the $2$-category of $\infty$-pretoposes):
\begin{equation*}
\begin{tikzcd}
(\@L', \@T') & (\@L_i, \@T_i) \lar \\
(\@L, \@T) \uar["E"] & (\@L_d, \@T_d) \lar["F_d"] \uar["F_i"']
\end{tikzcd}
\end{equation*}
By definition, this is a geometric theory $(\@L', \@T')$ such that a model of $\@T'$ (in an arbitrary Grothendieck topos) is the same thing as a model of $\@T$, a model of $\@T_i$, and an isomorphism between the two models of $\@T_d$ given by the interpretations $F_d, F_i$; this is equivalently an initial segment of $\#N$ together with a model of $\@T$ on that initial segment.  Thus, we may take $(\@L', \@T')$ to be the propositional theory presenting the frame of opens $\@O(\Mod(\@L, \@T))$ given in \cref{sec:kcohthy}, i.e., $\@L'$ consists of the proposition symbols $(\abs{\#X} \ge n)$ and $R(\vec{a})$ for $n$-ary $R \in \@L$ and $\vec{a} \in \#N$, and $\@T'$ consists of the axioms in $\@T_i$ together with the axioms
\begin{align*}
R(\vec{a}) => (\abs{\#X} \ge \max_i (a_i+1)), &&
\phi(\vec{a}) => \psi(\vec{a}),
\end{align*}
where $\forall \vec{x}\, (\phi(\vec{x}) => \psi(\vec{x}))$ is an axiom in $\@T$ and $\phi(\vec{a}), \psi(\vec{a})$ are the propositional $\@L'$-formulas defined by induction on $\phi, \psi$ in the obvious way.  The classifying topos is
\begin{align*}
\-{\ang{\@L' \mid \@T'}}_\infty = \Sh(\ang{\@L' \mid \@T'}_\infty) = \Sh(\@O(\Mod(\@L, \@T))) = \Sh(\Mod(\@L, \@T)),
\end{align*}
and the interpretation $E : \-{\ang{\@L \mid \@T}}_\infty -> \Sh(\Mod(\@L, \@T))$ in the diagram above is simply the functor $\den{-}$ (identifying sheaves on $\Mod(\@L, \@T)$ with étalé locales over $\Mod(\@L, \@T)$, and forgetting the $\MOD(\@L, \@T)$-action for now).

Similarly, the pushout $(\@L'', \@T'')$ of $E$ with itself
\begin{equation*}
\begin{tikzcd}
(\@L'', \@T'') & (\@L', \@T') \lar["D_1"'] \\
(\@L', \@T') \uar["D_0"] & (\@L, \@T) \lar["E"] \uar["E"']
\end{tikzcd}
\end{equation*}
is the theory of pairs of models of $\@T'$ together with an isomorphism between them, or equivalently of pairs $(\@M, g)$ of a model $\@M$ of $\@T'$ together with a permutation $g$ of its underlying initial segment of $\#N$; so we may take $(\@L'', \@T'')$ to be the propositional theory presenting $\Iso(\@L, \@T)$ from \cref{sec:kcohthy}, and $D_0, D_1$ to be induced by the frame homomorphisms $\partial_0^*, \partial_1^* : \@O(\Mod(\@L, \@T)) -> \@O(\Iso(\@L, \@T))$.  We have an interpretation
\begin{align*}
I : (\@L'', \@T'') -> (\@L', \@T')
\end{align*}
induced (via the universal property of the pushout) by the identity $1_{\@T'} : (\@L', \@T') -> (\@L', \@T')$ (so $I \circ D_0 \cong 1 \cong I \circ D_1$), which takes a model $\@M$ of $\@T'$ to the model $(\@M, 1_M)$ of $\@T''$; so $I$ is induced by the frame homomorphism $\iota^* : \@O(\Iso(\@L, \@T)) -> \@O(\Mod(\@L, \@T))$.

We also have the triple pushout $(\@L''', \@T''')$ of three copies of $E$ (equivalently of $D_0, D_1$), the theory of triples of models $\@M_1, \@M_2, \@M_3$ with isomorphisms $\@M_1 \cong \@M_2 \cong \@M_3$, which presents the frame $\@O(\Iso(\@L, \@T) \times_{\Mod(\@L, \@T)} \Iso(\@L, \@T))$.  The three injections
\begin{align*}
P_0, M, P_1 : (\@L'', \@T'') -> (\@L''', \@T''')
\end{align*}
take a model $(\@M_1 \cong \@M_2 \cong \@M_3)$ of $\@T'''$ to (respectively) the first isomorphism $\@M_1 \cong \@M_2$, the composite isomorphism $\@M_1 \cong \@M_3$, and the second isomophism $\@M_2 \cong \@M_3$, so are induced by $\pi_0^*, \mu^*, \pi_1^* : \@O(\Iso(\@L, \@T)) -> \@O(\Iso(\@L, \@T) \times_{\Mod(\@L, \@T)} \Iso(\@L, \@T))$.  All these data form a diagram
\begin{equation*}
\begin{tikzcd}[column sep=4em]
(\@L''', \@T''') &
(\@L'', \@T'') \lar[shift right=3,"P_0"'] \lar["M"{fill=white,anchor=center}] \lar[shift left=3,"P_1"] \rar["I"{fill=white,anchor=center}] &
(\@L', \@T') \lar[shift right=2,"D_0"'] \lar[shift left=2,"D_1"] &
(\@L, \@T) \lar["E"']
\end{tikzcd}
\end{equation*}
which is an augmented \defn{$2$-truncated simplicial topos}, in the sense of \cite[B3.4, C5.1]{Jeleph}.

\defn{Descent data} (see \cite[VIII~\S1]{JT} or \cite[B3.4, C5.1]{Jeleph}) on an object $A \in \-{\ang{\@L' \mid \@T'}}_\infty$ consists of a morphism $\theta : D_1(A) -> D_0(A)$ obeying the \defn{unit condition} that
\begin{align*}
A \cong I(D_1(A)) --->{I(\theta)} I(D_0(A)) \cong A
\end{align*}
is the identity $1_A$, and the \defn{cocycle condition} that the following two morphisms are equal:
\begin{gather*}
P_1(D_1(A)) \cong M(D_1(A)) --->{M(\theta)} M(D_0(A)) \cong P_0(D_0(A)), \\
P_1(D_1(A)) --->{P_1(\theta)} P_1(D_0(A)) \cong P_0(D_1(A)) --->{P_0(\theta)} P_0(D_0(A)).
\end{gather*}
Let $\Desc(\@L^{(-)}, \@T^{(-)})$ denote the category of objects in $\-{\ang{\@L' \mid \@T'}}_\infty$ equipped with descent data (and morphisms which commute with the descent data in the obvious sense).  For every $A \in \-{\ang{\@L \mid \@T}}_\infty$, we have an isomorphism $D_1(E(A)) \cong D_0(E(A))$ by definition of $(\@L'', \@T'')$, which is easily verified to be descent data on $E(A)$; this defines a lift of $E : \-{\ang{\@L \mid \@T}}_\infty -> \-{\ang{\@L' \mid \@T'}}_\infty$ to a geometric functor
\begin{align*}
E' : \-{\ang{\@L \mid \@T}}_\infty --> \Desc(\@L^{(-)}, \@T^{(-)}).
\end{align*}
Since $E$ is the pushout of $F_i$ which is the inverse image part of a surjective open geometric morphism, so is $E$ (see \cite[VII~1.3]{JT} or \cite[C3.1.26]{Jeleph}).  Thus by the \defn{Joyal--Tierney descent theorem} (see \cite[VIII~2.1]{JT} or \cite[C5.1.6]{Jeleph}), $E'$ is an equivalence of categories.

Under the above identifications $\-{\ang{\@L' \mid \@T'}}_\infty \cong \Sh(\Mod(\@L, \@T))$ and $\-{\ang{\@L'' \mid \@T''}}_\infty \cong \Sh(\Iso(\@L, \@T))$ as well as the identification between sheaves and étalé locales, $\Desc(\@L^{(-)}, \@T^{(-)})$ is equivalently the category of étalé locales $p : X -> \Mod(\@L, \@T)$ equipped with a morphism $t : \Iso(\@L, \@T) \times_{\Mod(\@L, \@T)} X -> X \times_{\Mod(\@L, \@T)} \Iso(\@L, \@T)$ commuting with the projections to $\Iso(\@L, \@T)$ and satisfying the obvious analogs of the unit and cocycle conditions.  By an easy calculation, via composition with the projection $X \times_{\Mod(\@L, \@T)} \Iso(\@L, \@T) -> X$, such $t$ are in bijection with $\MOD(\@L, \@T)$-actions $\Iso(\@L, \@T) \times_{\Mod(\@L, \@T)} X -> X$.  Also morphisms preserving the descent data correspond to equivariant morphisms; so we have
\begin{align*}
\Desc(\@L^{(-)}, \@T^{(-)}) \cong \Act_\infty(\MOD(\@L, \@T)).
\end{align*}
Finally, it is straightforward to verify that under this equivalence, the functor $E'$ above is just $\den{-}$, which completes the proof of the theorem in the case $\kappa = \infty$.

Now consider general $\kappa$.  Let $\@T$ be a decidable $\kappa$-coherent $\@L$-theory.  We know that
\begin{align*}
\den{-} : \-{\ang{\@L \mid \@T}}_\infty --> \Act_\infty(\MOD(\@L, \@T))
\end{align*}
is an equivalence, and restricts to
\begin{align*}
\den{-} : \-{\ang{\@L \mid \@T}}_\kappa --> \Act_\kappa(\MOD(\@L, \@T)).
\end{align*}
By \cref{thm:synptop-clstop}, $\-{\ang{\@L \mid \@T}}_\kappa$ consists of the $\kappa$-presentable objects in $\-{\ang{\@L \mid \@T}}_\infty$; so it suffices to verify that every $X \in \Act_\kappa(\MOD(\@L, \@T))$ is $\kappa$-presentable in $\Act_\infty(\MOD(\@L, \@T))$.  For $X \in \Act_\kappa(\MOD(\@L, \@T))$, we have $X \in \Sh_\kappa(\Mod(\@L, \@T)) \cong \-{\ang{\@L' \mid \@T'}}_\kappa$, whence by \cref{thm:synptop-clstop} again, $X$ is $\kappa$-presentable in $\-{\ang{\@L' \mid \@T'}}_\infty \cong \Sh(\Mod(\@L, \@T))$.  This implies that $X$ is $\kappa$-presentable in $\Act_\infty(\MOD(\@L, \@T))$ by \cref{thm:lkpcat-inserter,thm:lkpcat-equifier}, since the definition of $\Desc(\@L^{(-)}, \@T^{(-)})$ above can be rephrased as first taking the inserter $\Ins(D_1, D_0)$ and then taking two equifiers to enforce the unit and cocycle conditions, where the functors involved, namely $D_i, P_i, I, M$, clearly preserve colimits and $\kappa$-presentable objects.
%
%
\end{proof}

\bigskip\noindent
Department of Mathematics \\
University of Illinois at Urbana--Champaign \\
Urbana, IL 61801 \\
\medskip
\nolinkurl{ruiyuan@illinois.edu}


\begin{thebibliography}{000000}

\bibitem[ALR]{ALR}  J.~Adámek, F.~W.~Lawvere, and J.~Rosický, \emph{How algebraic is algebra?}, Theory Appl.\ Categ.\ \textbf{8(9)}~(2001), 253--283.

\bibitem[AR]{ARlpac}  J.~Adámek and J.~Rosický, \emph{Locally presentable and accessible categories}, London Math Society Lecture Note Series, vol.~189, Cambridge University Press, 1997.

\bibitem[AF]{AF}  S.~Awodey and H.~Forssell, \emph{First-order logical duality}, Ann.\ Pure Appl.\ Logic \textbf{164(3)}~(2013), 319--348.

\bibitem[BK]{BKpol}  H.~Becker and A.~S.~Kechris, \emph{The descriptive set theory of Polish group actions}, London Math Society Lecture Note Series, vol.~232, Cambridge University Press, 1996.

\bibitem[Bor]{Bor}  F.~Borceux, \emph{Handbook of categorical algebra}, Encyclopedia of Mathematics and its Applications, vol.~50--52, Cambridge University Press, 1994.

\bibitem[CLW]{CLW}  A.~Carboni, S.~Lack, and R.~F.~C.~Walters, \emph{Introduction to extensive and distributive categories}, J.\ Pure Appl.\ Algebra \textbf{84}~(1993), 145--158.

\bibitem[deB]{deB}  M.~de~Brecht, \emph{Quasi-Polish spaces}, Ann.\ Pure Appl.\ Logic \textbf{164(3)}~(2013), 356--381.

\bibitem[Gao]{Gao}  S.~Gao, \emph{Invariant descriptive set theory}, Pure and applied mathematics, vol.~293, CRC Press, 2009.

\bibitem[HMM]{HMM}  M.~Harrison-Trainor, R.~Miller, and A.~Montalbán, \emph{Borel functors and infinitary interpretations}, J.\ Symb.\ Logic \textbf{83(4)}~(2018), 1434--1456.

\bibitem[Hec]{Hec}  R.~Heckmann, \emph{Spatiality of countably presentable locales (proved with the Baire category theorem)}, Math.\ Structures Comput.\ Sci.\ \textbf{25(7)}~(2015), 1607--1625.

\bibitem[Hod]{Hod}  W.~Hodges, \emph{Model theory}, Encyclopedia of Mathematics and its Applications, vol.~42, Cambridge University Press, 1993.

\bibitem[J82]{Jstone}  P.~T.~Johnstone, \emph{Stone spaces}, Cambridge Studies in Advanced Mathematics, vol.~3, Cambridge University Press, 1982.

\bibitem[J02]{Jeleph}  P.~T.~Johnstone, \emph{Sketches of an elephant: a topos theory compendium}, Oxford Logic Guides, vol.~43--44, Oxford University Press, 2002.

\bibitem[JT]{JT}  A.~Joyal and M.~Tierney, \emph{An extension of the Galois theory of Grothendieck}, Mem.\ Amer.\ Math.\ Soc.\ \textbf{51(309)}~(1984).

\bibitem[Kec]{Kcdst}  A.~S.~Kechris, \emph{Classical descriptive set theory}, Graduate Texts in Mathematics, vol.~156, Springer-Verlag, 1995.

\bibitem[LMV]{LMV}  I.~J.~Le~Creurer, F.~Marmolejo, and E.~M.~Vitale, \emph{Beck's theorem for pseudo-monads}, J.\ Pure Appl.\ Algebra \textbf{173}~(2002), 293--313.

\bibitem[Lop]{Lop}  E.~G.~K.~Lopez-Escobar, \emph{An interpolation theorem for denumerably long formulas}, Fund.\ Math.\ \textbf{57}~(1965), 253--272.

\bibitem[Lup]{Lup}  M.~Lupini, \emph{Polish groupoids and functorial complexity}, Trans.\ Amer.\ Math.\ Soc.\ \textbf{369(9)}~(2017), 6683--6723.

\bibitem[M87]{Multra}  M.~Makkai, \emph{Stone duality for first order logic}, Adv.\ Math.\ \textbf{65(2)}~(1987), 97--170.

\bibitem[M88]{Mscc}  M.~Makkai, \emph{Strong conceptual completeness for first-order logic}, Ann.\ Pure Appl.\ Logic \textbf{40}~(1988), 167--215.

\bibitem[M90]{Mexact}  M.~Makkai, \emph{A theorem on Barr-exact categories, with an infinitary generalization}, Ann.\ Pure Appl.\ Logic \textbf{47}~(1990), 225-268.

\bibitem[MR]{MR}  M.~Makkai and G.~E.~Reyes, \emph{First order categorical logic}, Lecture Notes in Mathematics, vol.~611, Springer-Verlag, 1977.

\bibitem[MT]{MT}  J.~Melleray and T.~Tsankov, \emph{Generic representations of abelian groups and extreme amenability}, Israel J.\ Math.\ \textbf{198(1)}~(2013), 129--167.

\bibitem[PT]{PT}  H.~E.~Porst and W.~Tholen, \emph{Concrete dualities}, in: H.~Herrlich and H.~E.~Porst, \emph{Category theory at work}, Research and Exposition in Mathematics, vol.~18, Heldermann Verlag, 1991.

\bibitem[Ten]{Ten}  B.~R.~Tennison, \emph{Sheaf theory}, London Math Society Lecture Note Series, vol.~20, Cambridge University Press, 1975.

\bibitem[Ulm]{Ulmer}  F.~Ulmer, \emph{Properties of dense and relative adjoint functors}, J.\ Algebra \textbf{8}~(1968), 77--95.

\bibitem[Vau]{Vau}  R.~Vaught, \emph{Invariant sets in topology and logic}, Fund.\ Math.\ \textbf{82(3)}~(1974), 269--294.

\end{thebibliography}
\end{document}